\documentclass[11pt]{article}
\usepackage{mathtools}
\usepackage{paralist}
\usepackage[T1]{fontenc}
\usepackage{amsfonts}
\usepackage{amsmath}
\usepackage{amssymb}
\usepackage{amsthm}
\usepackage{float}
\usepackage{bbm}
\usepackage{bm}
\usepackage{mathrsfs}
\usepackage{color}
\usepackage{pdfsync}
\usepackage{enumitem}
\usepackage[colorlinks=true, linkcolor=blue, citecolor=blue, urlcolor=blue]{hyperref}

\usepackage{graphicx}
\usepackage{subfig}
\usepackage{tikz}
\usepackage[round]{natbib}

\numberwithin{equation}{section}

\newcommand{\bpi}{\boldsymbol{\pi}}
\newcommand{\X}{{X}}

\newcommand{\bx}{{x}}
\newcommand{\E}{\mathbb{E}}
\newcommand{\R}{\mathbb{R}}
\newcommand{\Z}{\mathbb{Z}}
\newcommand{\N}{\mathbb{N}}
\DeclareMathOperator*{\argmin}{argmin}

\newcommand{\EE}{\mathbb{E}}
\newcommand{\PP}{\mathbb{P}}
\newcommand{\NN}{\mathbb{N}}

\newcommand{\eps}{ \varepsilon}

\newcommand{\bX}{{X}}
\newcommand{\by}{{y}}
\usepackage[margin=31 mm]{geometry}
\newcommand{\mykill}[1]{}

\newcommand\independent{\protect\mathpalette{\protect\independenT}{\perp}}
\def\independenT#1#2{\mathrel{\rlap{$#1#2$}\mkern2mu{#1#2}}}

\usepackage[capitalize]{cleveref}

\usepackage{float}

\theoremstyle{plain}
\newtheorem{theorem}{Theorem}[section]

\newtheorem{lemma}{Lemma}[section]
\newtheorem{corollary}[theorem]{Corollary}
\theoremstyle{definition}
\newtheorem{definition}{Definition}[section]
\newtheorem{remark}{Remark}[section]

\newtheorem{assumption}{Assumption}[section]
\theoremstyle{remark}

\newlist{myenum}{enumerate}{3}
\setlist[myenum,1]{label={\rm (H\arabic*)},
                   ref  ={\rm (H\arabic*)}}
\crefname{myenumi}{property}{properties}

\renewenvironment{thebibliography}[1]{%
\begin{oldthebibliography}{#1}%
\setlength{\baselineskip}{.9em}
\linespread{1}
\small
\setlength{\parskip}{.40ex}%
\setlength{\itemsep}{.30em}%
}%
{%
\end{oldthebibliography}%
}

\title{ $\,$\vspace{-12mm}
\\ Nonparametric Vector Quantile Autoregression}
\date{\nodate}
\author{Alberto Gonz{\'a}lez-Sanz%
  \thanks{Department of Statistics, Columbia University, New York, USA <ag4855@columbia.edu>.} \and  Marc Hallin \thanks{D\' epartement de Mat\' ematique, Universit\' e libre de Bruxelles, Brussels, Belgium <mhallin@ulb.be> and Institute of Information Theory and Automation,  Czech Academy of Sciences, Prague, Czech Republic. } \and  Yisha Yao
  \thanks{Departments of Statistics, Columbia University, New York, USA <yy3381@columbia.edu>.} }
  
  \date{}
  
\begin{document}

\maketitle

$\,$\vspace{-4mm}

\begin{abstract}
Prediction is a key issue in time series analysis. Just as classical mean regression models, classical autoregressive  methods, yielding  L$^2$ point-predictions, provide rather poor predictive summaries; a much more informative approach is based on quantile  (auto)regression, where the whole distribution of future observations conditional on the past is consistently recovered. Since their introduction by Koenker and Xiao in 2006, autoregressive quantile autoregression methods have become a popular and successful alternative  to the traditional  L$^2$ ones. Due to the lack of a widely accepted concept of multivariate quantiles, however, quantile autoregression methods so far have been  limited to univariate time series. Building upon recent measure-transportation-based concepts of multivariate quantiles, we develop here a nonparametric vector quantile autoregressive approach to the analysis and prediction of (nonlinear as well as linear) multivariate time series.  
    \end{abstract}

{\small
\bigskip


}
 \vspace{0em}

\section{Introduction}
Classical time series analysis is firmly rooted in an L$^2$ approach and the linear geometry of the corresponding Hilbert spaces. That L$^2$ approach involves linear filters, linear ARMA, VAR, and VARMA models, second-order white noise innovation processes, and linear optimal point predictors minimizing expected quadratic prediction errors. It has, however,  
 two severe limitations:  it only deals with second-order dependencies and linear dynamics, and only yields point predictors of future values. 

Real-world data provide overwhelming evidence of nonlinear dynamics, and significant effort has been invested in modelling, estimating, and predicting nonlinear processes. The literature on nonlinear techniques in time series is extensive and still growing---see, e.g., \cite{Fan.2005.Book.TimeSeries} for a classical monograph---but it largely adheres to the same optimal point prediction paradigm as the classical approach.

Point predictors, just as point estimators of conditional means in regression analysis, are 
 providing poor summaries of future observations, and fail to exploit the full predictive information carried by the observed past. A remedy to this, in linear regression, was proposed in  the pathbreaking paper by \cite{KoenkerBassett78} with the introduction of the concept of {\it quantile regression}. Contrary to the classical point estimators of conditional means, quantile regression yields  consistent estimations of all the conditional quantiles of the response, hence of its entire conditional distribution. That appealing property of quantile regression was extended by \cite{koenker2006quantile}  to {\it quantile {\em auto}regression} and, since then, quantile autoregressive (QAR)  models have become a standard tool in time series econometrics as a powerful alternative to traditional AR models. 

While the original contributions by \cite{KoenkerBassett78}, \cite{Koul1995},  and \cite{koenker2006quantile} still involve some form of linearity,  they subsequently have been extended (see, e.g., \cite{Mukherjee1999, Cai_2002, koenker2005quantile, QU2015, koenker2017quantile, Handbook2017}) to more general  settings   with  nonlinear regression or autoregression---including  extensions to  Bayesian techniques, survival analysis,    instrumental variables,   high-dimensional and Banach-valued response,  cointegrated series, etc.\   (see Chapters 4, 7, 9, 14, 15, 17   of the Handbook volume \cite{Handbook2017} for references).  
The so-called nonparametric QAR   models, thus, allow for the consistent estimation of the conditional distribution of future observations without any specification of innovation densities nor analytical constraints on the form of conditional heterogeneity and~AR serial dependence.   They have been widely applied in a variety of forecasting and learning problems (see, for instance,~\cite{cheung2024quantiledeeplearningmodels}) and attracted considerable interest in financial econometrics, with the evaluation of Values at Risk and Expected Shortfalls, and the popular CAViaR models \citep{CAVIAR}.


However, the concept of quantile being based on the natural ordering of the real line, quantile regression and quantile autoregression so far  remain inherently restricted to univariate settings---single-output regression and univariate QAR models \citep{koenker2006quantile}, univariate  portfolio returns \citep{CAVIAR}, or linear vector autoregressive models involving vectors of (univariate) marginal quantiles \citep{Manganelli2023}, among many others. Due to the lack of  a canonical ordering   of ${\mathbb R}^d$ for~$d>1$,  genuinely multivariate quantile  concepts and quantile-based techniques for multiple-output regression and VAR models are more delicate.  Some interesting attempts  have been made---see \cite{ChaouchSaraccco2009}, \cite{HlubinkaSiman2013, HlubinkaSiman2015}, \cite{HallinSiman2016}, 
\cite{hallin2010multivariate}, \cite{HallinPdSiman2} 
 for multiple-output quantile regression,  \cite{GiannoneQVAR}, \cite{BayesQVAR2023} for   quantile vector autoregression.  However, as explained in~\cite{HallinSimanHandbook} and \cite{Barrio2022NonparametricMC},  none of these attempts (many of them based on the vector of univariate marginal quantiles  involves a genuinely multivariate and fully satisfactory concept of quantile. 
 
 Using the measure-transportation-based concept of {\it center-outward ranks} and {\it quantiles} introduced  by~\cite{chernoetal17} and \cite{Hallin2020DistributionAQ}, \cite{ Barrio2022NonparametricMC} have developed a multiple-output version of nonparametric  quantile regression that matches all the properties expected in a quantile regression approach: closed nested conditional quantile regions and contours, exact conditional coverage level irrespective of the underlying densities, etc. An earlier paper \citep{Carlier.Chernoz.Galichon.2016.AoS} had proposed a related  measure-transportation-based method for {\it linear} vector quantile regression which, however, does not lead to  the construction  of  conditional quantile regions and contours.   Based on the dual (in the sense of Kantorovich duality) concept of  {\it center-outward ranks},  rank-based testing and R-estimation for  linear VAR models with unspecified innovation densities have been developed in \cite{HLLBern, HLLJASA} and \cite{HLPortm}.   
 
 The objective of this paper is to propose a genuinely multivariate version (QVAR) of nonparametric univariate QAR  models based on the  concept of multivariate {\it center-outward quantiles} as introduced by \cite{Hallin2020DistributionAQ}.   Specifically, we   construct estimators of the predictive  $d$-dimensional  distribution---the conditional distribution at time $(t+1)$  of the variable under study given the observations up to time $t$. These estimators take the form of a collection of predictive center-outward quantile regions with a.s.\ conditional coverage probability $\tau\in [0,1]$, with obvious applications, e.g., in the prediction of multivariate value-at-risk or expected shortfall. Contrary to the depth-based concept considered in \cite{HallinPdSiman2} and the spatial or geometric quantiles introduced by \cite{chaudhuri1996geometric}  and \cite{Chowdhury19},  center-outward quantiles and the related ranks and signs enjoy (under absolutely continuity) all the properties expected from such notions. In particular, the predictive center-outward quantiles proposed in this paper  fully characterize the underlying (conditional) distributions,  yield  quantile regions with 
  exact (conditional) coverage probability, and define center-outward ranks and signs that are distribution-free and maximal ancillary: see \cite{Hallin2020DistributionAQ} and its online supplement for details and a discussion of these properties. $\,$\vspace{-4mm}

\paragraph{Outline of the paper.} The  paper is organized as follows. Section~2  deals with the popu\-lation concepts of multivariate conditional quantiles and predictive quantile regions for stationary nonparametric VAR processes (Markov processes of order $p$). Section~3 proposes empirical counterparts of these concepts, then studies their consistency 
 and consistency rates.  Section~4 provides simulation-based numerical results and a real-data application. All proofs are postponed to an online appendix.

\section{Center-Outward Quantiles}\label{Section:quantiles}
\subsection{Quantile functions}
Let $\mu_d$ denote the spherical uniform distribution  over the unit ball $\mathbb{B}^d\coloneqq \{u\in \R^d: \|u\|<1\}$ in ${\mathbb R}^d$---that is, the distribution of the random vector $U\coloneqq R \sigma$, where $R$ and $\sigma$ are mutually independent, $R$ is uniformly distributed over $[0,1]$, and $\sigma$ uniformly distributed on the unit sphere $\mathcal{S}^{d-1}\coloneqq \{u\in \R^d:\|u\|=1\}$.  \cite{Hallin2020DistributionAQ} define  the {\it center-outward quantile function} of a probability distribution $\rm P$ in the family $\mathcal{P}(\R^d)
$ of Lebesgue-absolutely continuous probability measures over ${\mathbb R}^d$ as follows.

\begin{definition}
    The {\it center-outward quantile function} $\mathbf{Q}_\pm$ of   ${\rm P}\in\mathcal{P}(\R^d)$ is the $\mu_d$-a.s.~unique gradient $\mathbf{Q}_\pm=\nabla\varphi $ of a convex function $\varphi$ 
     pushing $\mu_d$ forward to $P$. 
\end{definition}
\noindent This definition  is based on a  famous theorem  by McCann \citep{McCann}, which guarantees the existence and $\mu_d$-a.s.~uniqueness of $\mathbf{Q}_\pm$.


The mapping $\mathbf{Q}_\pm=\nabla\varphi$, however, is only a.e.~defined  in the open unit ball $\mathbb{B}^d$. It is easily extended via the sub-gradient 
\[ \mathbb{B}^d \ni u\mapsto 
\partial \varphi(u) \coloneqq \left\{ x \in \mathbb{R}^d : \varphi(v) \geq \varphi(u) + \langle x, v - u \rangle \quad \text{for all } v \in \mathbb{B}^d \right\}
\]
which, for a convex  $\varphi$, is a maximal monotone set-valued mapping. Refer to~$\mathbb{Q}_\pm\coloneqq \partial \varphi$ as the {\it set-valued quantile mapping of} $\rm P$.
Since the support of $\mu_d$ is connected and~$ \varphi$ is the unique (up to additive constants) convex function such that $\mathbf{Q}_\pm=\nabla\varphi$,  the set-valued quantile mapping~$\mathbb{Q}_\pm$ of $\rm P$ is uniquely defined on $\mathbb{B}^d$.

\subsection{Conditional quantile functions}\label{Section:quantiles-cond}
In this section, we introduce  conditional center-outward quantile functions as  set-valued~opera\-tors.  In the univariate setting, \cite{Castro.2023.Quantiles.BJ} recently considered a similar approach. 

Let  $X$, with values in $({\R}^d, {\mathcal B}^d)$ ($ {\mathcal B}^d$ the Borel sigma-field on $\R^d$), be defined on some probability space $(\Omega, {\mathcal A}, {\mathbb P})$. Denoting by ${\rm P}\coloneqq {\mathbb P}\circ X^{-1}$ its  distribution, assume that ${\rm P}\in~\!\mathcal{P}(\R^d)$.  
 Recall that the conditional probability distribution~$\mathbb{P}_{\X\vert \mathcal{G} }$ of $ \X $ given the (sub)-$\sigma$-field $\mathcal{G}\subseteq \mathcal{A}$ of $\mathcal{A}$ is defined as the unique (up to a set of $\omega$ values contained in a set $A\in{\mathcal A}$ of $\mathbb P$-probability zero) function~$\mathbb{P}_{\X\vert \mathcal{G} }:\mathcal{B}^d\times \Omega\to [0,1]$ such that \smallskip

\begin{compactitem}
    \item[--] for any $\omega\in \Omega$, the map
    $  \mathcal{B}^d \ni B\mapsto \mathbb{P}_{\X\vert \mathcal{G} }(B,\omega) $
    is a probability measure on $(\mathbb{R}^d,  \mathcal{B}^d)$,
    \item[--] for any $B\in \mathcal{B}^d $, the map 
    $ \Omega\ni \omega \mapsto \mathbb{P}_{\X\vert \mathcal{G} }(B,\omega) $
    is $\mathcal{G}$-measurable  and     satisfies the functional equation
$\int_G \mathbb{P}_{\X\vert \mathcal{G} }(B, \omega) {\rm d}\mathbb{P}(\omega)=\mathbb{P}(\{\X\in B\}\cap G)$ for all $G\in \mathcal{G}$ and $B\in \mathcal{B}^d$. 
\end{compactitem}

\begin{definition}\label{definition:setcond}
     The {\it set-valued center-outward quantile map}  of $\X$ conditional on $\mathcal{G}$ is the unique map $ \mathbb{B}^d\times \Omega  \ni (u,\omega) \mapsto \mathbb{Q}_{ X\vert \mathcal{G} }(u| \omega) \in 2^{\R^d}$ 
such that, for every $ \omega\in \Omega $,  $\mathbb{Q}_{\X\vert \mathcal{G} }(\cdot,  \omega)$ is the set-valued quantile mapping of   $ \mathbb{P}_{X\vert \mathcal{G} }(\cdot,\omega).$  Call {\it set-valued center-outward distribution map} of~$\X$ conditional on $\mathcal{G}$ 
the mapping 
 $ (y, \omega)\mapsto \mathbb{F}_{\X\vert \mathcal{G} }(y,  \omega) \coloneqq  \{ u\in \mathbb{B}^d:\, x\in \mathbb{Q}_{\X\vert \mathcal{G} }(u,  \omega) \}.$
\end{definition}
Denote by $\mathcal{B}(\mathcal{U})$ the Borel $\sigma$-field of a  Polish space $\mathcal{U}$. The following result shows that~$\mathbb{Q}_{\X\vert \mathcal{G} }(\cdot,  \cdot)$ and $\mathbb{F}_{\X\vert \mathcal{G} }(\cdot,  \cdot)$ are $\mathcal{G}\otimes \mathcal{B}(\mathbb{B}^d)$- and $\mathcal{G}\otimes \mathcal{B}^d$-measurable, respectively, where~$\otimes$ stands for the  product of $\sigma$-fields. 
Recall from \cite[chapter~14]{rockafellar2009variational} that a set-valued mapping $\mathbb{M}:\Omega\to 2^{\R^d}$, where $(\Omega,\mathcal{A})$ is a measurable space, is $\mathcal{A}$-measurable if \vspace{-1mm}
$$\mathbb{M}^{-1}(A)\coloneqq \{ \omega\in \Omega:\ \mathbb{M}(\omega) \cap A\neq \emptyset \} \in \mathcal{A}\vspace{-1mm} \quad\text{for any open or closed $A\subset \R^d$.}$$ 

\begin{lemma}\label{lemma:set-valued-prediction-measure}
    Let $(\Omega, \mathcal{A}, \mathbb{P})$ be a probability space and denote by $ \mathcal{G}
    $  a sub-$\sigma$-field of $\mathcal A$. Then,  
      (i)   $\mathbb{Q}_{ X\vert \mathcal{G} }$ is $\mathcal{G}\otimes \mathcal{B}(\mathbb{B}^d)$-measurable  and 
     (ii)    $\mathbb{F}_{ X\vert \mathcal{G} }$ is $\mathcal{G}\otimes \mathcal{B}^d$-measurable. 
\end{lemma}
\begin{definition}
    Call {\it conditional center-outward quantile function $\mathbf{Q}_{ X\vert \mathcal{G} }$ of $X$ given $\mathcal{G}$} any measurable selection of $\mathbb{Q}_{ X\vert \mathcal{G} }$ and  {\it conditional center-outward distribution function $\mathbf{F}_{ X\vert \mathcal{G} }$ of~$X$ given $\mathcal{G}$} any measurable selection of $\mathbb{F}_{ X\vert \mathcal{G} }$.
\end{definition}
\begin{remark} 
 Theorem 2.1 in  \cite{Carlier.Chernoz.Galichon.2016.AoS} establishes the joint measurability of~$\mathbf{Q}_{X|\mathcal{G}}$ for $\Omega=\R^d\times \R^m $, with $\mathbb P$  the joint probability distribution of the $(d+m)$-dimensional  random vector $(X,Z)$,  and $\mathcal{G}$   the $\sigma$-field generated by  the {\it vertical strips}, that is, the product sets of the form  $\R^d\times E$ with $E\in \mathcal{B}
 ^m$. Their proof  readily extends   to general measurable spaces, yielding an analog of \cref{lemma:set-valued-prediction-measure}.  This is not sufficient to conclude the measurability of $\mathbb{Q}_{X|\mathcal{G}}$, though. In other words, what  Theorem~2.1 of \cite{Carlier.Chernoz.Galichon.2016.AoS} proves is the existence of a measurable selection while the measurability of a set-valued mapping requires  the existence of a dense countable family of measurable selections---a {\it Castaing representation} \cite[see][Theorem 14.5]{rockafellar2009variational}.\vspace{-3mm}
\end{remark}
\subsection{Prediction quantile functions and regions}\label{Section:quantiles-pred}
Let ${\bf X}\coloneqq \{X_t\vert\, t\in{\mathbb Z}\}$ be a time series defined over a probability space $(\Omega, \mathcal{A}, \PP)$ and denote by~$\mathcal{F}_{\leq t}{\subset \mathcal{A}}$ the $\sigma$-field generated by $\{X_s\vert\, s\leq t \}$.  Define the {\it one-step-ahead prediction quantile set-valued mapping $\mathbb{Q}_{t+1|t}$ of $\bf X$ at time $t$}  as the conditional center-outward quantile set-valued mapping $\mathbb{Q}_{ X_{t+1}\vert \mathcal{F}_{\leq t} }$  of $X_{t+1}$ given $\mathcal{F}_{\leq t}$ and call {\it one-step-ahead prediction quantile function of~$\bf X$ at time $t$}  
  any measurable selection $\mathbf{Q}_{t+1|t}$ of $\mathbb{Q}_{t+1|t}$.   Similarly, define the {\it one-step-ahead prediction  distribution set-valued  function   of $\bf X$ at time $t$ as  
$\mathbb{F}_{t+1|t}
{ \coloneqq \mathbb{F}_{ X_{t+1}\vert \mathcal{F}_{\leq t} }}$}  and   call  {\it one-step-ahead prediction distribution function of~$\bf X$ at time~$t$} any measurable selection $\mathbf{F}_{t+1|t}$ of~$\mathbb{F}_{t+1|t}$. 

In practice, conditional prediction quantiles are used to construct prediction quantile regions. Define the {\it one-step-ahead prediction quantile region of order $\tau\in (0,1)$  of $\bf X$ at time~$t$}   as 
 the set-valued mapping
$$\Omega \ni \omega\mapsto \mathcal{R}_{t+1|t}(\tau| \omega)\coloneqq \mathbb{Q}_{t+1|t}(\tau \overline{\mathbb{B}^d}|\omega)\coloneqq \bigcup_{\|u\|\leq \tau} \mathbb{Q}_{t+1|t}(u|\omega) ,  \vspace{-2mm}$$
where $\tau \overline{\mathbb{B}^d}$ denotes the  closed unit ball with center $0$ and radius $\tau$, and 
the  {\it one-step-ahead autoregression median} as the set-valued mapping 
$$\Omega \ni \omega\mapsto m_{ t+1|t}(\omega)=\bigcap_{\tau\in (0,1)}\mathcal{R}_{t+1|t}(\tau| \omega) . \vspace{-2mm}$$
These one-step-ahead concepts straightforwardly extend to $k$-steps-ahead ones, $k\in{\mathbb N}$, with obvious notation $\mathcal{R}_{t+k|t}(\tau| \omega)$ and $m_{ t+k|t}(\omega)$ and similar properties. 

%
%

The following result shows that the prediction quantile regions and autoregression median are $\mathcal A$-measurable, and  that the probability content of the region of order $\tau$ is $\tau$.

\begin{lemma}\label{lemma:proba-control}
    For every $\tau\in (0,1)$, the event\vspace{-1mm}
\begin{equation}\label{Lemma2.6(1)}
X_{t+1}\in \mathcal{R}_{t+1|t}(\tau| \cdot)  =  \{ \omega\in \Omega: \mathbb{F}_{t+1|t}( X_{t+1}(\omega), \omega) \cap  \tau\, \overline{\mathbb{B}^d}\neq \emptyset \} \in \mathcal{A}\vspace{-2mm}
\end{equation}
and  satisfies \vspace{-1mm}
\begin{equation}\label{Lemma2.6(2)}  \mathrm{P}\Big( X_{t+1}\in  \mathcal{R}_{t+1|t}(\tau|\cdot) \bigg \vert \mathcal{F}_{\leq t}\Big)  \geq  \tau  \quad \mathbb{P}-\text{\rm a.s.}\vspace{-1mm}
\end{equation}

 \color{black} If, moreover, 
\begin{equation}\label{moreover}
   \mathbb{P}_{X_{t+1}\vert \mathcal{F}_{\leq t} }(\cdot ,\omega) \ll \ell_d \quad\text{ $\mathbb P$-a.s.}
   \end{equation}
    (where $\ell_d$ denotes  the Lebesgue measure over $(\R^d, \mathcal{B}^d)$), 
then, for every $\tau\in (0,1)$, \vspace{-1mm}
\begin{equation}\label{Lemma2.6(3)}  \mathrm{P}\Big( X_{t+1}\in  \mathcal{R}_{t+1|t}(\tau|\cdot) \bigg \vert \mathcal{F}_{\leq t}\Big)  =  \tau  \quad \mathbb{P}-\text{\rm a.s.}
\end{equation}
\vspace{-6mm}\end{lemma}

\section{Estimation and Prediction}\label{Section:Estimation}
\subsection{Empirical prediction quantiles}
  Recall that a time series $\{X_t\vert\, t\in{\mathbb Z}\}$ is  strictly stationary if, for all $h\in \Z$, $m\in{\mathbb N}$, and~$\{t_1, \dots, t_m\}\subset ~\!\Z$,  the random vectors  $(X_{t_1}, \dots, X_{t_m})$ and $(X_{t_1+h}, \dots, X_{t_m+h})$ are equally distributed. The same $\{X_t\vert\, t\in{\mathbb Z}\}$  is Markov of order $p$ if, for any $f:\R^d\to \R$ continuous and  bounded,
   $\E[f(X_{t+1})\vert \mathcal{F}_{\leq t} ]=\E[f(X_{t+1})\vert (X_t, X_{t-1},\ldots,X_{t-p+1})].$

Let $x^T\coloneqq ( x_1, x_2, \ldots, x_T )$ be an observed sample from the  strictly stationary Markovian  time series of order $p=1$   (extensions to $p>1$ are straightforward) ${\bf X}\coloneqq \{X_t\vert\, t\in{\mathbb Z}\}$. Denote by  ${\rm P}_1$ the distribution of $X_1$, by  $ {\rm P}_{1,2}$ the distribution of the pair $(X_1,X_2)$, and  assume   that~${\rm P}_1\ll \ell_d$, with density $p_1$. For a density point $x$ of ${\rm P}_1$,   denote by $ {\rm P}_{2|1}(\cdot\vert x) $ the conditional distribution of~$X_{2}$ given $X_1=x$. Assuming that $x$ is such that~$ {\rm P}_{2|1}(\cdot\vert x) $ has density ${p}_{2|1}(\cdot | x)$, write~$\mathbf{Q}_{t+1|t}(\cdot|x)$ for an arbitrary one-step-ahead prediction quantile mapping of $ {\rm P}_{2|1}(\cdot\vert x) $.

Our estimates of the predictive quantiles of ${\bf X}$ require the construction of a $k$-point \emph{regular grid} $\mathfrak{U}=\mathfrak{U}(T)\coloneqq \{u_1, \dots, u_{k}\}$ of $\mathbb{B}^d$ where  $ k= k(T)$ factorizes into $k_R k_S+ 1$ and the integers~$k_R=k_R(T)$ 
and~$k_S=k_S(T)$ tend to infinity as $T\to\infty$; these $k(T)$ gridpoints  are obtained  as the intersections  between
\begin{compactenum}
	\item[--] the $k_S(T)$
	 rays associated with a $k_S(T)$-tuple of unit vectors ${v}_1, \dots,{v}_{k_S(T)}\in\R^d$ such that~${1}/{(k_S(T))}\sum_{j=1}^{k_S(T)} \delta_{{v}_j}$ tends weakly, as $T\to\infty$, to the uniform distribution over the unit sphere $\mathcal{S}^{d-1}$, and
	\item[--] the $k_R(T)$ hyperspheres with center $\mathbf{0}$ and radius ${j}/({k_R(T)+1})$, $j=1, \dots, k_R(T)$,
\end{compactenum}
along with  the origin.  Associated with this grid is the empirical measure\vspace{-1mm} 
  \begin{equation}\label{measure_cond_Reg}
     {\rm \mu}_d^{(k(T))}\coloneqq  \frac{1}{k(T)}\sum_{j=1}^{k(T)}\delta_{u_j},
\vspace{-1mm}  \end{equation}
  which, as $k_R\rightarrow\infty$ and $k_S\rightarrow\infty$, converges weakly to the spherical uniform $\mu_d$ over the unit ball $\mathbb{B}^d$.   Let \vspace{-2mm}
  \begin{equation}\label{NW-estimator}
      \widehat{\rm P}_{X_{t+1} \vert X_t=x} \coloneqq \sum_{i=1}^{T-1}  w_{i+1}^x \cdot \delta_{X_{i+1}}\quad \text{ with } \quad w_{i+1}^x \coloneqq  \frac{ K\left(\frac{X_i-x}{h} \right) }{\sum_{j=1}^{T-1}  K\left(\frac{X_j-x}{h} \right) }\vspace{-1mm}
  \end{equation}
 denote a   Nadaraya-Watson estimator, based on some  appropriate kernel $K$  and bandwidth~$h$,  of the  predictive probability measure ${\rm P}_{X_{t+1} \vert X_t=x}$.  This estimator is 
  used  in 
   the following empirical optimal transport problem from $  {\rm \mu}_d^{(k(T))}$ to    $\widehat{\rm P}_{X_{t+1} \vert X_t=x}$:\vspace{-2mm} 
 \begin{align}
    \begin{split}\label{kanto_reg}
    	&{\hat{\pi} \in  \argmin_{\pi}\sum_{i=1}^k \sum_{j=2}^{T}   \frac{1}{2}\, \| u_i - X_j\|^2 \pi_{i,j} \ ,}\vspace{-1mm}\\
    	\text{subject to} &\ \ \sum_{j=2}^{T}  \pi_{i,j}= \frac{1}{k} \ \text{for all $i\in \{1, 2, \dots, k \}$},\vspace{-1mm}\\
            &\ {\sum_{i=1}^k \pi_{i,j}{= w^x_{j}} =\frac{  K\left(\frac{X_{j-1}-x}{h} \right) }{\sum_{t=2}^{T}  K\left(\frac{X_{t-1}-x}{h} \right) } \ \text{for all $j\in \{2, \dots, T \}$},} \\
    	&\ \pi_{i,j}\geq 0 \ \text{for all } i\in \{1, 2, \dots, 
        k \} \text{ and } j\in \{2, \dots, T \}.
\end{split}
\end{align}
It follows from  \cite{villani2003topics} that the solution $\hat{\pi}$  of \eqref{kanto_reg} has monotone support, i.e.,  is such that 
$  \langle x_{i_1}-x_{i_2},  u_{j_1}-u_{j_2} \rangle \geq 0    $ for all $(i_1,j_1)$ and $(i_2,j_2)$ for which  $\pi_{i_1, j_1}>0$  and  $\pi_{i_2,j_2}>0$. 
 We then define the empirical prediction quantile at the gridpoints as 
\begin{equation}\label{cond_quantile_estimator}
    \{u_1, \dots, u_k\} \ni  u_i\mapsto \widehat{\mathbf{Q}}_T(u_i|x)\coloneqq k \sum_{j=2}^T  \hat{\pi}_{i,j}\cdot x_j  .
\end{equation}

Note that for some choices of the kernel $K$, such as the indicator $K(x)=\mathbb{ I}_{[\|x\|\leq 1]}$,   the vertexes of the polytope defining the linear program \eqref{kanto_reg}, as a consequence of the Birkhoff theorem (see \cite{Birkhoff.46}), are (weighted) permutation matrices. In this case, $\hat \pi$ is already concentrated in the graph of $ u\mapsto \widehat{\mathbf{Q}}_{ T }(u|x)$.   The following result shows that $\widehat{\mathbf{Q}}_{ T }(\cdot| x)$ has monotone support. 
\begin{remark} Due to stationarity, $  \widehat{\rm P}_{X_{t+1} \vert X_t=x} $ in \eqref{measure_cond_Reg} is, for all $t$ and $x$,  an   estimator of the one-step-ahead predictive distribution of $X_{t+1}$ computed at time $t$. In practice, however, being based on  observations up to time $T$, $  \widehat{\rm P}_{X_{t+1} \vert X_t=x} $ cannot be used as a predictor for~$t<T$. Therefore, in the sequel, we are only considering $  \widehat{\rm P}_{X_{T+1} \vert X_T=x} $ and the empirical one-step-ahead prediction quantiles, quantile regions, and quantile contours  computed at time~$T$.\vspace{-1mm}
\end{remark}
\color{black}
\begin{lemma}\label{lemma:monotone}
The empirical prediction quantile $ u\mapsto \widehat{\mathbf{Q}}_{ T }(u|x)$ is  monotone at the gridpoints, i.e.,    for all $r,s\in \{1, \dots, k\}$ and $x\in \R^d$, 
    $ \langle \widehat{\mathbf{Q}}_{ T }(u_s|x)- \widehat{\mathbf{Q}}_{ T }(u_r|x), u_s-u_r \rangle \geq 0$, $\mathbb P$-a.s. 
   
\end{lemma}

 If the function $u\mapsto {\widehat{\mathbf{Q}}}_{ T }(u| x)$ is to be extended  beyond the gridpoints, we choose any continuous maximal monotone interpolator of the  points $(u_i, \widehat{\mathbf{Q}}_{ T }(u_i| x))$, $i=1,\ldots,k$; see \cite{Hallin2020DistributionAQ,Barrio2022NonparametricMC} for details. For the sake of simplicity, we concentrate on autoregressions of   order $p=1$; the $p>1$ case readily  follows along the same~lines.\vspace{-1mm}

\subsection{Consistency}\label{Section:Estimation-consitency}

The consistent estimation of time series requires some assumptions on the impact of the observation  $X_t$ at time $t$ on the observation at time  $t+m$ as $m\to \infty$. In the literature, the evolution of this impact is generally measured by the so-called mixing conditions (see  \cite{Bradley.Mixing.Survey}). Another common assumption is the recurrence of the process (see \cite{Yakowitz.NN.SPaA,Sancetta.2009.JMA, Cai_2002}; \cite{Karlsen.AoS}, among others,  for $K$-nearest neighbors and Nadaraya–Watson autoregressors). The following  mixing condition is standard in nonpara\-metric time series estimation and was originally introduced in \cite{Rosenblatt.CLT.1956}.   Throughout, let ${\bf X}\coloneqq \{X_t\vert\, t\in\Z\}$,  $ \mathcal{F}_{\leq t}\coloneqq\sigma(\{X_{s}\}_{s\leq t})\subset{\mathcal A}$, and~$ \mathcal{F}_{\geq t}\coloneqq\sigma(\{X_{s}\}_{s\geq t})\subset{\mathcal A}$.
\begin{definition}[$\alpha$-mixing]
    A strictly stationary time series ${\bf X}$ is {\it $\alpha$-mixing} if\vspace{-1mm}
    $$ \alpha(m) \coloneqq \sup_{A\in \mathcal{F}_{\leq t},B \in \mathcal{F}_{\geq t+m}} |\PP(A\times B)-\PP(A)\PP(B)| \to 0\quad\text{ as $m\to \infty$.} $$
\end{definition}
Note that $\alpha(m)$ is upper- and lower-bounded by \vspace{-1mm}
$$ \alpha'(m)\coloneqq \sup_{U\in  \mathcal{B}_{\leq t,\infty}, U \in \mathcal{B}_{\geq t+m,\infty}} |\E(UV)-\E[U]\E[V]|,\vspace{-2mm}$$
where $\mathcal{B}_{\leq t,p}$ and $\mathcal{B}_{\geq t,p}$  denote the unit balls in $L^{p}(\mathcal{F}_{\leq t}, \PP)$ and $L^{p}(\mathcal{F}_{\geq t}, \PP )$, respectively,   for~$p\in~\![1, \infty]$.
The following condition, which is related to the notion of {\it $\beta$-mixing} (see \cite[Theorem 3.7]{Bradley.Mixing.Survey})
was used by Rosenblatt to derive the consistency of  kernel density estimators for Markov processes  \citep{Davis2011}. 
\begin{definition}[Geometric ergodicity]
    A strictly stationary time series ${\bf X}$ is {\it geometrically  ergodic of order two} if 
    $ \beta(m)\coloneqq \sup_{U\in  \mathcal{B}_{\leq t,2}, U \in \mathcal{B}_{\geq t+m,2}} |\E(UV)-\E[U]\E[V]| $ 
    decreases exponentially fast as $m\to\infty$.
\end{definition}
The following assumptions are  standard in   regularity results for  center-outward quantiles (see \cite{Barrio2023RegularityOC,FIGALLI2018413,DELBARRIO2020104671}). For the regularity of  conditional (with respect to~covariates) quantiles, we refer the reader to \cite{gonzalezsanz2024linearizationmongeampereequationsstatistical}.   
\begin{assumption}\label{Assumption-On-density}(Regularity condition)
For all $x$ in the support  ${\rm supp}(\rm P_1)$ of~$\rm P_1$, the conditional distribution ${\rm P}_{2|1}(\cdot|x)$  is supported on a convex set,  and   its  density~${p}_{2|1}(\cdot|x)$ is continuous and bounded away from zero in that support.
\end{assumption}
\begin{remark}
    Under \cref{Assumption-On-density},  it follows from \cite{Barrio2023RegularityOC}, \cite{FIGALLI2018413}, and \cite{DELBARRIO2020104671}  that, for all  $x$ in the support of ${\rm P}_1$, $\mathbf{Q}_{t+1|t}(\cdot|x)$ is continuous in $\overline{\mathbb{B}^d}\setminus \{0\}$ and~$\mathbf{F}_{t+1|t}(\cdot|x)$ can be extended to be  continuous over $ \R^d $. However, as $\mathbf{Q}_{t+1|t}(\cdot|x)$ might   be discontinuous at $0$, the median
     $\mathbb{Q}_{t+1|t}(0|x)$ could be set-valued (not a singleton). 
\end{remark}

Next, let us introduce a kernel function $K$ with the following properties.

\begin{assumption}\label{Assumption-On-kernel}
The   kernel function $K$  
    is nonnegative,
 bounded, and 
 integrates to one in  the Lebesgue measure.
\end{assumption}
We then have the following results. 
 \begin{lemma}\label{lemma:consistency}
   Let $X_1, \dots, X_T$ be a realization of a strictly stationary Markov   pro\-cess~${\bf X}$ of order one satisfying \cref{Assumption-On-density}.  Let the kernel~$K$ satisfiy \cref{Assumption-On-kernel}.  Fix~$x\!\in~\!\!{\rm supp}(\rm P_1)$ with $p_1(x)>0$ and assume that one of the following conditions  holds: 
    \begin{compactenum}
        \item ${\bf X}$  is geometrically ergodic of order two,  $h\to 0$, and $h^d  T \to \infty$ as $T\to\infty$; or  
        \item  ${\bf X}$  is $\alpha$-mixing with $\alpha = \alpha (m)$ decreasing exponentially fast as $m\to\infty$,  $h\to 0$, and~$h^{2d}  T \to ~\!\infty$ as $T\to\infty$.
    \end{compactenum}
    Then, for any continuous bounded function $f:\R^d\to \R$, letting $\widehat{\rm P}_{x}\coloneqq  \widehat{\rm P}_{X_{t+1} \vert X_t=x} $,
    $$ \int f {\rm d} \widehat{\rm P}_{x}= \frac{\sum_{t=1}^{ T -1}  f(X_{t+1})K\left(\frac{x-X_t}{h} \right) }{\sum_{t=1}^{ T -1}  K\left(\frac{x-X_t}{h} \right) } \overset{  \PP  }{\longrightarrow} \E[f(X_2)\vert X_1=x]\quad\text{as $T\to\infty$}.  $$
\end{lemma}

Arguing as in \cite{Barrio2022NonparametricMC}, these results entail the pointwise consistency of the autoregression  quantiles $ \widehat{\mathbf{Q}}_T  ( u\vert x)$. The proof being exactly the same, it  is omitted; details are  left to the reader.

\begin{theorem}\label{theorem-consistency}   Let $X_1, \dots, X_{ T }$ be a realization of a strictly stationary Markov  process ${\bf X}$  of order one 
 satisfying \cref{Assumption-On-density}. Let the kernel~$K$ satisfiy \cref{Assumption-On-kernel}.\linebreak Fix $x\in {\rm supp}(\rm P_1)$ with~$p_1(x)>0$ and assume that one of the two conditions (i) and (ii) of Lemma~\ref {lemma:consistency} holds. 
 Then, for any compact subset  $\mathcal{K}$ of $\mathbb{B}^d\setminus \{0\}$, 
$ {\mathbf{Q}}_{2|1}( u\vert x) ={\mathbf{Q}}_{t+1|t}( u\vert x)$ is well defined (and, due to stationarity, does not depend on $t$) for all $x\in\mathcal{K}$, and
  $$ \sup_{u\in \mathcal{K}}\| \widehat{\mathbf{Q}}_T  ( u\vert x)- {\mathbf{Q}}_{2|1}( u\vert x) \| \overset{  \PP  }{\longrightarrow} 0\quad\text{as $T\to\infty$} .$$
\end{theorem}
These pointwise limits, which are standard in the literature, are not very practical for prediction, though. 
 The following result addresses (for one-step-ahead prediction in stationary  Markov  processes of order one) this issue under the following additional assumption. 
\begin{assumption}\label{assumtpion-density-bounded}
For each $R>0$, there exists $ \Lambda_R>0$ such that 
$$ p_{1|2}(x_2|x_1) \leq \Lambda_R \quad \text{for all $x_1\in {\rm supp}(\rm P_1)\cap R\,\mathbb{B}^d$ and $x_2\in {\rm supp}(\rm P_{1|2}(\cdot|x_1))\cap R\,\mathbb{B}^d$.} $$ 
\end{assumption}

\begin{theorem}\label{Markov:Estimation} Let $\bf X$ be a strictly stationary Markov    process of order one satisfying Assumptions~\ref{Assumption-On-density} and~\ref{assumtpion-density-bounded}. Let the kernel~$K$ satisfiy \cref{Assumption-On-kernel}  and assume  that one  of the two conditions (i) and~(ii) of Lemma~\ref {lemma:consistency} holds. Then, 
for any compact subset  $\mathcal{K}$ of $\mathbb{B}^d\setminus \{0\}$,\vspace{-1mm}  
    $$ \sup_{u\in \mathcal{K}}\| \widehat{\mathbf{Q}}_{ T } ( u\vert X_{{ T }})- {\mathbf{Q}}_{2|1}( u\vert X_{{ T }}) \| \xrightarrow{\PP } 0 \quad\text{as $T\to\infty$}.\vspace{-1mm} $$
\end{theorem}
\cref{lemma:proba-control} and \cref{Markov:Estimation} provide the  asymptotic probability control over the quantile prediction regions, thereby allowing for ``interval prediction.''  Proofs are omitted as they follow the same arguments as  the proof of \cite[Corollary~3.4]{Barrio2022NonparametricMC}.  
\begin{corollary}
   Under the assumptions of \cref{Markov:Estimation},       for any $\tau\in [0,1)$,  \vspace{-1mm}    
    $$ \PP\left( X_{{ T +1}}\in \mathcal{R}^{( T )}(\tau \vert X_{{ T }}) \vert X_{{ T }}\right) \xrightarrow{\PP } \tau  \quad\text{as $T\to\infty$}.$$
\end{corollary}

\begin{remark}
Instead of a unique realization of the process $\bf X$, one might observe $N>1$ independent realizations $X^{n}_1,\ldots,X^{n}_{T_n}$, $n=1,\ldots,N$, of $\bf X$ (see Section~\ref{realsec}
 for an example). Then, the averaged estimators\vspace{-2mm}  
  \begin{equation}\label{NW-estimatorn}
      \widehat{\rm P}_{X^n_{t+1} \vert X^n_t=x}^{(N)} \coloneqq\frac{1}{N}\sum_{n=1}^N \sum_{i=1}^{T_n-1}  w^{x;n}_{i+1} \cdot \delta_{X^n_{i+1}}\quad \text{ with } \quad w_{i+1}^{x;n} \coloneqq  \frac{ K\left(\frac{X^n_i-x}{h} \right) }{\sum_{j=1}^{T_n-1}  K\left(\frac{X^n_j-x}{h} \right) }
\vspace{-1mm}   \end{equation}
  naturally replace $\widehat{\rm P}_{X_{t+1} \vert X_t=x} $ as defined in \eqref{NW-estimator}, to which they reduce for $N=1$; the resul\-ting~$\widehat{\mathbf{Q}}_{ T }^{(N)} $ enjoy, {\it mutatis mutandis,} the same properties as soon as $T\coloneqq \sum_{n=1}^N T_n\to\infty$.  Details are left to the reader.\vspace{-3mm}
  \end{remark}
 \color{black}
 
\subsection{Consistency rates}\label{Section:Estimation-rates}
 Our first result shows a upper bound in local $L^2$-distance between the empirical and population quantile functions.  { Let $\nu_1$ and $\nu_2$ be probability measures over $\R^d$. Define \vspace{-3mm} }
$$d_{{\rm BLC}}(\nu_1, \nu_2)\coloneqq \sup_{f\in {\rm BLC}(\R^d)} \left\vert \int f {\rm d}\nu_1 -\int f {\rm d}\nu_2 \right\vert  \vspace{-6mm}$$
where\vspace{-2mm} 
\begin{multline*}
    {{\rm BLC}}(\R^d)\coloneqq \big\{f:\R^d\to \R \ \text{is convex and such that}\\   \vert f( x )-f( y )\vert \leq \| x - y  \| \ {\rm and}\ \vert f( x )\vert \leq 1, \ \text{for all }  x , y \in \R^d\big\}\vspace{-3mm} 
\end{multline*}
denotes the {\it Bounded-Lipschitz-Convex (BLC) semi-metric}.

In the sequel we use the following notation. Let $\{a_n\}$ and $\{b_n\}$ be deterministic sequences of real numbers.  Write $a_n\lesssim b_n$ if there exists a constant $C$ independent of $n$ such that~$a_n\leq~\!C b_n$ for all $n$. For a real-valued  random process  $\{Z_t\vert\, t\in\N\}$ {defined over some~$(\Omega, \mathcal{A}, \PP)$},  write~$Z_t=\mathcal{O}_\PP(|a_t|)$ if~$Z_t/|a_t|$ is stochastically bounded, i.e.~if, for any $\epsilon>0$, there exists~$M_\epsilon>0$ such that  
$ \PP\left({|Z_t|}/{|a_t|} \geq M_\epsilon\right) \leq \epsilon$ for all $t$.

\begin{lemma}\label{lemma:wassersteinEstability}  Let $X_1, \dots, X_n$ be a realization of a strictly stationary Markov process $\bf X$ of order one satisfying \cref{Assumption-On-density}. Let the kernel~$K$ satisfiy \cref{Assumption-On-kernel}. \linebreak Fix~$x\in {\rm supp}(\rm P_1)$ with~$p_1(x)>0$ and such that  ${\rm P}_{2|1}(\cdot|x)$ is  $\alpha$-H\"older in ${\rm int}({\rm supp}({\rm P}_{2|1}(\cdot|x)))$ for some $\alpha\in (0,1)$. Assume that 
 one  of the two conditions (i) and~(ii) of Lemma~\ref {lemma:consistency} holds 
 and set 
$$ \mathcal{V}_{{ T }}\coloneqq \big(\mathbf{Q}_{2|1}(\cdot|x)\big)^{-1}(\mathcal{K}') \cap \big(\widehat{\mathbf{Q}}_{T}(\cdot|x)\big)^{-1}(\mathcal{K}' ) $$
where
$\mathcal{K}'$ is a compact subset of ${\rm int}({\rm  supp}({\rm P}_{2|1}(\cdot|x))) \setminus \mathbf{Q}_{2|1}(0|x). $ 
Then,  
    $$ \E\left[\int_{\mathcal{V}_{{ T }}}\| \widehat{\mathbf{Q}}_T  ( u\vert x)- {\mathbf{Q}}_{2|1}( u\vert x) \|^2 {\rm d}\mu_d^{(k)}(u) \right]\lesssim  
\EE\left[d_{\rm BCL}( \widehat{\rm P}_{ T }, {\rm P}) \right]+d_{\rm BCL}( \mu_d^{(k)},\mu_d), $$
where $ \mu_d^{(k)}$ is defined as  in \eqref{measure_cond_Reg} for $k=k(T)$.

\end{lemma}

If $\bf X$ is strictly stationary and  Markov of order one,   $\{(X_{2t}, X_{2t-1})\}_{t\in\Z}$ is also strictly stationary and  Markov of order one, with Markov operator 
$$ \Theta:\mathbb{L}^2_0(\rm P_{1,2})\ni f\mapsto  \int   f(x_{3}, x_4) {\rm dP}_{(X_{3}, X_4)| (X_{1}, X_2)}((x_{3}, x_4)| (\cdot , \cdot))\in \mathbb{L}^2_0(P_{1,2}),$$
where $\mathbb{L}^2_0(P_{1,2})$ stands for the space of  $\rm P_{1,2}$-squared-integrable   Borel-measurable functions with zero  $\rm P_{1,2}$ mean. The following assumption is fundamental in our proof technique in order to apply a Hoeffding lemma for Markov sequences (see Theorem~1 in \cite{JianqingFan-etal.JMLR.2021}) and use standard chaining arguments.  
\begin{assumption}\label{Assumption:MArkov-kenel} The operator norm of  $\Theta$ 
is upper-bounded by $\delta\in (0,1)$.  
\end{assumption}

Under this assumption,  
 which is stronger than geometric ergodicity, we obtain rates of convergence for $\widehat{\bf Q}_T$.
 
%
%
\begin{theorem}\label{Theorem:rates}  Let $X_1, \dots, X_n$ be a realization of a strictly stationary Markov process $\bf X$ satisfying Assumptions~\ref{Assumption-On-density} and \ref{Assumption:MArkov-kenel}. Suppose that   $\rm P_1$ is  supported on a compact set $\mathcal{X}$, that the kernel~$K$ satisfies \cref{Assumption-On-kernel}, and that  $ \int v K(v) {\rm d}v=0$.  Then, 
\begin{equation}\label{threebounds}\EE\left[d_{\rm BCL}( \widehat{\rm P}_{ T }, \rm P) \right]\lesssim \gamma({ T },h,d)\coloneqq \begin{cases}
\frac{ 1 }{{ T }^{{1}/{2}} h^{{d}/{2}}}+h^2&{\rm if} \ d<4\vspace{2mm}\\
\frac{ \log({{ T } h^{4}}) }{{ T }^{{1}/{2}} h^{2}}+h^2  &{\rm if} \ d=4 \vspace{2mm}\\
   \frac{1}{{ T }^{{2}/{d}} h^2}+h^2  &{\rm if} \ d>4.
\end{cases}
\end{equation}
Moreover, fixing  $x\in {\rm supp}(\rm P_1)$ with $p_1(x)>0$ and such that  ${\rm P}_{2|1}(\cdot|x)$ is $\mathcal{C}^2$ in ${\rm supp}({\rm P}_{2|1}(\cdot|x))$,  
 as $ T \to \infty$, $h\to 0$, and~$h^d  T \to \infty$,
\begin{enumerate}
    \item for $\mathcal{V}_{ T }$ as in \cref{lemma:wassersteinEstability}, 
    $$ \E\left[\int_{\mathcal{V}_{ T }}\| \widehat{\mathbf{Q}}_T  ( u\vert x)- {\mathbf{Q}}_{2|1}( u\vert x) \|^2 {\rm d}\mu_d^{(k)}(u) \right]\lesssim  \gamma({ T },h,d)+d_{\rm BLC}( \mu_d^{(k)},\mu);$$
    \item    for any compact subset  $\mathcal{K}$ of $\mathbb{B}^d\setminus \{0\}$, 
 $$ \int_\mathcal{K}\| \widehat{\mathbf{Q}}_T  ( u\vert x)- {\mathbf{Q}}_{2|1}( u\vert x) \|^2 {\rm d}\mu_d^{(k)}(u)=\mathcal{O}_{\mathbb{P}}\left(\gamma({ T },h,d)+d_{\rm BLC}( \mu_d^{(k)},\mu)\right).$$ 
\end{enumerate}
\end{theorem}
\begin{remark}
    Analog results for  unconditional transport maps  can be found in \cite{GhosalSenAOS}, \cite{NabarunGhosalSenNips}, and \cite{Manole-et-al.AOS-plug-in} where the sharpest bound is provided. Our proof technique is closer to that of \cite{NabarunGhosalSenNips}, and we therefore  do not expect our bound to be sharp.  
The proof of \cite{Manole-et-al.AOS-plug-in}, however,  is
 not easily adaptable to this context, for two reasons. The first reason is the fact that it 
deals with 
        semidiscrete versions of   empirical optimal transport maps, while we are considering the discrete-discrete one;  the second reason is the singularity of the spherical uniform~$\mu_d$ at zero, which forces us to use localization arguments to avoid the origin. 
\end{remark}
\begin{remark}
    Note that, for 
    $h^2= T ^{- {1}/{d}}$ and $d>4$, we get, in  Theorem~\ref{Theorem:rates} (ii),  the rate\vspace{-1mm} 
     $$ \int_\mathcal{K}\| \widehat{\mathbf{Q}}_T  ( u\vert x)- {\mathbf{Q}}_{2|1}( u\vert x) \|^2 {\rm d}\mu_d(u)=\mathcal{O}_{\mathbb{P}}\left({{ T }^{- {1}/{d}}}+d_{\rm BLC}( \mu_d^{(k)},\mu)\right).\vspace{-1mm}$$
    This rate is not as good as in \cite{NabarunGhosalSenNips} for the unconditional empirical transport map estimator, which is of order $ T ^{-{2}/{d}} $. This, however, is to be expected, as the estimation of  conditional quantiles involves two nonparametric methods---the estimation of the conditional measure,  then the estimation of the transport map---both of which are affected by the curse of dimensionality.\vspace{-1mm}
 \end{remark}

\section{Numerical Applications}
In this section, we assess the empirical performance of our proposed method in 
 simulated examples  (Section~\ref{simulated_examples}) and real data (Section~\ref{realsec}). 
The numerical results show that our method captures conditional heteroskedasticity and    nonconvex quantile contours in highly nonlinear autoregressive models.

\subsection{Simulated examples}\label{simulated_examples}
We simulated two examples (Cases~1 and~2)  of highly nonlinear $d$-dimensional asymptotically stationary\footnote{See Appendix~B.} vector autoregressive  series of order one with conditional heteroskedasticity   and  (Case~3) one example of a nonlinear and nonstationary series with highly nonconvex quantile contours;  simulated series lengths $T$ are up to $80,000$, after a warming-up period of $T_0 \approx~\!10000$  observations. For the sake of simplicity, we do not reflect that warming-up period $T_0$ in the notation, though, and write $X_{t}$ for $X_{T_0+t}$.  To allow for visualization, we focus on $d=2$ and $p=1$, but the method applies to any  $d$ and~$p$. 

%
%
For each simulated time series, two tasks were performed. \smallskip

\begin{compactenum}
\item[(1)]First, we kept track of the empirical conditional 
 quantile functions~$\widehat{\mathbf{Q}}_{X_{t+1} \vert X_t=x_t}$ ($x_t$~the realized value of $X_t$) along $t$ and illustrate their variation over time by plotting the corresponding quantile contours at time points $1\leq t_1< t_2< \ldots < t_M \leq  T-1$. These conditional quantiles are estimated based on~\eqref{NW-estimator}--\eqref{cond_quantile_estimator},  
along the following steps. 
\begin{compactenum}
    \item[(i)] Step 1: compute 
      \vspace{-3mm} 
          \begin{equation*}
            \widehat{\rm P}_{X_{t_m+1}|X_{t_m}=x_{t_m}}\coloneqq  \sum_{i=1}^{T-1} w_{i+1}(x_{t_m})  \delta_{x_{i+1}} \quad\!\!\text{where}\!\!\quad  w_{i+1}(x_{t_m})\coloneqq \frac{  K\left(\frac{x_i-x_{t_m}}{h} \right) }{\sum_{j=1}^{T-1}  K\left(\frac{x_j-x_{t_m}}{h} \right) }, 
          \end{equation*}
with a truncated Gaussian kernel  $K$  supported on the set of $k_Sk_R$ nearest neighbors~$x$ 
of~$x_{t_m}$, each 
of them  being assigned a weight proportional to~$e^{-\|x-x_{t_m}\|^2/h^2}$. 
       \item[(ii)] Step 2: compute the empirical optimal transport plan $\hat{\bpi}$ from the pre-determined uniform spherical grid $\mu_d^{(k)}$ to $\widehat{\rm P}_{X_{t_m+1}|X_{t_m}=x_{t_m}}$.
       \item[(iii)] Step 3: evaluate the target quantile contours (or regions) by cyclically monotone interpolation.
\end{compactenum}
In Case~1, moreover, the theoretical quantiles can be computed analytically, allowing us to compare empirical conditional quantiles to their theoretical counterparts for various values of $T$. The results are shown in Figures~\ref{case1_sample_sizes}, \ref{case2_sample_sizes}, and \ref{case3_sample_sizes}, respectively;  the time axes in these figures, and also in Figures~3,  7,  and 12,  have been rescaled to $t' = t/1250$ in Cases 1 and 3, to $t' = t/2500$ in Case 2.  
. 

\item[(2)] Second, for each series, we estimated the empirical unconditional quantile contours of its asymptotically stationary distribution (see Appendix~B for asymptotic stationarity). This estimation is based on a simulation of length $T^\prime$ (after adequate warming-up), independent of the simulation considered in (1); let $T^\prime$ be sufficiently large and, for convenience, let it be even. The computation goes along the same lines as in (i)--(iii) above, except that the empirical conditional distribution $\widehat{\rm P}_{X_{t_m+1}|X_{t_m}=x_{t_m}}$ is replaced by an empirical stationary distribution of the form (summing over even values of $t$ yields independent summands)\vspace{-3.5mm}

$\,$\vspace{-2mm}
 \begin{equation}\label{uncond}
    \widehat{\rm P}_{\bX} \coloneqq  \frac{2}{T^\prime} \sum_{k=1}^{T^\prime/2} \delta_{x_{2k}}.\vspace{-2mm}
\end{equation}
 Parallel to this, we also estimate, for a set~$x^1, \ldots, x^M$ of points chosen on these empiri\-cal unconditional quantile contours, the one-step-ahead predictive quantile functions~$\widehat{\mathbf{Q}}_{X_{t+1} \vert X_t=x^m}$ for various current values $x^m, \ m=1, \ldots, M$  (with~$M=~\!8$). The results are shown in  Figures~\ref{case1_predict}, \ref{case2_predict}, and \ref{case3_predict}, respectively, and illustrate  the dependence of  one-step-ahead predictive quantiles on  current quantile values---a dependence which is the essence of quantile autoregression.

\end{compactenum}\smallskip

The three data-generating processes considered in the simulations   are as follows.\medskip

\noindent
\textbf{Case 1.} The data-generating equation is \vspace{-2mm}
\begin{equation}\label{case1}
         X_{t+1} = \begin{bmatrix}
        \frac{1}{3}{(X_t^1 + X_t^2)}\vspace{1mm}  \\           
        \frac{1}{2}\sqrt{ {\|X_t\|^2+5} }
        \end{bmatrix} +
        \sin\Big( \frac{\pi}{10}\|X_t\|\Big) {\varepsilon}_{t+1}\vspace{-1mm}
        \end{equation}         
        with $
       {\varepsilon}_{t+1} \sim N( 0 , {\rm I})$, ${\varepsilon}_{t+1}\independent X_s$  for all $s\leq t$, and $X_0 \sim N( 0 , {\rm I})$.\vspace{-2mm}
       \pagenumbering{gobble}
       
   \begin{figure}[H]
    \centering
    \includegraphics[width=0.8\linewidth, height= 0.3\linewidth]{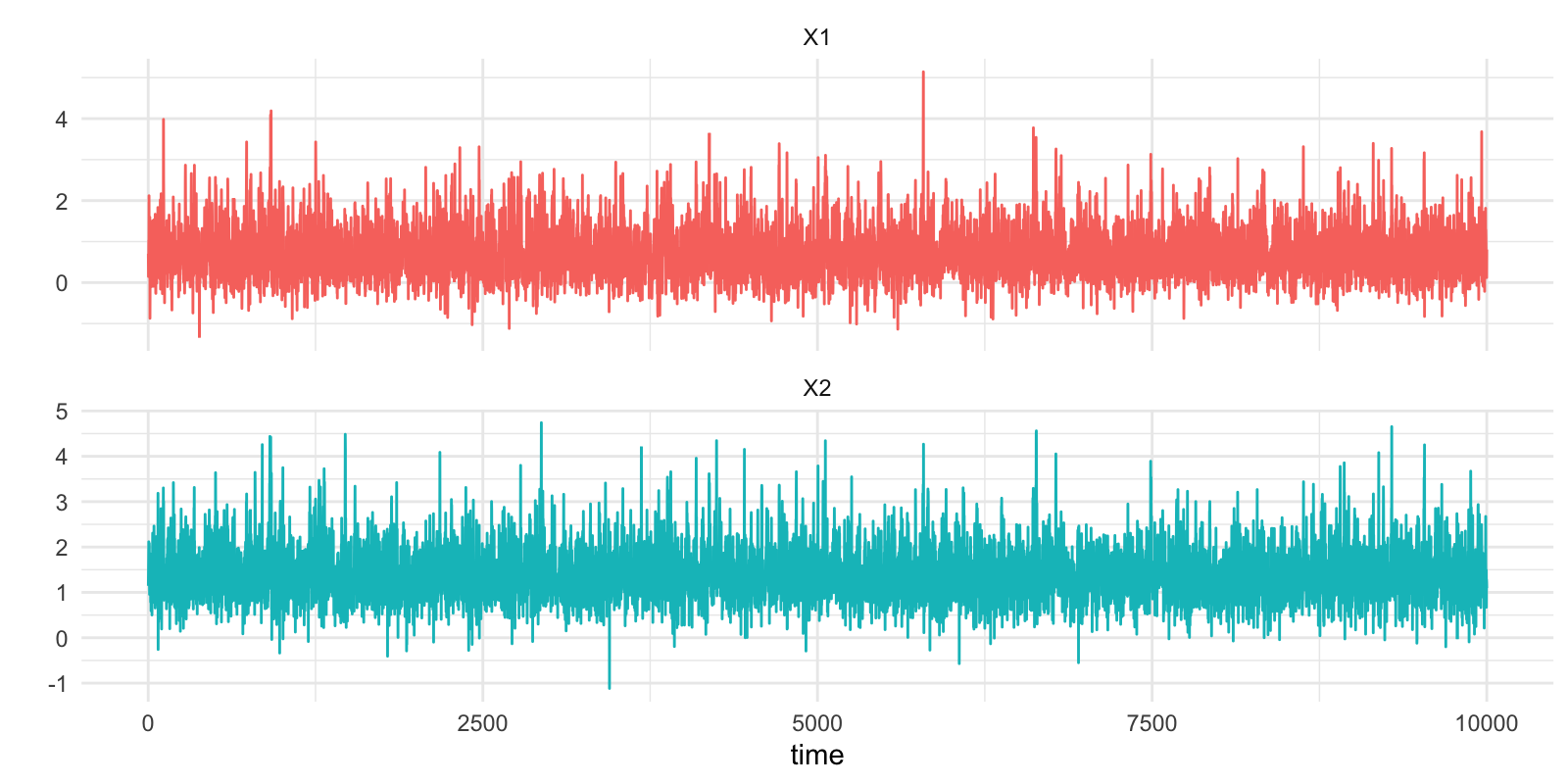}\vspace{-3mm}
    \caption{(Case 1) Simulated trajectories of the first (red)  and second (blue) components of~$X$ for~$T=10,000$}.
    \label{case1_ts}\vspace{-12mm}
\end{figure}
\begin{figure}[t!]
    \centering
    \includegraphics[width=0.9\linewidth]{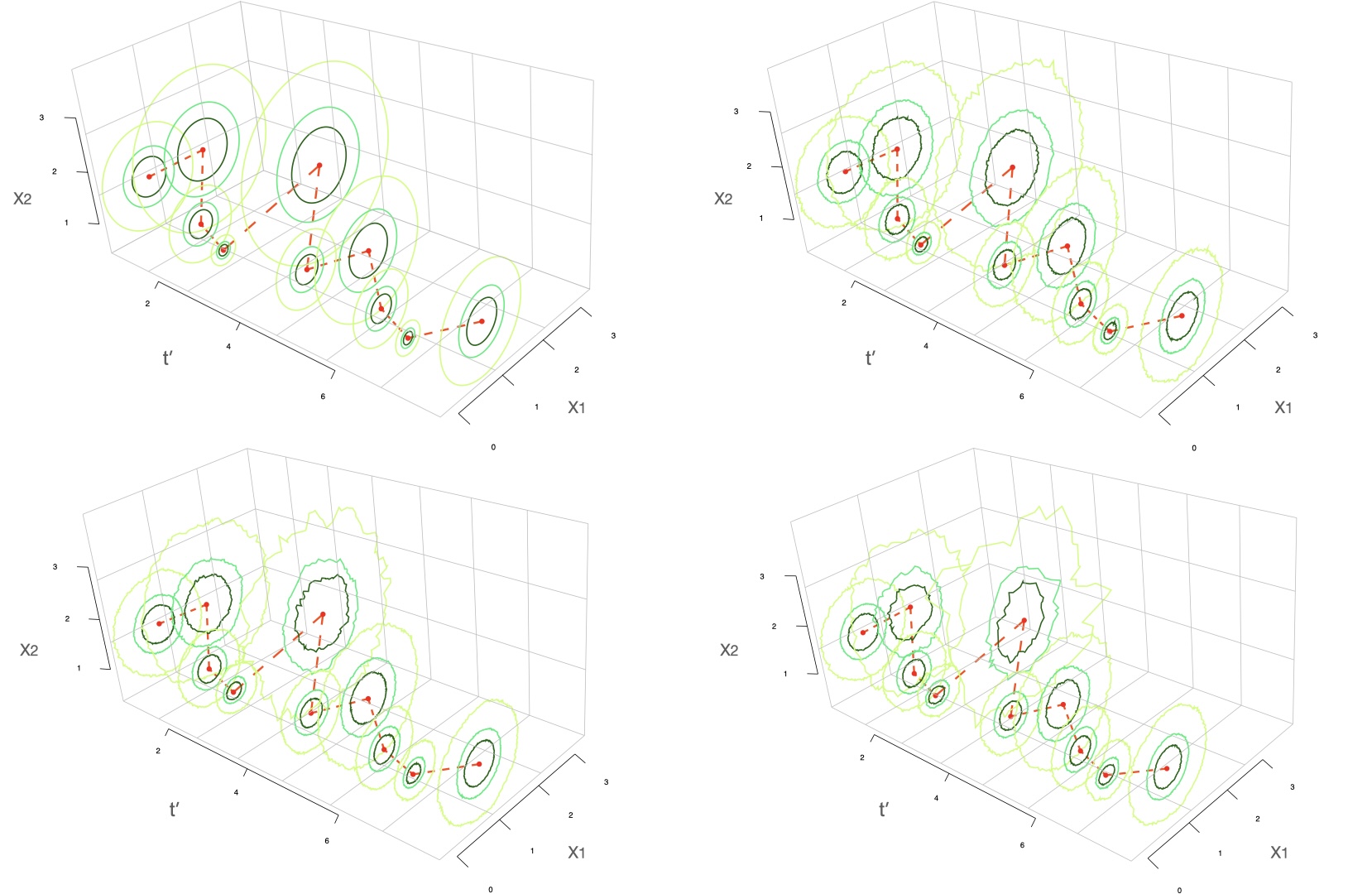}\vspace{-1mm}
    \caption{(Case 1) The empirical conditional center-outward quantile contours of orders~$\tau=~\!0.2$ (dark green), 0.4 (green), 0.8 (light olive), and conditional median (red) at randomly selected time points with different sample sizes $T=800,000$ (upper right panel),~$T=80,000$ (lower left panel), and $T=40,000$ (lower right panel). The upper left panel provides the corresponding theoretical conditional  contours and medians computed via equation \eqref{case1_theoretic_quantiles}. 
Kernel bandwidths were chosen as $h=0.5 \times \text{average pairwise distance}$. \vspace{-1mm}}
    \label{case1_sample_sizes}
\end{figure}
\begin{figure}[H]
    \centering
    \includegraphics[width=0.9\linewidth]{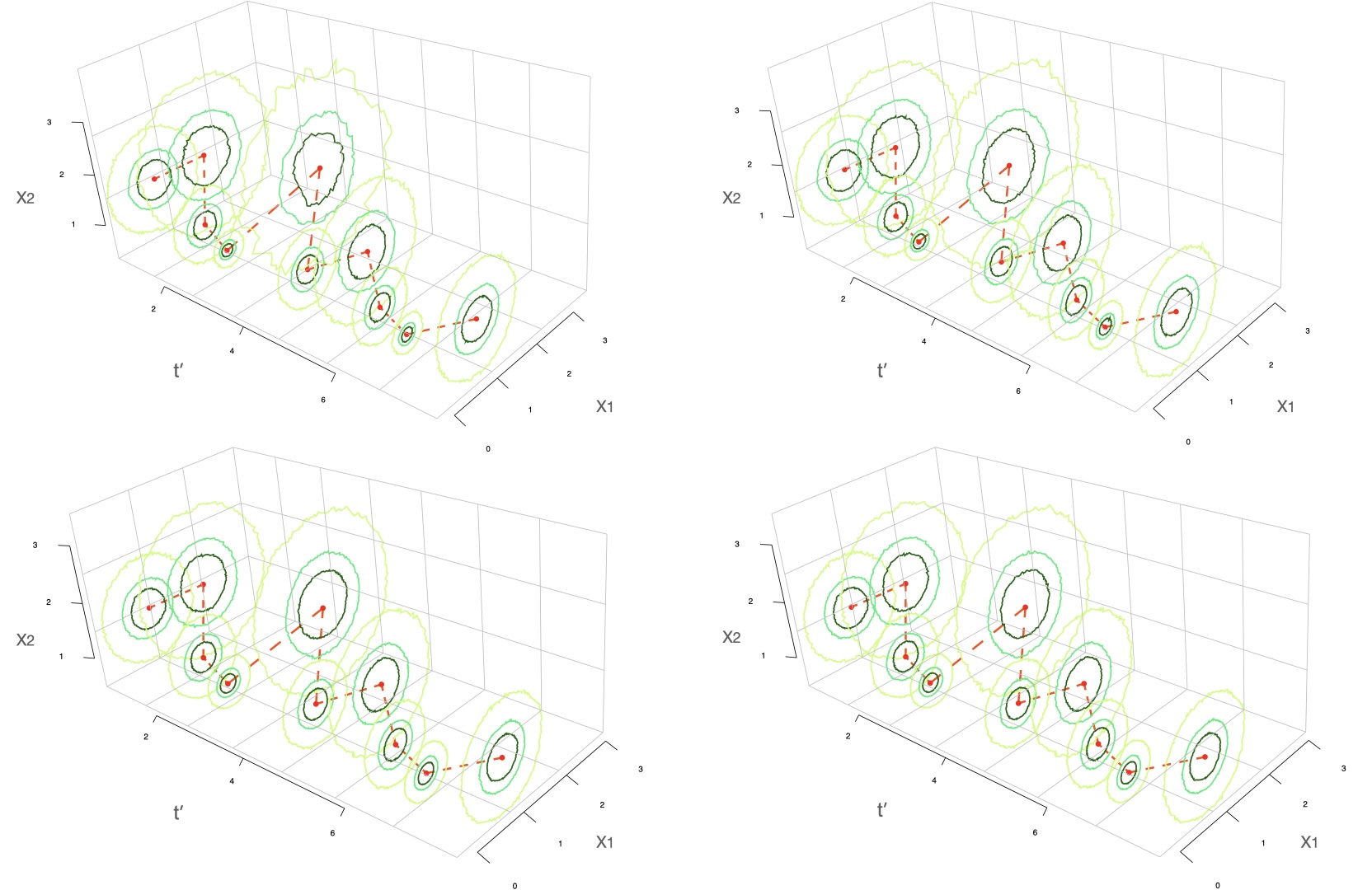}\vspace{-1mm}
    \caption{(Case 1)  Estimated conditional center-outward quantile contours and medians for fixed sample size $T=800,000$, based on  kernel bandwidths  $h=\ell\times \text{average pairwise distance}$, with $\ell = 0.2$ (upper left panel),   $\ell =0.5$ (upper right panel),  $\ell =1.2$  (lower left panel), and~$\ell =3.0$ (lower right panel).
    }
    \label{case1_Kwidths}
\end{figure}
$\,$\vspace{-22mm}

\pagenumbering{arabic} \setcounter{page}{14}
\begin{figure}[b!]
    \centering
    \includegraphics[width=0.8\linewidth, height=0.8\linewidth]{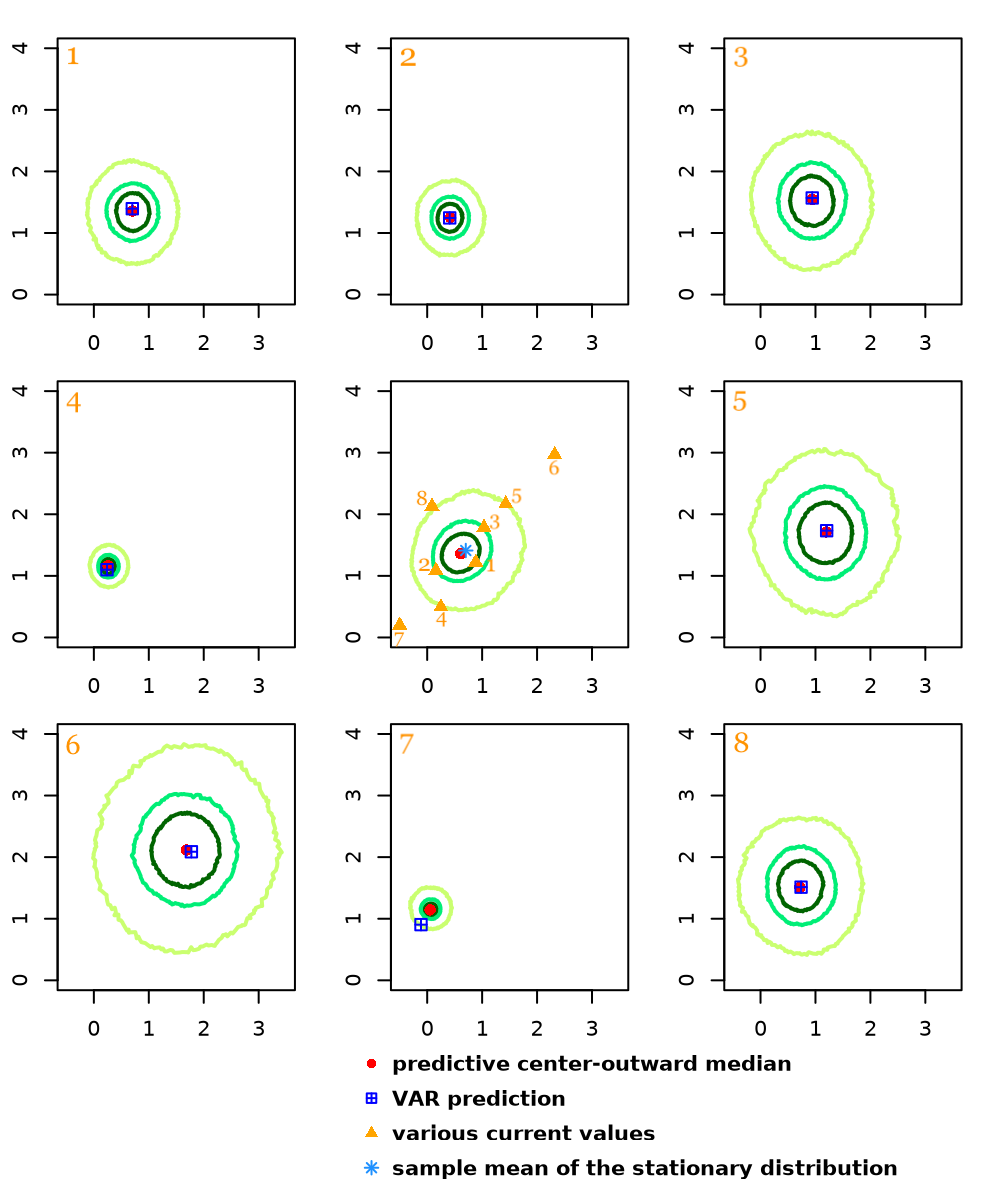}\vspace{-3mm}
    \caption{ (Case 1) 
 The estimated one-step-ahead conditional quantile contours and medians at selected current values.  The central panel shows the estimated center-outward quantiles of orders  $\tau=0.2$ (dark green), 0.4 (green), 0.8 (light olive), the center-outward median (red), and the sample mean (light blue) of the (unconditional) stationary distribution, and the eight current values (orange) at which quantile prediction is implemented in the surrounding panels. The surrounding panels show the one-step predictive center-outward quantile contours of order $\tau=0.2$ (dark green), 0.4 (green), 0.8 (light olive), the conditional center-outward median (red), and the conventional VAR(1) one-step-ahead mean prediction (blue) at these eight particular current values. 
  }
    \label{case1_predict}
\end{figure} \vspace{-5mm}

Figure \ref{case1_ts} shows the marginal trajectories generated by \eqref{case1}.  Visual inspection does not reveal any trends, but the two series 
exhibit conditional heteroskedasticity. Since the  target distribution here is spherical, the optimal transport map from the spherical uniform (Step 2 above) admits an analytical form, and the theoretical center-outward quantile contours can be calculated explicitly. 
More precisely, the theoretical conditional (on $X_t$) quantile region of order $\tau$ at time ${t}+1$ of the process generated~by 
   $ X_{t+1} = g(X_t) + v(X_t){\varepsilon}$ with ${\varepsilon} \sim N( 0 , {\rm I})
$ 
has the explicit form\vspace{-1mm}
\begin{equation}\label{case1_theoretic_quantiles}
    \Big\{ x: \big(v(X_{t})\big)^{-2} \big(\bx - g(X_{t})\big)^\top  \big(\bx - g(X_{t})\big) \leq \chi_{d,\tau}^2 \Big\}.\vspace{-0mm}
\end{equation}
This is how we compute the theoretical 
 conditional quantile contours in the upper left panel of Figure \ref{case1_sample_sizes}.  Empirical conditional contours  can then be compared, for different   $T$ values,  to the theoretical ones. Note that interpolating between the empirical~condi\-tional medians would make no sense here (and in Figures~3, 6, 7, 11, 12, 14--17), 
as these medians are indexed by the actual values $x_t$ of $X_t$, with a highly discontinuous 
mapping~$t\!\mapsto~\!\!x_t$; we nevertheless draw a dashed red line connecting these medians to help visualize their ordering over time.

 In Figure \ref{case1_Kwidths}, we also explore the impact of  the choice of the kernel bandwidth $h$ on the accuracy of the estimation by implementing our method across a grid of different kernel bandwidths. Apparently, this impact strongly depends on the pairwise distance between sample point values. For a set of widely spread (concentrated) sample points, $h$ should be larger (smaller). Therefore, we explored a grid of $h$ values equal to $\ell$  times the average pairwise distance between  sample points, with $\ell =  0.1$, $0.2, \ldots, 3.0$. With estimation accuracy  measured by the MSE between the estimated and theoretical quantile contours, we conclude that  estimation accuracy here is best for $\ell\in [0.5,0.6]$. 


Next, as explained in (2) above, we 
explore the dependence of the estimated one-step-ahead predictive quantiles 
on the  {\it unconditional} quantile level of the current value. 
To do so, we estimate two types of empirical center-outward quantile functions: the unconditional ones, characterizing the empirical (unconditional) stationary distribution \eqref{uncond}; the conditional ones or one-step-ahead predictive quantile functions given current value. Then, we provide the empirical one-step-ahead predictive contours, respectively, when $X_t = x^m$, $m=1,\ldots,8$  
for    $x^m$ with various quantile levels in  the stationary distribution.  The results are shown in Figure~\ref{case1_predict}, where the central panel displays the empirical center-outward quantile contours of the stationary distribution and the eight va\-lues of~$x^m$,  while the other panels show the empirical contours of $X_{t+1}$ conditional on~$X_t=~\!x^m$. Inspection of this figure illustrates the impact of the current quantile value on the one-step-ahead prediction of quantiles, accounting for huge variations in the predicted  location, scale, and shape. 
%
%
We also provide the conventional VAR(1) one-step-ahead prediction of  conditional means (computed via the~R~package ``vars''). 

\medskip\vspace{1mm}

\noindent
\textbf{Case 2.}  The data-generating equation is 
\begin{equation}\label{case2}
    X_{t+1} \!=\! \begin{bmatrix}
        \tanh\Big(\frac{1}{2}\big(X_t^1 + X_t^2\big)\Big) - \frac{1}{2}\\         
        \cos\Big(\frac{\pi}{10} f\big(X_t^1 + X_t^2\big) \Big)
        \end{bmatrix} +
        \frac{\|X_t\|}{2} \, {\varepsilon}, \quad\!\! \tanh(x)= \frac{e^x-e^{-x}}{e^x+e^{-x}}, \quad\!\! f(x) = \frac{x}{1+|x|} 
        \end{equation}
       where $
       {\varepsilon} \sim \text{Unif}[-1,1] \times \text{Unif}[-1,1], \quad X_0 \sim \text{Unif}[-1,1] \times \text{Unif}[-1,1]$. \vspace{2mm}

    \begin{figure}[b!]
    \centering
    \includegraphics[width=0.8\linewidth, height= 0.3\linewidth]{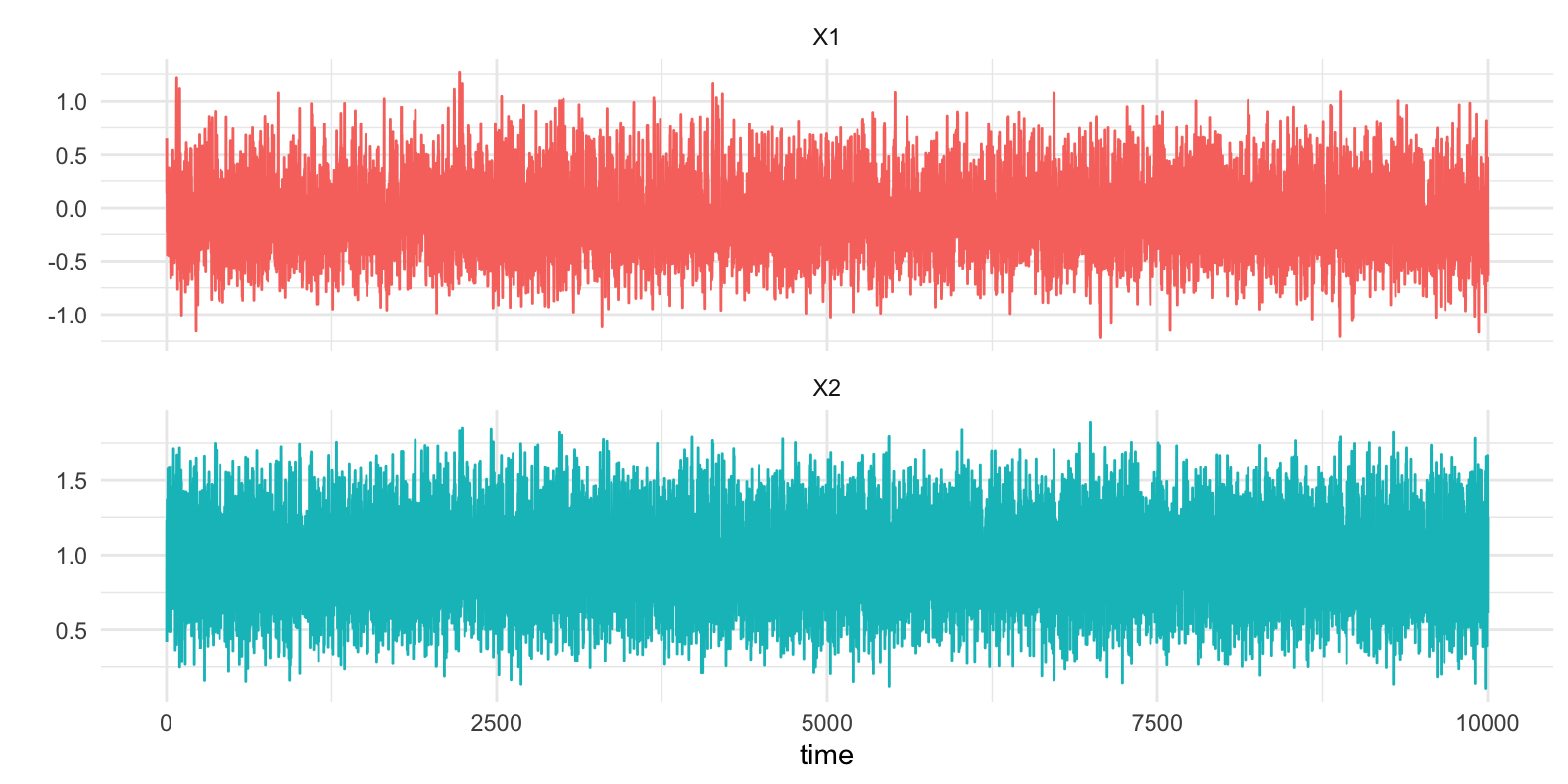}\vspace{-3mm}
    \caption{(Case 2) Simulated trajectories of the first (red)  and second (blue) components of~$X$ for~$T=10,000$.\vspace{1mm}}
    \label{case2_ts}
\end{figure}

\begin{figure}[t!]
    \centering
    \includegraphics[width=0.9\linewidth]{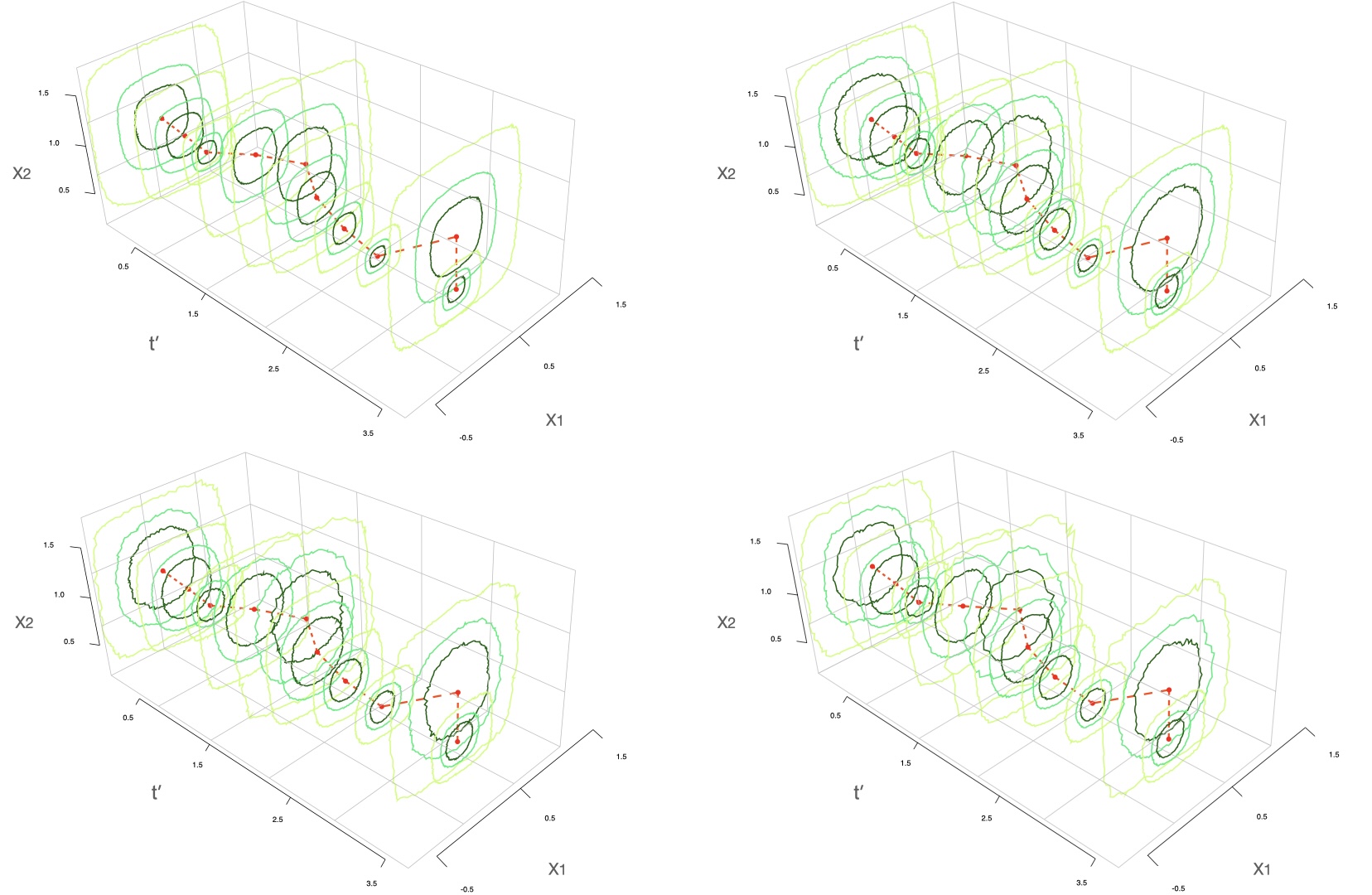}\vspace{-1mm}
    \caption{(Case 2) The conditional center-outward quantile contours of orders $\tau=0.2$ (dark green), $0.4$ (green), and $0.8$ (light olive), along with the conditional median (red) at randomly selected time points,  with  sample sizes $T=800,000$ (upper right panel), $T=~\!80,000$ (lower left panel), and $T=40,000$ (lower right panel).  The chosen  kernel bandwidths\linebreak are 
      $h=0.4 \times \text{average pairwise distance}$. \vspace{-3mm}}
    \label{case2_sample_sizes}
\end{figure}
\begin{figure}[H]
    \centering
    \includegraphics[width=0.9\linewidth]{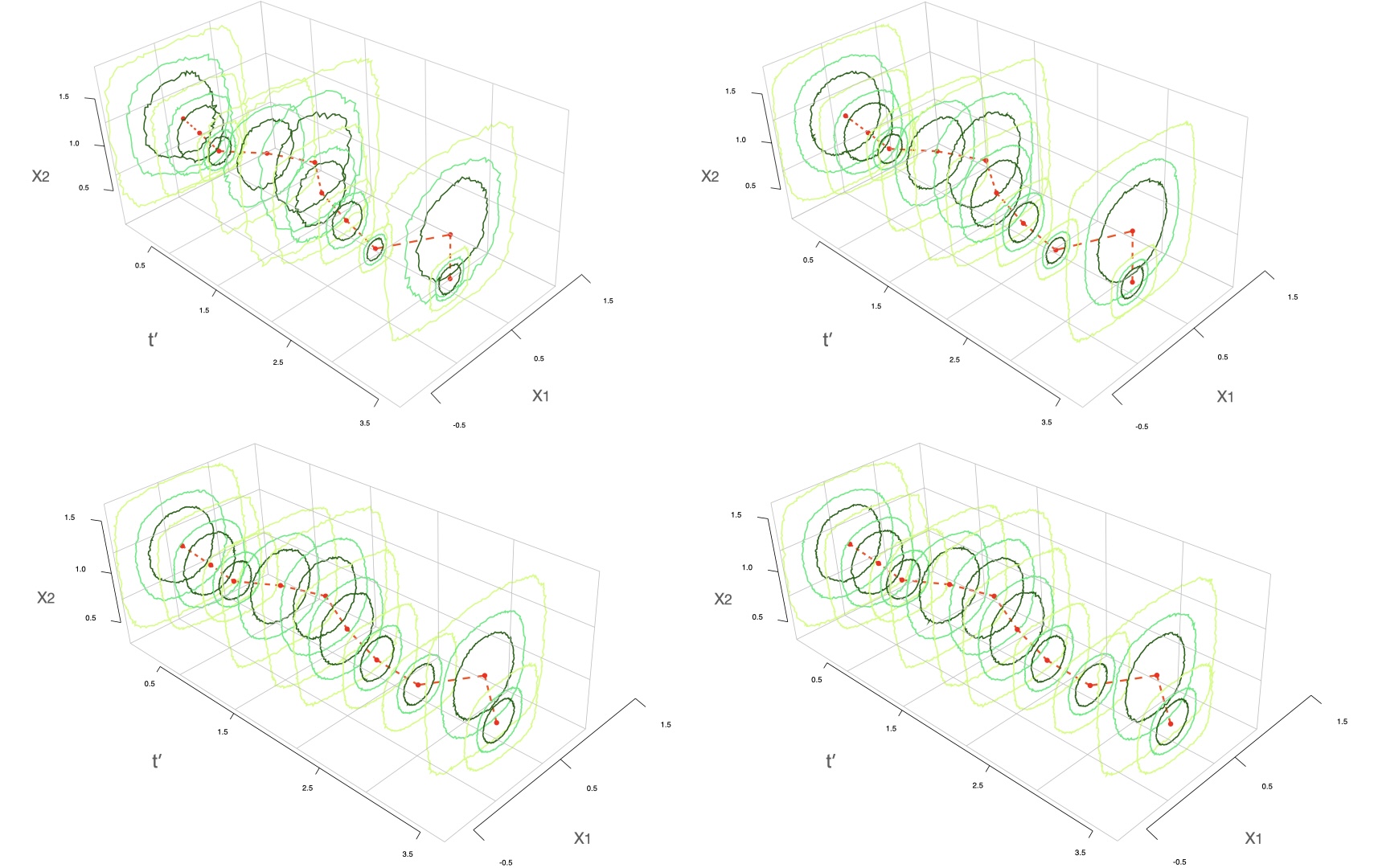}\vspace{-1mm}
    \caption{(Case 2)  Estimated conditional center-outward quantile contours and medians for fixed sample size $T=800,000$, based on  kernel bandwidths  $h=\ell\times \text{average pairwise distance}$, with $\ell = 0.2$ (upper left panel),   $\ell =0.4$ (upper right panel),  $\ell =1.2$  (lower left panel), and~$\ell =3.0$ (lower right panel).}\vspace{-11mm}
    \label{case2_Kwidths}
\end{figure}
  \pagenumbering{gobble}
\begin{figure}[b!]
    \centering
    \includegraphics[width=0.8\linewidth, height=0.8\linewidth]{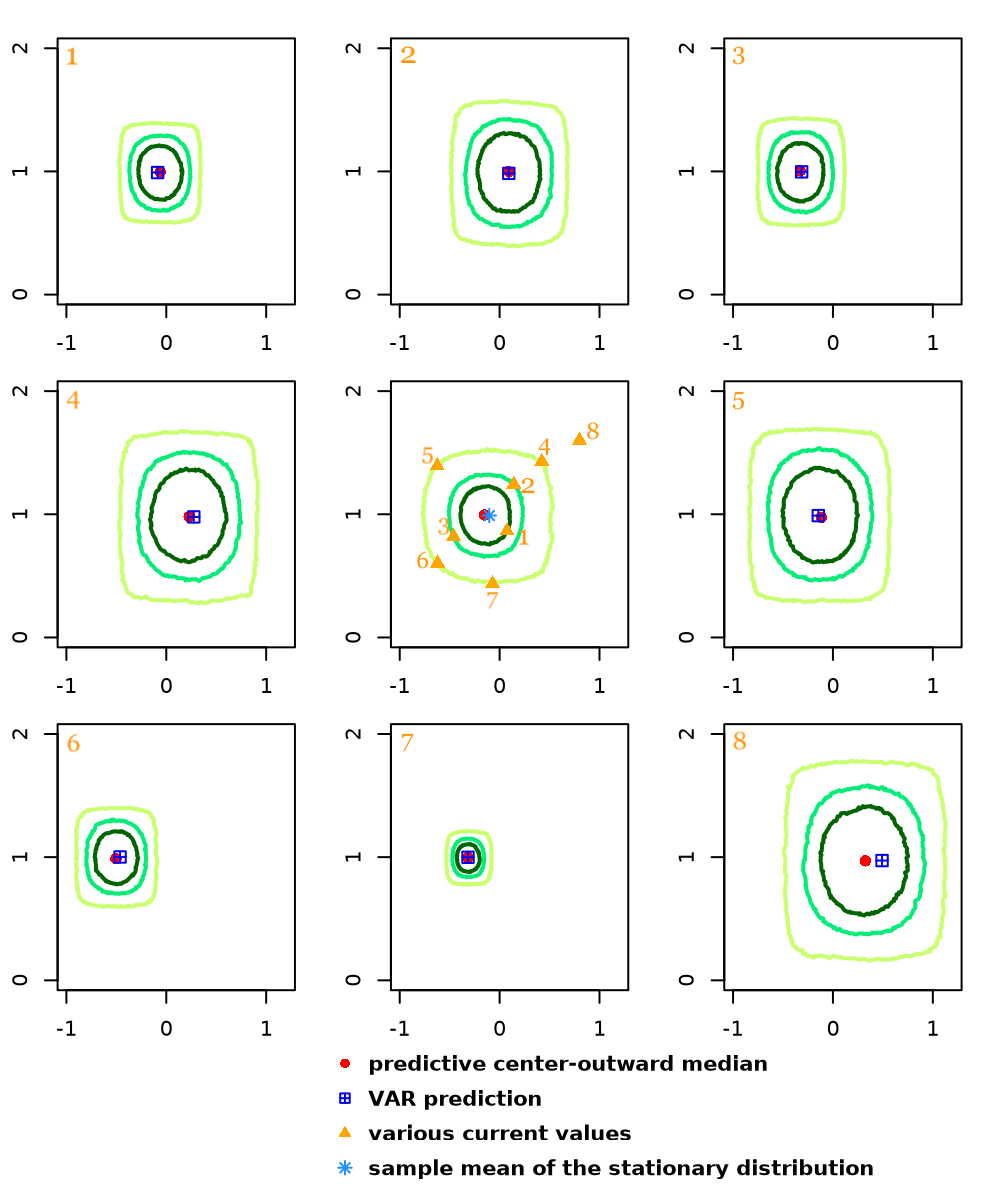}\vspace{-1mm}
    \caption{
(Case 2) The estimated one-step-ahead conditional quantile contours and medians at selected current values.  The central panel shows the estimated center-outward quantiles of orders $\tau=0.2$ (dark green), 0.4 (green), 0.8 (light olive), the center-outward median (red), and the sample mean (light blue) of the (unconditional) stationary distribution, and the eight current values (orange) at which quantile prediction is implemented in the surrounding panels. The surrounding panels show the one-step predictive center-outward quantile contours of order $\tau=0.2$ (dark green), 0.4 (green), 0.8 (light olive), the conditional center-outward median (red), and the conventional VAR(1) one-step-ahead mean prediction (blue) at these eight particular current values.
    }
    \label{case2_predict}
\end{figure}
\newpage 
$\,$\vspace{-10mm}
\pagenumbering{arabic} \setcounter{page}{17}

Figure \ref{case2_ts} depicts the marginal trajectories generated by \eqref{case2}, which look globally stationary but exhibit potential conditional heteroscedasticity.  
Figure \ref{case2_sample_sizes} shows how estimation accuracy improves with increasing $T$. The innovation here is not spherical, so the transport map in (ii) 
 has no explicit form (as in Case~1); instead, the contours in the upper left panel are obtained via  simulation---at every $x_t$,  a large sample of $X_{t+1}$ values is generated from the actual model \eqref{case2}, from which the conditional center-outward quantile function is estimated.

Figure \ref{case2_Kwidths} visualizes the estimated conditional quantiles for various kernel band\-widths of the form~$h=\ell \times \text{average pairwise distance}$, $\ell =0.2$, 0.4, 1.2, and 3.0. The best results (in terms of squared deviations from the quantiles in the upper left panel) are obtained for~$\ell =~\!0.4$. 

The one-step-ahead predictions at a variety of current values are shown in Figure~\ref{case2_predict}: the predictive center-outward quantiles and medians wildly vary with the current values: compare, for instance, current values 7 and 8. Some predicted center-outward medians also are closer to the true median/mean (computed based on \eqref{case2}) than the corresponding VAR(1)-prediction, because the model in \eqref{case2} is highly nonlinear, a feature   traditional  VARs cannot account for.

\bigskip    
\noindent
\textbf{Case 3.} The data-generating equation is 
\begin{equation}\label{case3}
    X_{t+1} = \begin{bmatrix}
         \dfrac{\log(\|X_t\|+2)}{\|X_t\|+2} \vspace{1mm}
         \\
         \dfrac{\|X_t\|}{\|X_t\|+ \sqrt{2}}
        \end{bmatrix} +
        \sqrt{\|X_t\|+1}\, 
        {\rm R}(t)\,
       {\varepsilon}_t, \quad X_0 \sim N( 0 , {\rm I})
\end{equation}  
where ${\rm R}(t)$ is the rotation matrix of angle $\pi t/5000$ and the $ {\varepsilon}_t$'s are i.i.d.\ with distribution 
\begin{equation*}
\label{gaussian_mixture}
    \frac{1}{4}N\big( 0 , \frac{1}{25}{\rm I}\big) + \frac{1}{4}N\big((0.866, -0.5)^\top, \frac{1}{25}{\rm I}\big) + 
    \frac{1}{4}N\big((-0.866, -0.5)^\top, \frac{1}{25}{\rm I}\big) + \frac{1}{4}N\big((0,1)^\top, \frac{1}{25}{\rm I}\big)
\end{equation*}
(a  clover-shaped mixture of four independent Gaussians; see Figure \ref{nonconvex} for a scatterplot). 
\begin{figure}[h!]
    \centering
    \includegraphics[width=0.35\linewidth, height=0.3\linewidth]{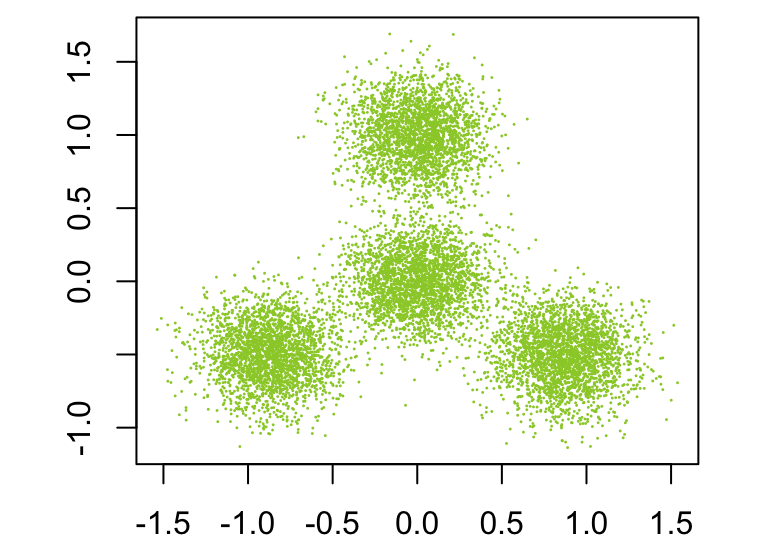}
    \caption{(Case 3) A sample from the mixture distribution of $\varepsilon _t$.
    }\vspace{-4mm}
    \label{nonconvex}
\end{figure}
\begin{figure}[h!]
    \centering
    \includegraphics[width=0.8\linewidth, height= 0.3\linewidth]{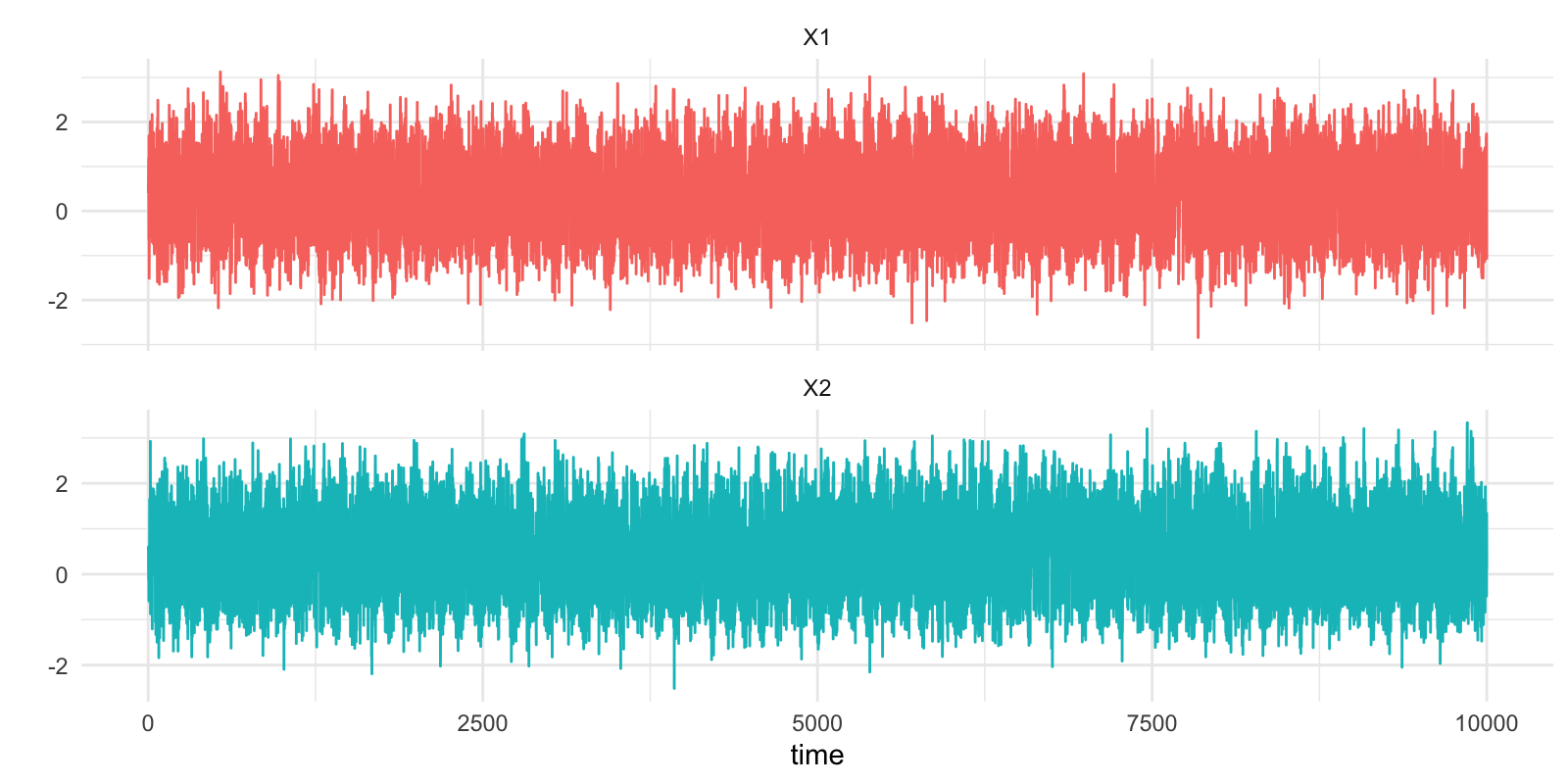}
    \caption{(Case 3) Simulated trajectories of the first (red)  and second (blue) components of~$X$ for~$T=10,000$. }\vspace{-4mm}
    \label{case3_ts}
\end{figure}

\newpage

$\,$\vspace{-23mm}

\begin{figure}[h!]
    \centering
    \includegraphics[width=0.9\linewidth]{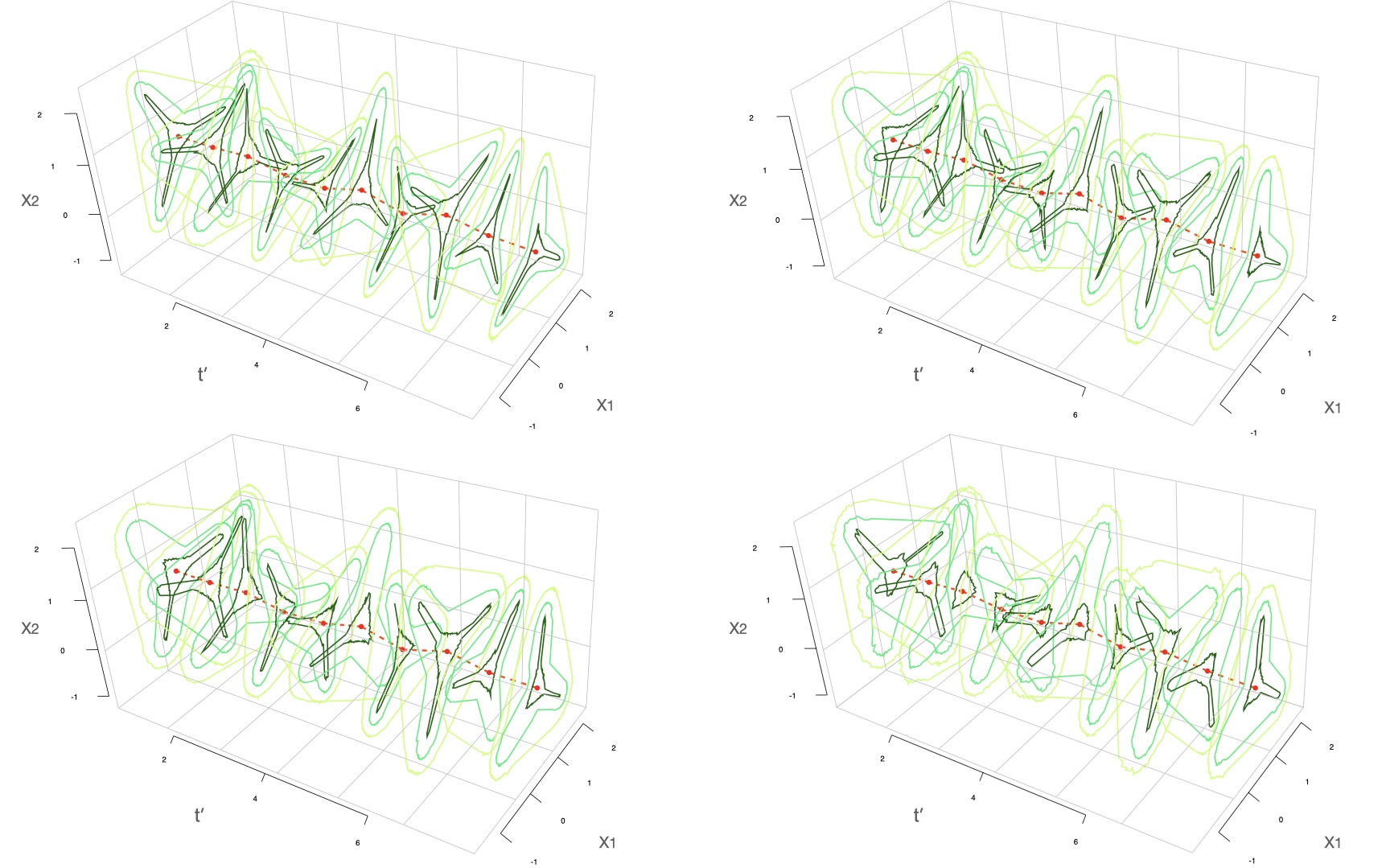}\vspace{-1mm}
    \caption{(Case 3) The empirical conditional center-outward quantile contours of orders~$\tau=~\!0.2$ (dark green), 0.4 (green), 0.8 (light olive), and conditional median (red) at randomly selected time points with different sample sizes $T=2,000,000$ (upper left panel),~$T=800,000$  (upper right panel), $T=80,000$ (lower left panel), and $T=40,000$ (lower left panel). Kernel bandwidths were chosen as $h=0.1 \times \text{average pairwise distance}$.} \vspace{-2mm}
        \label{case3_sample_sizes}
\end{figure}
\begin{figure}[H]
    \centering
    \includegraphics[width=0.9\linewidth]{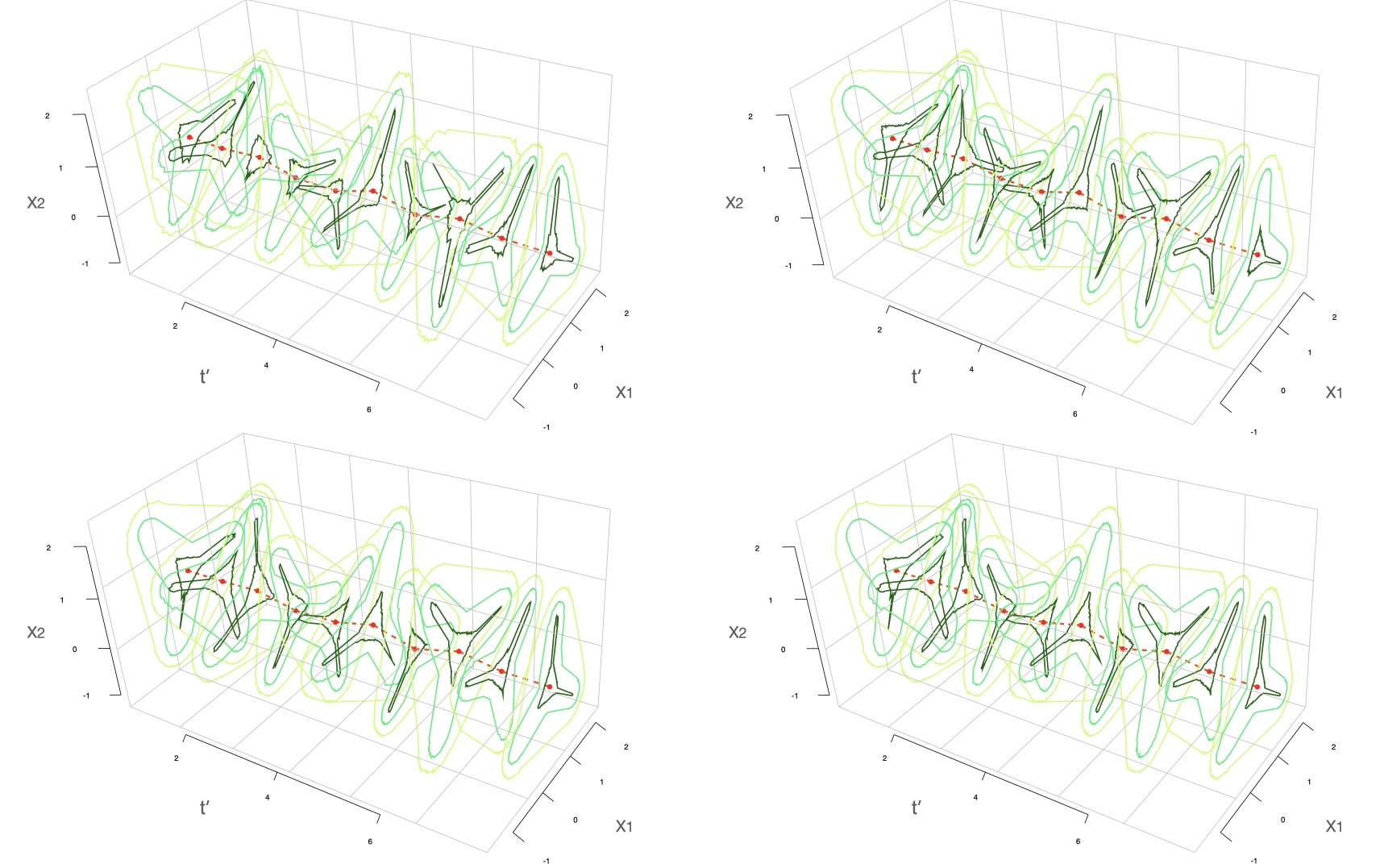}\vspace{-1mm}
    \caption{(Case 3) Estimated conditional center-outward quantile contours and medians for fixed sample size $T=800,000$, based on  kernel bandwidths  $h=\ell\times \text{average pairwise distance}$, with $\ell = 0.03$ (upper left panel),   $\ell =0.1$ (upper right panel),  $\ell =1.0$  (lower left panel), and~$\ell =2.0$ (lower right panel).}\vspace{-3mm}
    \label{case3_Kwidths}
\end{figure}

\newpage
$\,$\vspace{-8mm}

Figure \ref{case3_ts} shows the marginal trajectories generated by \eqref{case3} which, misleadingly,  look globally stationary---marginal stationarity does not imply    joint stationarity, though.  

Case~3, however, differs significantly from Cases 1 and 2.
First, as shown in Figure~\ref{nonconvex}, the distribution of ${\varepsilon}$ has a highly nonconvex shape. Second, due to the $t$-dependent  rotation ${\rm R}(t)$,  the distribution of $X_t$ is not  even asymptotically stationary (it is, however, asymptotically stationary 
 for~${\rm R}(t)={\bf I}$: see Appendix~B. The  conditions for consistency, thus, are violated. Ignoring this fact, we ran our method as in Cases~1 and~2 to obtain Figures~\ref{case3_sample_sizes} and~\ref{case3_Kwidths}.


\begin{figure}[b!]
    \centering
    \includegraphics[width=0.8\linewidth, height=0.8\linewidth]{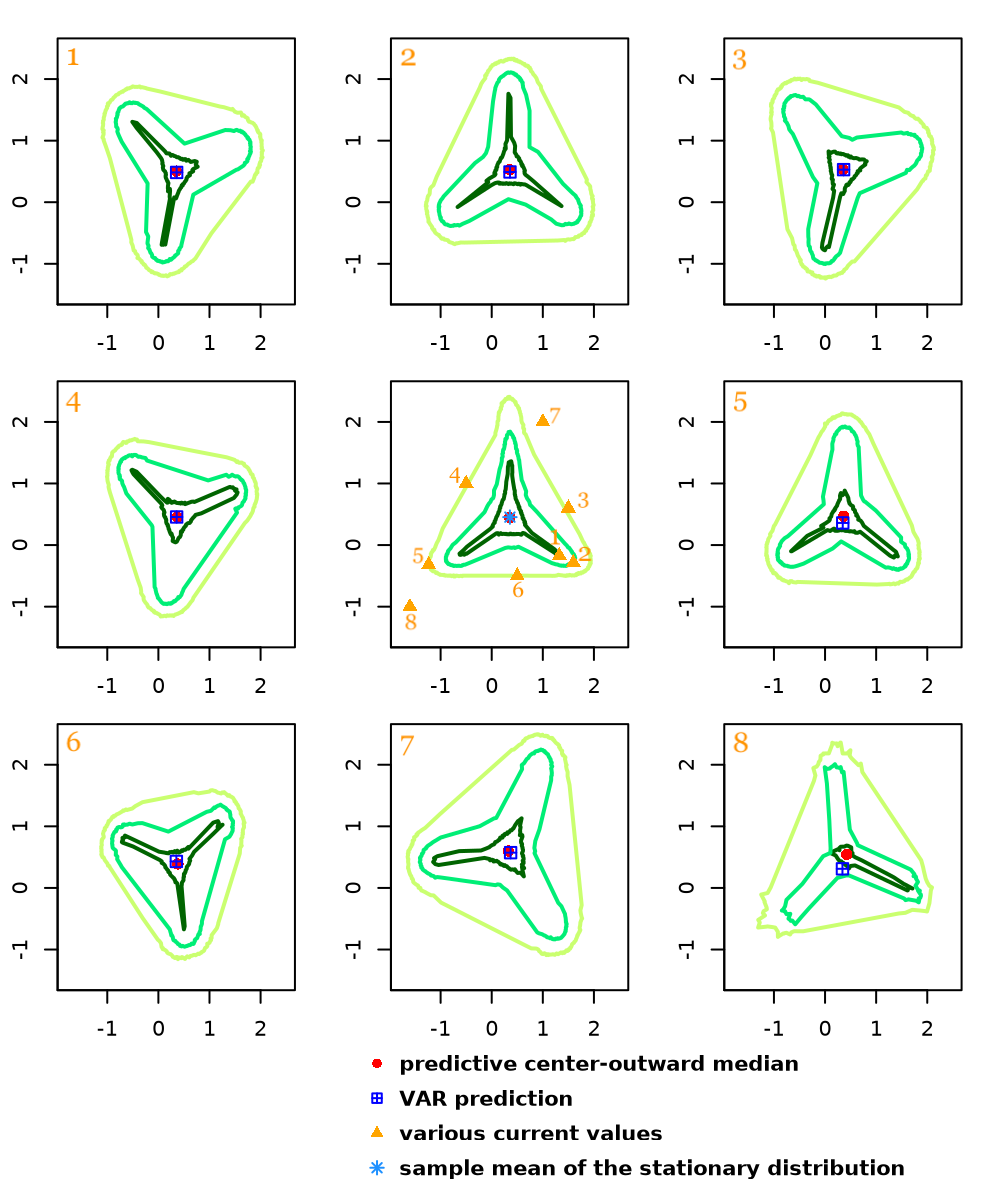}
    \caption{
    (Case 3) The estimated one-step-ahead conditional quantile contours and medians at selected current values $x_T$ of $X_T$, $T=20,000$.    The central panel shows the estimated unconditional center-outward quantiles of orders $\tau=0.2$ (dark green), 0.4 (green), 0.8 (light olive), the center-outward median (red), and the sample mean (light blue) at time $T$ 
     and (orange) the eight current values $x_T$  at which one-step-ahead quantile prediction is implemented in the surrounding panels. The surrounding panels show the one-step predictive center-outward quantile contours of orders $\tau=0.2$ (dark green), 0.4 (green), 0.8 (light olive), the conditional center-outward median (red), and the conventional VAR(1) one-step-ahead mean prediction (blue) at these eight particular current values.}
    \label{case3_predict}
\end{figure}
The accuracy of the estimation is investigated in Figure \ref{case3_sample_sizes}. The transport  map in Step~2 does not have an explicit form; therefore, as in Case~2, we approximate the theoretical conditional quantiles and medians at each selected~$t$ via   a  very large ($T_0=2,000,000$) simulated  sample of~$X_{t+1}$ values based on the actual data-generating equation~\eqref{case3}.  The result of this simulation, shown in the upper left panel,  can be used as a benchmark. Although the conditions for consistency are not met, the quality of the approximation in the three other panels of Figure \ref{case3_sample_sizes}  (with kernel bandwidth  $h=0.1 \times \text{average pairwise distance}$) is surprisingly good and nicely picks up both the clover shape of the conditional quantile contours and their orientation; quite understandably, that quality improves as $T$ increases. 

 The impact of the bandwidth choice is illustrated in Figure \ref{case3_Kwidths}, with bandwidths of the form~$\ell\times\text{average pairwise distance}$ for $\ell =0.003$, 0.1,  1.0,  and 2.0; the best results (shown in Figure \ref{case3_sample_sizes}) are obtained for~$\ell = 0.1$. 
 Such a relatively small $h$ adequately captures the clover-like shape of the conditional distributions but produces somewhat rugged contours. A larger $h$ yields smoother contours while slightly blurring their shapes. 
 
The influence of the current unconditional quantile value on the corresponding one-step ahead predictive contours is studied in Figure~\ref{case3_predict}. Since stationarity does not hold (not even approximately), the central panel provides an estimation of the unconditional contours of~$X_T$ for $T=20,000$. The surrounding panels are obtained as in Cases~1 and~2; note that the estimations they are providing are the same for all values of $t$, hence for the predictive contours of $X_{T+1}$ computed at time $t=T$.  
Inspection of Figure~\ref{case3_predict} reveals that while the predictive center-outward  medians for $X_{T+1}$ are essentially the same (and coincide with the  center-outward  median of the current value of $X_T$) for all current values $x_T$, the quantile contours wildly vary a lot with the current unconditional quantile value at time $T$. 

This example demonstrates the considerable added value of our method: as far as the central value of $X_{T+1}$ is concerned, the predictive power of the current value $x_T$, hence of point predictors of $X_{T+1}$, is essentially nil; the same   current value $x_T$ of $X_T$, however, carries a great deal of information on the quantile contours of $X_{T+1}$. This has crucial implications, for instance, when forecasting risk levels at time $T+1$. 

\subsection{A real data analysis}\label{realsec}

We implemented our method to analyze a dataset of electroencephalogram (EEG) time series from Alzheimer's disease (AD) patients,  Frontotemporal Dementia (FTD) patients, and healthy (CN) controls. EEG is a non-invasive neurophysiological technique that records the brain’s electrical activity along a certain period of time via electrodes placed on the scalp. Each electrode keeps track of the synchronous electrical signals generated by the cerebral cortex area underneath it.
Our goal is to detect alterations in EEG signals and connectivity patterns between different brain regions in AD and FTD patients. Unlike the traditional univariate quantile autoregressive methods, our multivariate quantiles are  capturing the joint distributions of interrelated variables, hence are better able to detect and predict alterations in brain connectivity patterns. 

Alzheimer’s disease (AD) is a chronic, progressive neurodegenerative disorder and one of the most common incurable diseases \citep{safiri2024alzheimer}. It typically begins with memory loss, gradually affecting language, reasoning, and behavior, ultimately impairing daily functioning. 
Currently, more than 50 million people worldwide live with AD, imposing huge care and economic burden.  
Frontotemporal Dementia (FTD) is a group of neurodegenerative disorders that primarily affect the frontal and temporal lobes of the brain---the areas responsible for personality, behavior, decision making, and language \citep{bang2015frontotemporal}. 
It often occurs earlier than AD, typically between 45 and 65 years old. 
Unlike Alzheimer’s disease, memory is often preserved at the beginning, and the earliest signs tend to be changes in behavior, personality, or language ability \cite{bang2015frontotemporal}. 
The progression is featured by spreading atrophy from frontal/temporal lobes to other brain regions, leading to more global cognitive decline.
Studying the disease mechanisms and evolution/progression of AD and FTD would allow early detection/prevention, thereby facilitating appropriate treatment.  
In this section, we compare the EEG signal trajectories of AD patients and FTD patients, respectively, to that of healthy subjects (CN) to detect potential disease-specific signatures.

We explore two datasets from OpenNEURO repository (https://openneuro.org/), a public platform for brain imaging data. The first one is titled ``A dataset of EEG recordings from Alzheimer's disease, Frontotemporal Dementia, and Healthy subjects", available at https://openneuro.org/datasets/ds004504/versions/1.0.8. It contains the EEG resting state (closed eyes) recordings from 88 subjects, among whom 36 were diagnosed with Alzheimer's disease (AD group), 23 with Frontotemporal Dementia (FTD group), and 29 were healthy subjects (CN group). Assume that within each group, the observed time series are independent realizations of the same process. 
For recording, the 10-20 International System with 19 scalp electrodes (Fp1, Fp2, F7, F3, Fz, F4, F8, T3, C3, Cz, C4, T4, T5, P3, Pz, P4, T6, O1, and O2) were used and two reference electrodes (A1 and A2) were placed on the mastoids for impedance check. Each recording was performed according to the clinical protocol with participants in a sitting position and their eyes closed. 
The 19 electrodes are positioned at specific scalp locations, approximately corresponding to 19 brain regions (see Table \ref{tab:EEG_electrode_grouping}).
The second dataset is titled  ``A complementary dataset of open-eyes EEG recordings in a photo-stimulation setting from: Alzheimer's disease, Frontotemporal Dementia, and Healthy subjects'', available at https://openneuro.org/datasets/ds006036/versions/1.0.5. It provides eyes-open EEG recordings of the same cohort in multiple photic stimulations, complementary to the first dataset. All  EEG recordings have length $T$ between $150,000$ and $160,000$.\medskip

\begin{table}[ht]
\centering
\begin{tabular}{|l|p{4cm}|p{6cm}|}
\hline
\textbf{Functional Region} & \textbf{Electrodes Included} & \textbf{Approximate Brain Functions} \\
\hline
Frontal & Fp1, Fp2, F7, F8, F3, F4, Fz & Executive functions, decision-making, attention, working memory, motor planning \\
\hline
Central & C3, Cz, C4 & Primary motor cortex,  somatosensory processing \\
\hline
Temporal & T3, T4, T5, T6 & Auditory processing, language comprehension, memory \\
\hline
Parietal & P3, Pz, P4 & Sensory integration, spatial orientation,  attention \\
\hline
Occipital & O1, O2 & Visual processing \\
\hline
\end{tabular}
\caption{Grouping of 19 standard EEG scalp electrodes into functional regions with specific brain functions. Odd-numbered electrodes are on the left hemisphere, even-numbered ones on the right hemisphere. Electrodes with ``z'' are located along the midline.}
\label{tab:EEG_electrode_grouping}\vspace{-0mm}
\end{table}

We fit distinct nonparametric vector quantile autoregressive model for each group of subjects. Denote by $ \boldsymbol{A}^i\coloneqq\{A^i_t\ , t=1,\ldots, T^{i}_{A}\}$,  $i=1, \ldots, N_A$, $\boldsymbol{F}^i\coloneqq\{F^i_t,\  t=1,\ldots, T^{i}_{F}\}$, \linebreak $i=1, \ldots, N_F$, and $ \boldsymbol{C}^i\coloneqq\{C^i_t,\ t=1,\ldots, T^{i}_{C}\}$,  $i=1, \ldots, N_C$, respectively,  the EEG time series of the $i$-th subject within the AD, FTD, and CN groups. 
The procedure is as follows.\smallskip

\begin{compactenum}
    \item[(i)] Step 1: compute a consensus or representative time series $\boldsymbol{A}^* \coloneqq \{A^*_t\ , t=1,\ldots, T^*_{A}\} $ of $\{ \boldsymbol{A}^i \}_{i=1}^{N_A}$ via the R package ``dtwclust''. This representative time series is the DTW barycenter averaging \citep{petitjean2011global} of all the series within $\{ \boldsymbol{A}^i \}_{i=1}^{N_A}$. Similarly, compute $\boldsymbol{F}^*$ and $\boldsymbol{C}^*$, respectively, for $\{ \boldsymbol{F}^i \}_{i=1}^{N_F}$ and $\{ \boldsymbol{C}^i \}_{i=1}^{N_C}$.
    \item[(ii)] Step 2: align the time series within $\{ \boldsymbol{A}^i \}_{i=1}^{N_A}$, the time series within $\{ \boldsymbol{F}^i \}_{i=1}^{N_F}$, and the time series within $\{ \boldsymbol{C}^i \}_{i=1}^{N_C}$ to $\boldsymbol{A}^*$, $\boldsymbol{F}^*$, and $\boldsymbol{C}^*$, respectively, via the R package ``dtw'' \citep{giorgino2009computing}.
    \item[(iii)] Step 3: apply the method in Section \ref{Section:Estimation} (with \eqref{NW-estimatorn} instead of \eqref{NW-estimator}) to $\{\boldsymbol{A}^i\}_{i=1}^{N_A}$; this~yields \vspace{-1mm} 
\begin{equation*}
    \widehat{\rm P}_{A_{t+1} \vert A_t=A^*_t} = \frac{1}{N_A}\sum_{i=1}^{N_A}\sum_{t=1}^{T^i_A -1}  w_{t+1}^i(A^*_t) \cdot \delta_{A^i_{t+1}}, \quad \text{with}\quad w_{t+1}^i(A^*_t) = \frac{ K\left(\frac{A_t^i-A^*_t}{h} \right) }{\sum_{t=1}^{T^i_A-1}  K\left(\frac{A_t^i-A^*_t}{h} \right) }\vspace{-2mm}
\end{equation*}
(all other steps remain unchanged);  the current values to be conditioned on are the values of~$A^*_t$,  $t\in [1,T_A^i -1] $.
\item[(iv)] Proceed similarly with $\{ \boldsymbol{F}^i \}_{i=1}^{N_F}$ and $\{ \boldsymbol{C}^i \}_{i=1}^{N_C}$.\smallskip
\end{compactenum}

To visualize the conditional quantiles evolving over time, we fit the nonparametric vector quantile autoregressive model on pairs of EEG waves from different electrodes one by one. For example, we may pick the EEG signals from (Fz, F4) electrodes as the sample of a time series in $\mathbb{R}^2$. 
Note that our method applies to any fixed dimension, and we are able to fit the EEG waves from the 19 electrodes as a time series in $\mathbb{R}^{19}$. Quantiles in dimension 19, however, cannot be visualized or eye-inspected, and we therefore focus on bivariate series associated with pairs of electrodes. The main findings of our analysis are summarized in Figures~\ref{F1F2_closeeye},~\ref{O1O2_openeye}, and~\ref{LR}. 
\begin{figure}[b!]
    \centering
    \includegraphics[width=0.99\linewidth]{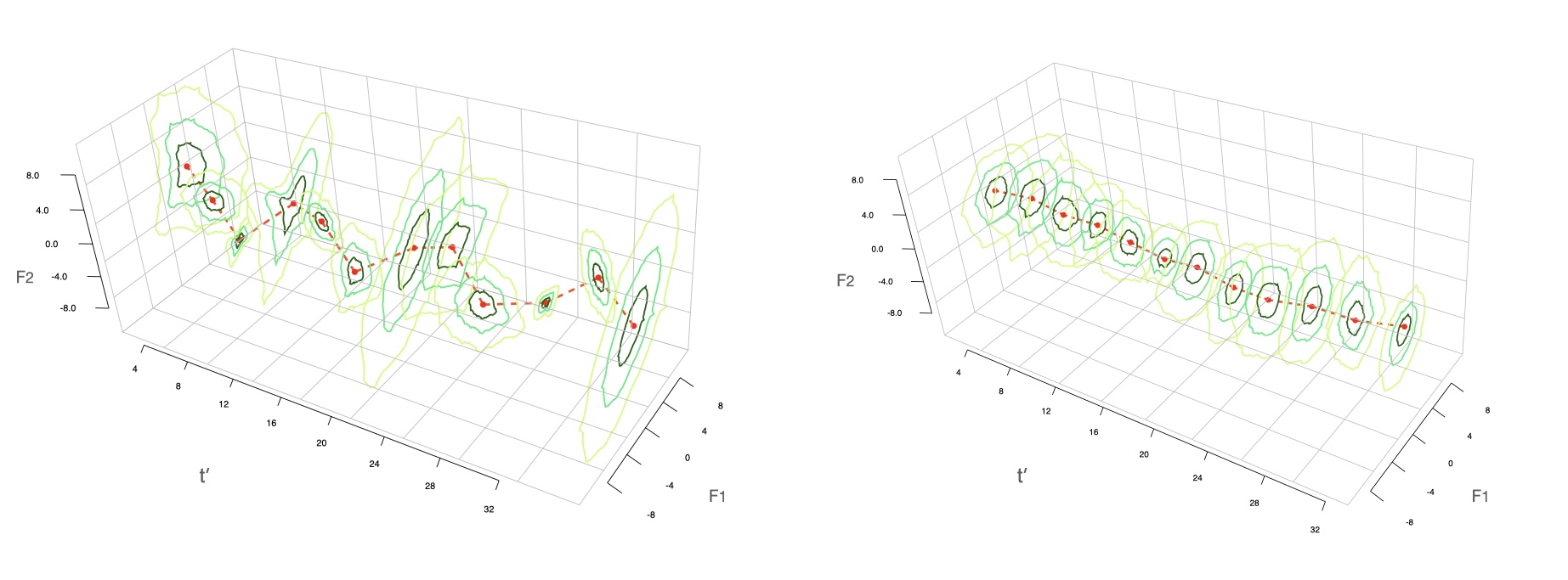}
    \caption{The estimated one-step-ahead conditional quantile contours of orders  $\tau=0.2$ (dark green), 0.4 (green), 0.8 (light olive), and the conditional one-step-ahead median (red) at selected time points for  the (F1, F2) EEG signals in healthy subjects (left panel) and FTD patients (right panel). In both panels, the horizontal axis stands for the rescaled time~$t'=t/50,000$. }
      \label{F1F2_closeeye}
\end{figure}
\begin{figure}[htbp!]
    \centering
    \includegraphics[width=0.99\linewidth]{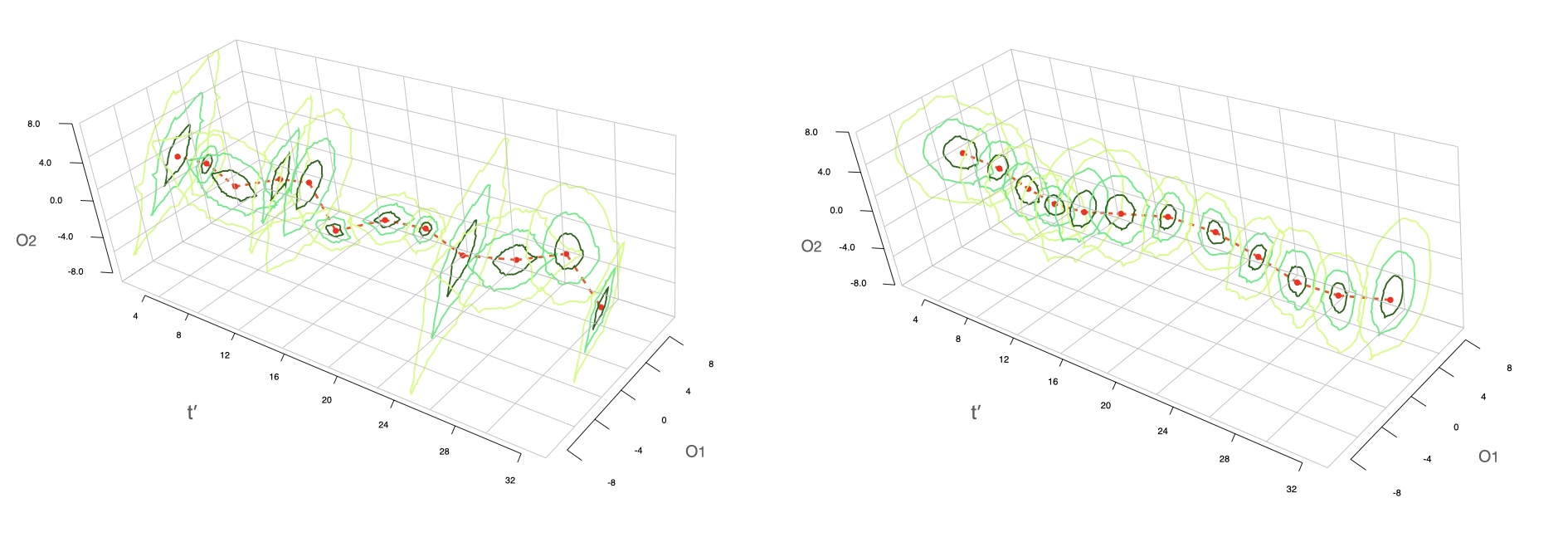}
    \caption{The estimated one-step-ahead conditional quantile contours of orders $\tau=0.2$ (dark green), 0.4 (green), 0.8 (light olive), and the conditional medians (red) at selected time points for the (O1, O2) EEG signals in healthy subjects (left panel) and  AD patients (right panel). In both panels, the horizontal axis stands for the rescaled time $t'=t/50,000$. }\vspace{-0mm}
    \label{O1O2_openeye}
\end{figure}
\begin{figure}[H]
    \centering
    \includegraphics[width=0.99\linewidth]{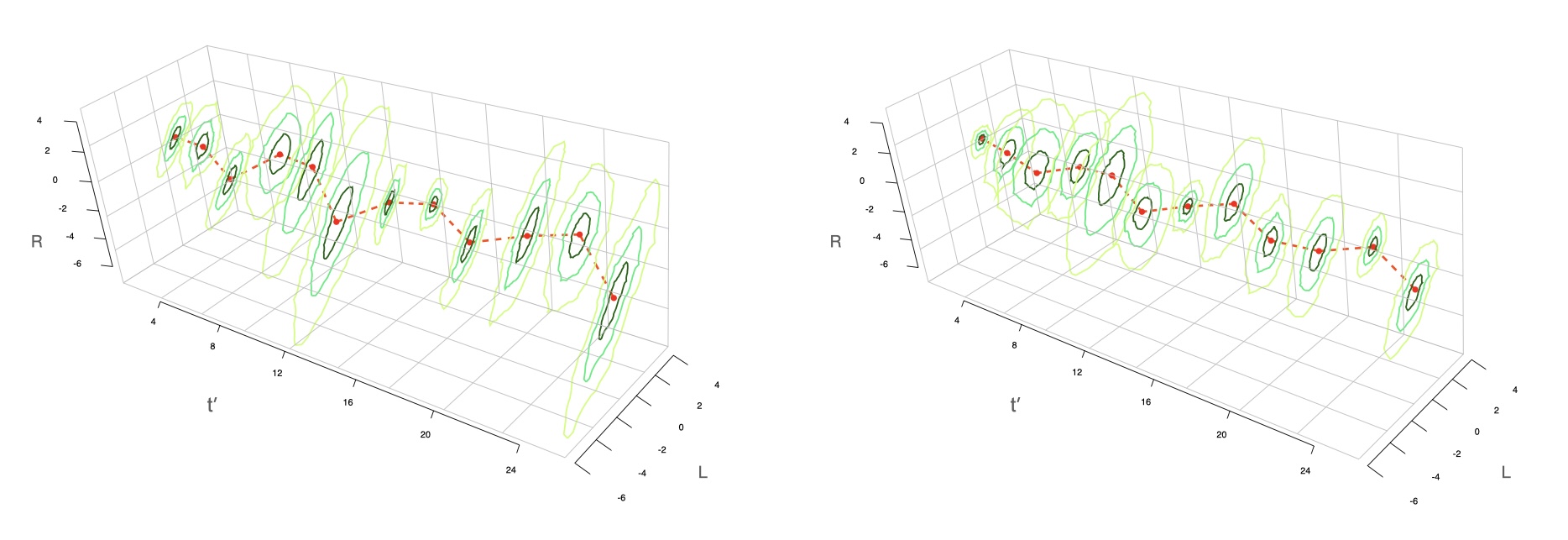}
    \caption{The estimated one-step-ahead conditional quantile contours of orders~$\tau=0.2$ (dark green), 0.4 (green), 0.8 (light olive), and the conditional medians (red) at selected time points  for the Principal Components of EEG signals in the left and right hemispheres in healthy subjects (left panel) and  AD patients (right panel).  In both panels, the horizontal axis stands for the rescaled time $t'$.\vspace{-3mm}}
    \label{LR}
\end{figure}

\begin{figure}[t!]
    \centering
    \includegraphics[width=0.8\linewidth, height=0.8\linewidth]{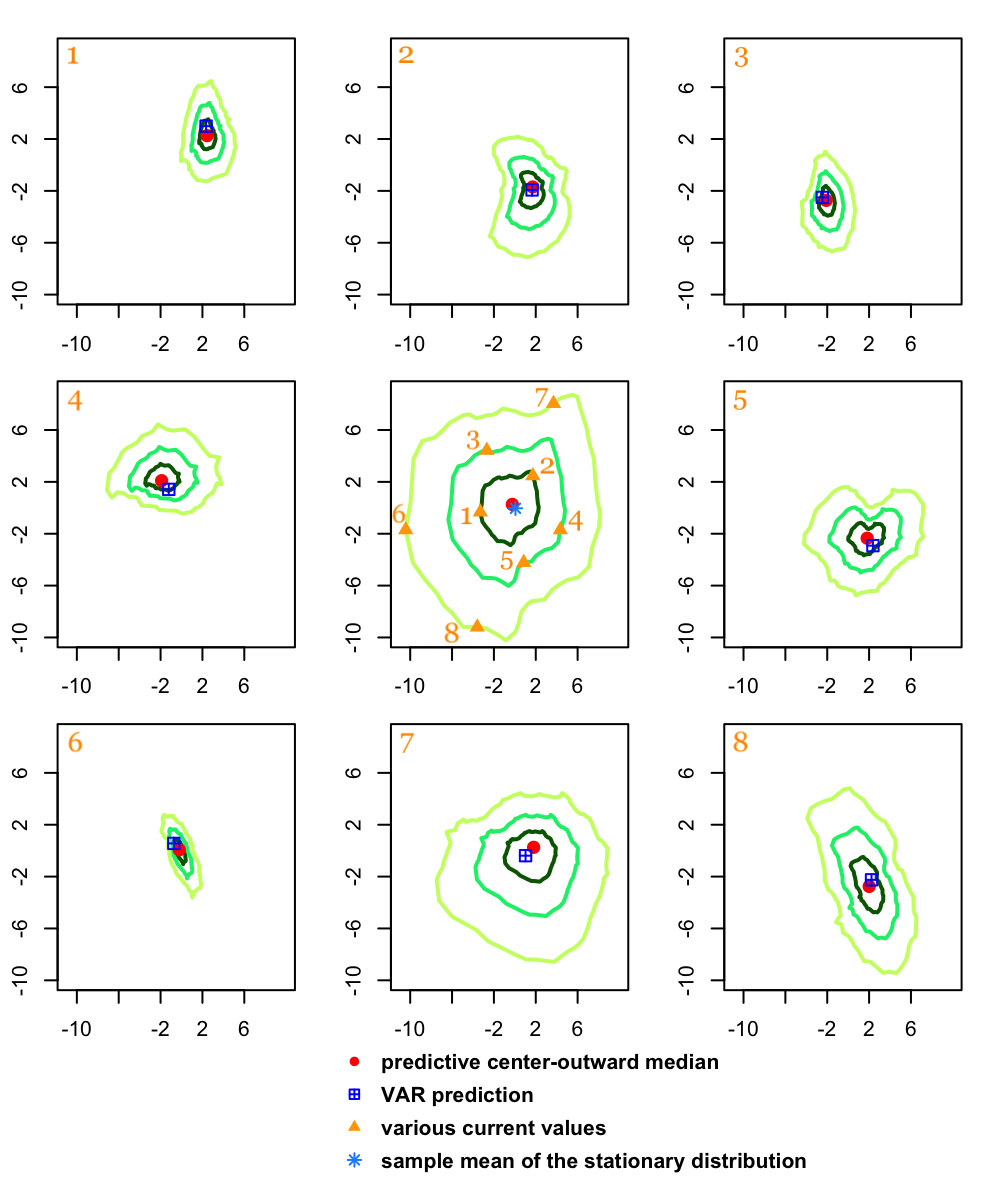}
    \caption{
 {The estimated one-step-ahead conditional quantile contours and medians at selected current values for the (F1, F2) EEG signals in healthy subjects.  The central panel shows the center-outward quantiles of orders $\tau=0.2$ (dark green), 0.4 (green), 0.8 (light olive), the center-outward median (red), and the sample mean (light blue) of the unconditional empirical distribution, and the current values (orange) at which quantile prediction is implemented in the surrounding panels. Each surrounding panel shows the one-step predictive center-outward quantile contours of orders $\tau=0.2$ (dark green), 0.4 (green), 0.8 (light olive), the conditional center-outward median (red), and the conventional VAR(1) one-step-ahead mean prediction (blue) at a particular current value.  } \vspace{-6mm}
    }
    \label{EEG_predictive_F1F2_CN}
\end{figure}

Figure \ref{F1F2_closeeye} compares the EEG signals from the (F1, F2) electrodes in the FTD and CN groups under closed-eye status. We observe that  the (F1, F2) EEG signals in FTD patients exhibit (relative to the CN group of healthy patients)
\begin{compactenum}  
\item[(a)]  {lower variation, with a flat median trajectory and homogeneous quantile contours}; 
\item[(b)]  less coherence/connectivity between F1 and F2 signals, as attested by the circular shape of  FTD patient's quantile contours; 
\item[(c)]   less entropy (spontaneous activity), with less conditional heteroskedasticity along the trajectory. \smallskip
\end{compactenum}
These findings are consistent with the fact, reported in the literature, that FTD patients have impaired activity and disrupted functional connectivity in their left and right prefrontal cortex \citep{bang2015frontotemporal}.  

\begin{figure}[h!]
    \centering
    \includegraphics[width=0.8\linewidth, height=0.8\linewidth]{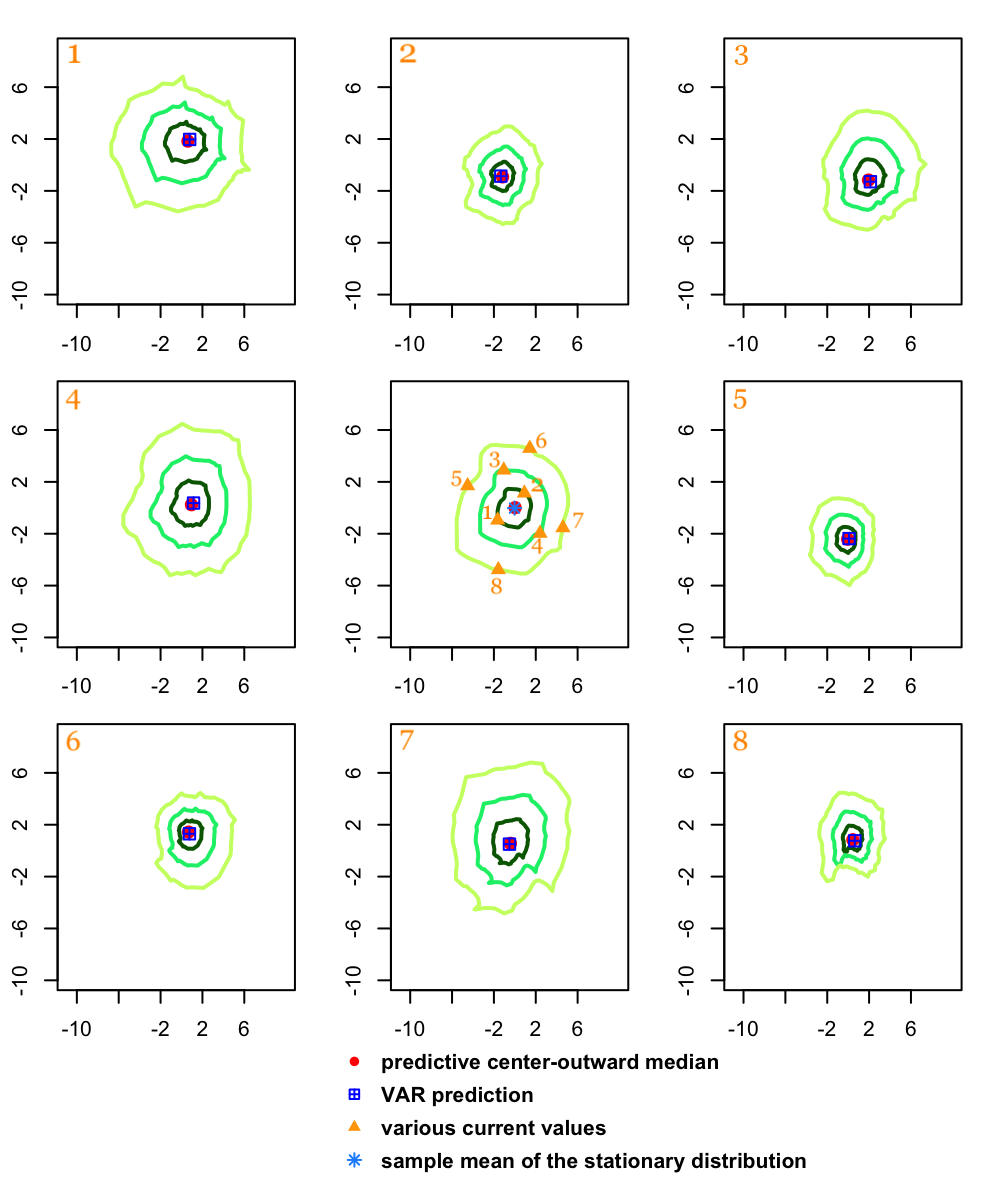}\vspace{-1mm}
    \caption{
{The estimated one-step-ahead conditional quantile contours and medians at selected current values for the (F1, F2) EEG signals in FTD patients.  The central panel shows the center-outward quantiles of orders $\tau=0.2$ (dark green), 0.4 (green), 0.8 (light olive), the  center-outward median (red), and the sample mean (light blue) of the unconditional empirical distribution, and the current values (orange) at which quantile prediction is implemented in the surrounding panels. Each surrounding panel shows the one-step predictive center-outward quantile contours of orders $\tau=0.2$ (dark green), 0.4 (green), 0.8 (light olive), the conditional center-outward median (red), and the conventional VAR(1) one-step-ahead mean prediction (blue) at a particular current value. \vspace{-10mm} } 
    }
    \label{EEG_predictive_F1F2_FTD}
\end{figure}

Figure \ref{O1O2_openeye} compares the EEG signals of (O1, O2) electrodes
in AD patients  and the healthy~CN group under open-eye status. Each time point where quantiles are depicted corresponds to a photic stimulus. 
It shows that (O1, O2) signals in AD patients are 
\begin{compactenum}  
\item[(a)]  less complex (lower entropy); 
\item[(b)]  less responsive to photic stimulations (less dispersion); 
\item[(c)]  with reduced synchronization/connectivity between the O1 and O2 signals (more circular contours). \smallskip
\end{compactenum}

\begin{figure}[h!]
    \centering
    \includegraphics[width=0.8\linewidth, height=0.8\linewidth]{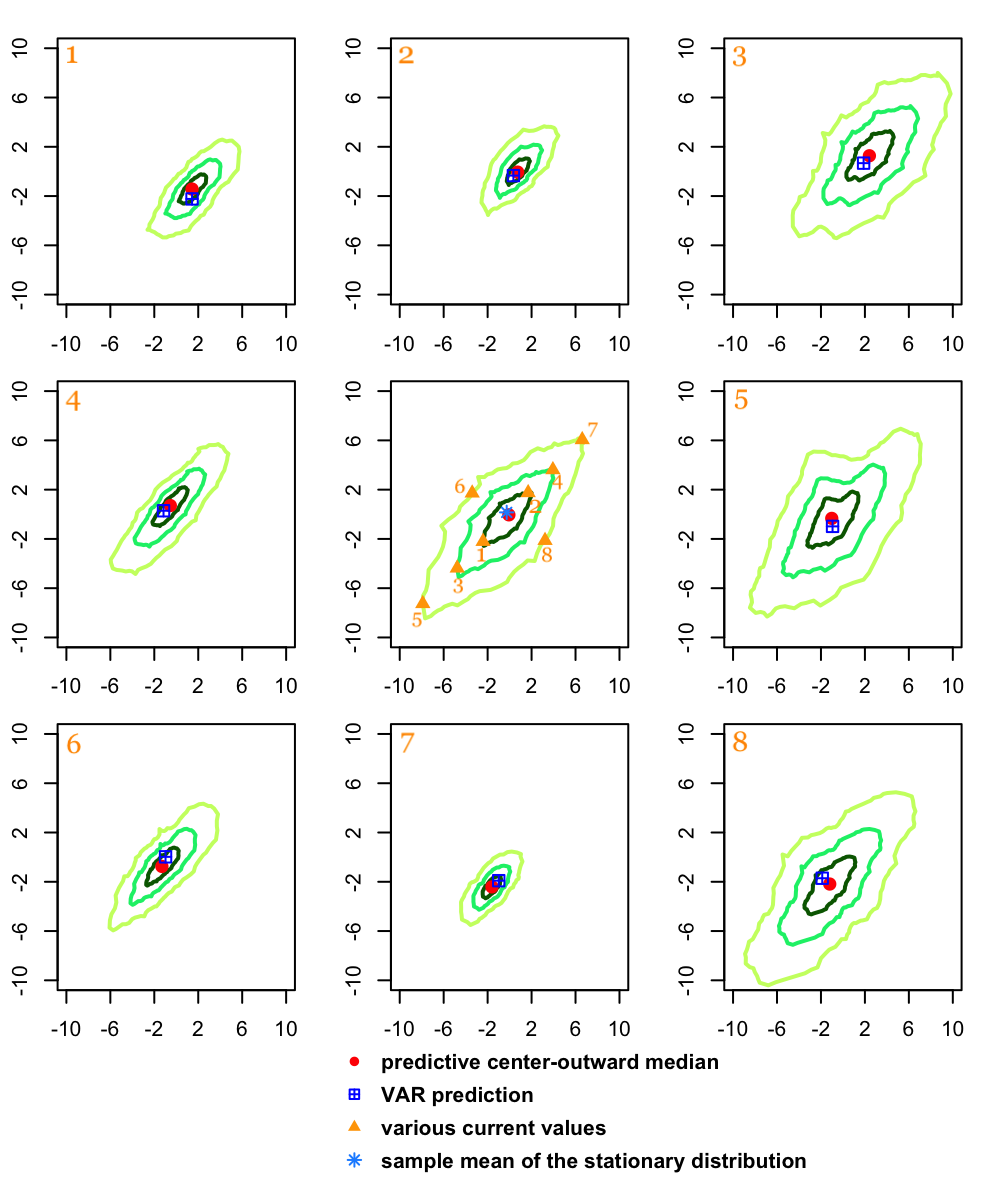}
    \caption{
{The estimated one-step-ahead conditional quantile contours and medians at selected current values for the (O1, O2) EEG signals in healthy subjects.  The central panel shows the center-outward quantiles of orders $\tau=0.2$ (dark green), 0.4 (green), 0.8 (light olive), the center-outward median (red), and the sample mean (light blue) of the unconditional empirical  distribution, and the current values (orange) at which quantile prediction is implemented in the surrounding panels. Each surrounding panel shows the one-step predictive center-outward quantile contours of orders $\tau=0.2$ (dark green), 0.4 (green), 0.8 (light olive), the conditional center-outward median (red), and the conventional VAR(1) one-step-ahead mean prediction (blue) at a particular current value.\vspace{-2mm}} 
    }
    \label{EEG_predictive_O1O2_CN}\vspace{-2mm} 
\end{figure}
Overall, in AD patients, the conditional quantiles/medians are less volatile or oscillating, and more predictable than those from healthy brains. This finding, again, aligns with the conclusions in the literature on AD symptoms \citep{safiri2024alzheimer},  which they complement with a quantitative assessment.  

\begin{figure}[h!]
    \centering
    \includegraphics[width=0.8\linewidth, height=0.8\linewidth]{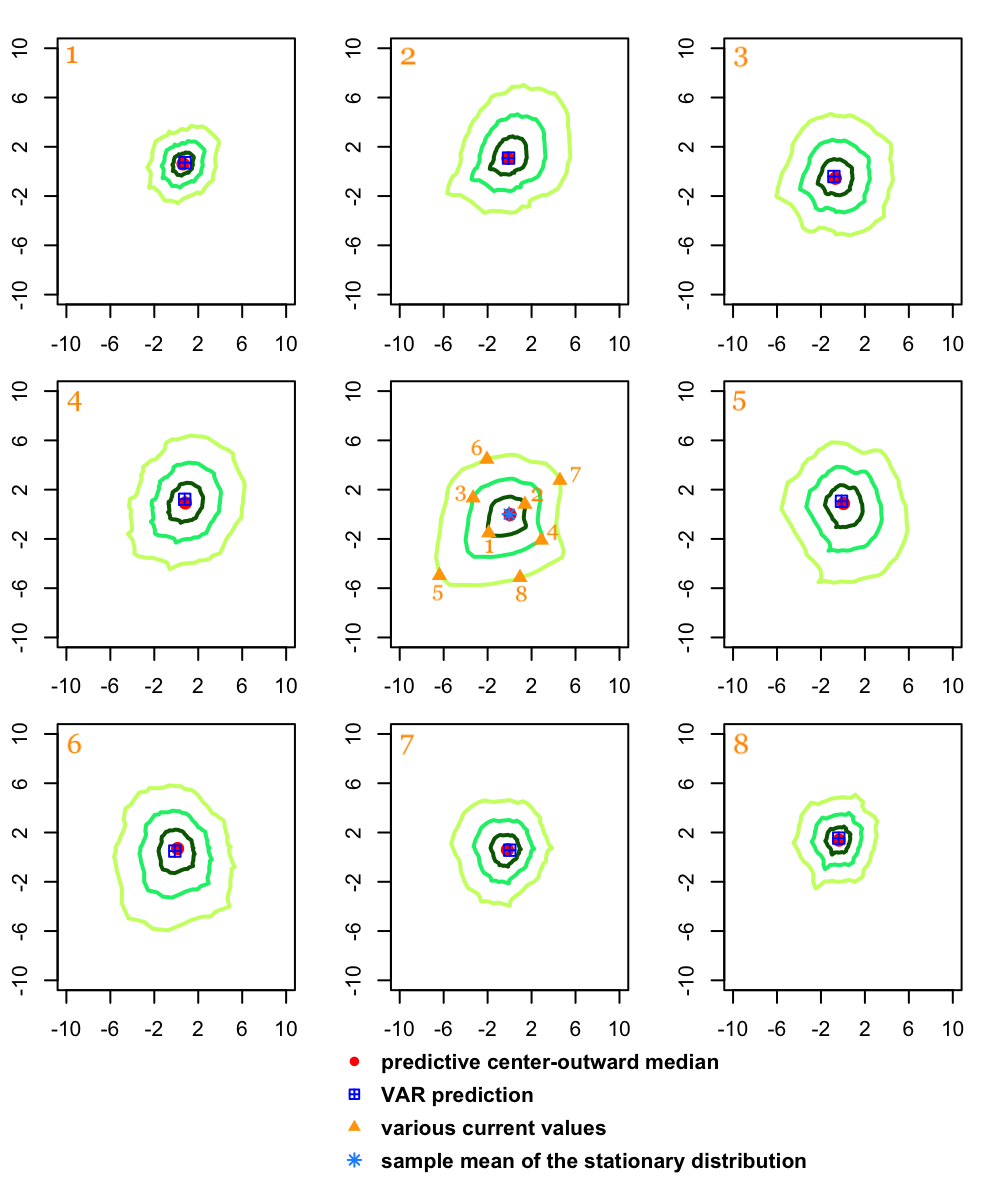}
    \caption{The estimated one-step-ahead conditional quantile contours and medians at selected current values for the (O1, O2) EEG signals in AD patients.  The central panel shows the center-outward quantiles of orders $\tau=0.2$ (dark green), 0.4 (green), 0.8 (light olive), the center-outward median (red), and the sample mean (light blue) of the unconditional empirical  distribution, and the current values (orange) at which quantile prediction is implemented in the surrounding panels. Each surrounding panel shows the one-step predictive center-outward quantile contours of orders $\tau=0.2$ (dark green), 0.4 (green), 0.8 (light olive), the conditional center-outward median (red), and the conventional VAR(1) one-step-ahead mean prediction (blue) at a particular current value.\vspace{-3mm}  } 
    \label{EEG_predictive_O1O2_FTD}
\end{figure}

Figure \ref{LR} compares the interhemispheric coherence or synchronization in AD patients and the healthy CN group ones under open-eye status. In this case, we take the first Principal Component (PC) of the EEG signals recorded by the electrodes on the left hemisphere (Fp1, F3, F7, C3, T3, T5, P3, O1), and the first PC of those on the right hemisphere (Fp2, F4, F8, C4, T4, T6, P4, O2) as the sample time series; these PCs summarize  the activities of the left and right cortexes. 
As shown in Figure \ref{LR}, the left and right EEG signals have reduced synchronization (more circular   quantile contour shapes) and less response to photic stimulations   (less volatile   trajectories) in the group of AD patients. 

We can also predict future trajectories based on the observed past. For illustration purposes, we show below   the one-step-ahead predictive quantiles for the EEG signals from (F1, F2) electrodes and (O1, O2) electrodes. 
A comparison between Figures \ref{EEG_predictive_F1F2_CN} and  \ref{EEG_predictive_F1F2_FTD} indicates that healthy brains exhibit more diverse/versatile and less predictable next-step distributions (conditional on  current values). Similar conclusions follow from comparing Figures \ref{EEG_predictive_O1O2_CN} and~\ref{EEG_predictive_O1O2_FTD}.

Summing up, our methods allow us to detect different patterns in the evolving trajectories of the conditional quantile contours of the EEG signals from several electrodes (corresponding to different cortex regions) in the groups of AD or FTD patients and the group of healthy subjects. Contrary to traditional univariate quantile autoregression models, our method is able to handle multi-dimensional   time series and detect alterations in the conditional joint distributions. Compared to the traditional vector autoregression model, which focuses on mean regression, our method is capable of depicting the entire conditional distribution, hence providing much richer information.

\appendix

\section{Appendix: Proofs}
\subsection{Measurability and the control of probability contents}
\subsubsection{Proof of \cref{lemma:set-valued-prediction-measure}}
The proof of \cref{lemma:set-valued-prediction-measure} requires a few preparatory steps. \medskip

\noindent {\it Preparatory Step 1: Fell and graphical topologies.} Let $\mathcal{V}$ be an open subset of $\R^d$ and consider~a sequence~$\{B_t\}_{t\in \N}$ of subsets of $\mathcal{V}$. Define the {\it inner} and {\it outer limits} of   $\{B_t\}_{t\in \N}$   {\it relative} to~$\mathcal{V}$   as
\begin{align*}
{\rm Liminn}_{t\to\infty}^{\mathcal{V}}B_{t}\coloneqq \{ u\in \mathcal{V}: \ \text{exists $\{x_t\}_{t\in\NN}$ with $x_t\in B_{t}$ such that $x_t\to u$ as $t\to\infty$}\}
\end{align*}
and
\begin{align*}
{\rm Limout}_{t\to\infty}^{\mathcal{V}}B_{t}\coloneqq \{ u\in \mathcal{V}: \  \text{exists $\{x_{n_k}\}_{k\in\N}$ with $x_{n_k}\in B_{n_k}$ such that $x_{n_k}\to u$ as $k\to\infty$}\},
\end{align*}
respectively. 
If $B=\operatorname{Liminn}_{t\to\infty}^{\mathcal{V}}B_t=\operatorname{Limout}_{t\to\infty}^{\mathcal{V}}B_{t}$, we say that $B$ is the {\it Kuratowski-Painlevé} limit of $\{B_t\}_{t\in \N}$ relative to $\mathcal{V}$ and write $ B=\operatorname{Lim}_{t\to\infty}^{\mathcal{V}}B_t$  or $B_t\xrightarrow{\mathcal{V}} B$. 

Denote  by  $CL_{\neq \empty}(\mathcal{V})$ the set of closed non-empty sets of $\mathcal{V}$. 
For a set $B\in 2^{\mathcal{V}}$, let  
$$ B^+\coloneqq \{ C\in 2^{\mathcal{V}}:\ C\subset B\}\quad {\rm and}\quad B^-\coloneqq \{ C\in 2^{\mathcal{V}}:\ C\cap B\neq \emptyset\}.$$
The Fell topology $\tau_{F}$ on $CL_{\neq \empty}(\mathcal{V})$ has as a subbase all sets of the form $B^-$, where $B$ is a nonempty open subset of $\mathcal{V}$, plus all sets of the form $W^+$, where $W\in \tau_{\mathcal{V}}\setminus \{ \emptyset\} $ has compact complement (see Definition 5.1.1 in \cite{Beer1993TopologiesOC}). 

Now consider the case of $\mathcal{V}$ being an open subset of $\R^{d}\times \R^d$. The topological space~$(\mathcal{V},\tau_{\mathcal{V}})$ then  is locally compact and second countable, so that (Ibid., Theorem  5.1.5) $(CL_{\neq \empty}(\mathcal{V}), \tau_{F})$ is a Polish space. We use the notation $B_t\xrightarrow{\tau_{F}} B$ for a sequence $\{ B_t\}_{t\in \N}\subset CL_{\neq \empty}(\mathcal{V})$ converging,  as $t\to\infty$,  to $B$ with respect to the topology $\tau_{F}$. The Kuratowski-Painlevé convergence and the Fell topology $\tau_{F}$ are related via this sequential characterization of the topology: indeed,~$B_t\xrightarrow{\tau_{F}} B$ if and only if $B=\operatorname{Lim}_{t\to\infty}^{\mathcal{V}}B_t$ (Ibid., Theorem 5.2.10). 

A maximal monotone operator $M:\R^d\to 2^{\R^d}$ is a convex-closed-valued mapping \cite[Exercise 12.8]{rockafellar2009variational}. That is, $M (u)$ is closed and convex for all $u\in \R^d$. Moreover, the graph ${\rm graph}(M)\coloneqq\{( u , v ):\  v \in M(u) \}$ of $ M$ is closed \cite[Theorem 24.4]{rockafellar1970}.  Therefore, if $M( u )\neq \emptyset$ for some $ u$ in some open subset $U$ of $\R^d$,  ${\rm graph}(M)\in CL_{\neq \empty}({U}\times \R^d)$  

It is well known (see e.g.~\cite[Theorem 1.12.4]{Vart_Well}) that the space of probability measures
$\mathcal{P}(\R^d)$ endowed with the weak topology (i.e., $\nu_n\xrightarrow{w}\nu$ if $\int f {\rm d}\nu_n\to \int f {\rm d}\nu $ for all bounded continuous function $f:\R^d\to \R$) is complete, separable, and metrizable by  the bounded Lipschitz metric 
$$d_{\rm BL}(\nu_1, \nu_2)=\sup_{f\in {\rm BL}(\R^d)}\Big\vert \int f {\rm d}\nu_1 -\int f {\rm d}\nu_2 \Big\vert, $$
where $${{\rm BL}}(\R^d)\coloneqq \{f:\R^d\to \R: \ \vert f( x )-f( y )\vert \leq \| x - y  \| \ {\rm and}\ \vert f( x )\vert \leq 1, \ \forall  x , y \in \R^d\}.$$ 

\noindent {\it Preparatory Step 2: Continuity and definition of $\Gamma$.} 
Since $ \mu_d \ll \ell_d $,  for any $\nu\in \mathcal{P}(\R^d),$ McCann's theorem (see \cite{McCann}) guarantees the existence of a unique probability distribu\-tion~$\gamma_{\nu}\in \mathcal{P}(\R^d\times \R^d)$ with cyclically monotone support such that 
$\gamma_{\nu}(\R^d\times B)=\nu(B) $ and~$\gamma_{\nu}(B\times \R^d)={ \mu_d (B)} $ for all $B\in \mathcal{B}^d$. A well-known result of Rockafellar (see \cite{RockafellarMaximalMonot}) establishes the existence of a convex function $\varphi_{\nu}$ from $\R^d$ to $\R$  such that~${\rm supp}(\gamma_{\nu})\subset{\rm graph}(\partial \varphi_\nu) $. 
Define the mapping 
\begin{equation}\label{Gammadef}\Gamma:(\mathcal{P}(\R^d), d_{\rm BL}) \ni {\nu} \mapsto    (\mathbb{B}^d\times \R^d)\cap {\rm graph}(\partial \varphi_\nu) \in ({\rm CL}_{\neq \emptyset}(\mathbb{B}^d\times \R^d), \tau_{F}).
\end{equation}

It follows from  \cite[Lemma 4.2]{Segers2022GraphicalAU} that $\Gamma$ is well defined---i.e., although several distinct versions of~$\varphi_\nu$   exist, the corresponding $\Gamma$'s agree in $\mathbb{B}^d$. 
The following result shows that $\Gamma$, moreover, is continuous.

\begin{lemma}\label{Lemma:continouity}
  The map $\Gamma$ defined in \eqref{Gammadef} continuous.
\end{lemma}
\begin{proof}
    Since both  $(\mathcal{P}(\R^d), d_{\rm BL}) $ and $({\rm CL}_{\neq \emptyset}(\mathbb{B}^d\times \R^d), \tau_{F})$ are separable  metric spaces, continuity of $\Gamma$ is equivalent to sequential continuity. Therefore, let $ \{\nu_t\}_{t\in \N}\subset \mathcal{P}(\R^d)$ be a sequence such that $ \nu_t\xrightarrow{w}{\nu}\in \mathcal{P}(\R^d)$ as $t\to\infty$.   Theorem 1.1 in \cite{Segers2022GraphicalAU} implies that~$ {\rm graph}( \partial \varphi_{\nu_t})\xrightarrow{\tau_A} {\rm graph}(\partial \varphi_{\nu})$, which completes the proof. 
\end{proof}

\noindent {\it Preparatory Step 3: Measurability of the distance function.} 
Let $C\subset\R^d$ be a closed set. The distance between $C$ and $ x\in\R^d $ is defined as
 $ d(x,C)\coloneqq \inf_{c\in C}\| c-x\|$. Defining  
$$ g_x: \R^d\times \mathcal{P}(\R^d)  \ni( u ,{\nu})\mapsto g_x(u, \nu)\coloneqq d( x ,\partial \varphi_\nu( u ))\in \R, $$
let us show  that $g_x$ is lower semicontinuous, that is,
\begin{equation}
        \label{lsc}
 \liminf_{(u_n,\nu_n)\to (u,\nu)} g_x (u_n,\nu_n)=\lim_{T\to \infty}\inf_{\substack{\| u-u'\|\leq 1/T\\ d_{{\rm BL}}(\nu',\nu)\leq 1/T}} d(x, \partial \varphi_{\nu'}(u')) \geq g_{x}(u,\nu).
    \end{equation}
To how this, suppose that,  for some $\epsilon>0$, 
$$\liminf_{( u_n,\nu_n)\to ( u ,\nu)} g_x ( u_ n,\nu_n)\leq d( x , \partial \varphi_{\nu}( u ))-\epsilon .$$ 
Then, there exists a sequence $ \{ ( x_ n, u_ n,\nu_n)\}_{n\in \N}\subset \R^d\times \mathbb{B}^d\times \mathcal{P}(\R^d)$ and $n_0=n_0(\epsilon)\in \N$ such that~$ u_ n \to   u  $,  $\nu_n\xrightarrow{w} \nu$, and $ x_n\in \partial \varphi_{\nu_n}( u_ n)$ with 
 \begin{equation}\label{previous display} \|  x_n- x \|=d( x ,\partial \varphi_{\nu_n}( u_ n))\leq d( x , \partial \varphi_{\nu}( u ))-\epsilon/2 \quad \text{for all $T\geq n_0$.} \end{equation}
 The sequence $ \{  x_n\}_{n\in \N}$ is bounded, so that it has a limit point $ x ^*$. It follows from \eqref{previous display} that 
 $$ \|  x ^*- x \|\leq d( x , \partial \varphi_{\nu}( u ))-\epsilon/2\leq \|  v - x \| -\epsilon/2 \quad \text{for all $ v \in \partial \varphi_{\nu}( u )$}.$$
 However, from \cref{Lemma:continouity}, $  x ^*\in \partial \varphi_{\nu}( u )$, yielding the contradiction $\|  x ^*- x \|<\|  x ^*- x \|$.

Therefore, $g_{ x }$ is lower semicontinuous, so that, due to  \cite[Theorem~3.87]{InfiniteDimensionalAnalysis_2006}, it is the pointwise limit of a sequence of continuous functions. As a consequence of Corollary~4.30 in \cite{InfiniteDimensionalAnalysis_2006}, $g_{ x }$ thus is $ (\mathcal{B}^d\otimes \mathcal{B}(\mathcal{P}(\R^d)))/ \mathcal{B}^d $-measurable. \medskip 

\noindent {\it Preparatory Step 4:  Conclusion.} Set $ x \in \R^d$, $A\in {\mathcal B}(\R^d)$, and define the map 
$$
    \xi_{ x }: A\times \Omega \ni  
( u ,\omega) \mapsto \xi_{ x }( u ,\omega)\coloneqq g_{ x }( u , \mathbb{P}_{\X\vert \mathcal{G}}(\cdot, \omega))\in\R.
$$
Being the composition of the $ (\mathcal{B}^d\otimes \mathcal{B}(\mathcal{P}(\R^d)))/ \mathcal{B}^d $-measurable function $ g_{ x } $ with\linebreak  the $(\mathcal{B}^d\otimes \mathcal{G})/(\mathcal{B}^d\otimes \mathcal{B}(\mathcal{P}(\R^d)))$-measurable  function $( u ,\omega)\mapsto ( u , \mathbb{P}_{\X\vert \mathcal{G}}(\cdot, \omega))\in A\times \mathcal{P}(\R^d) $, 
 $\xi_x$ is~$({\mathcal{B}^d}\otimes{\mathcal A})/{\mathcal{B}}$-measurable.\medskip

We now turn to the proof of \cref{lemma:set-valued-prediction-measure}. \medskip

\noindent{\it Proof of \cref{lemma:set-valued-prediction-measure}(i)\! (Measurability of the conditional quantile function).} 
Denoting by~$(\Omega', \mathcal{A}')$  a measurable space and by $S:\Omega'\to 2^{\R^d}$  a closed-valued map, recall that a set-valued map~$S$ is measurable if and only if the function $\varpi\mapsto d( x , S(\varpi))$ is measurable for all $ x \in \R^d$ (see Theorem~14 in \cite{rockafellar2009variational}).   
The conclusion of {\it Preparatory Step~4} is that~$( u ,\omega)\mapsto \xi_{ x }( u ,\omega)\coloneqq d( x , \mathbb{Q}_{\X\vert \mathcal{G} }( u ,\omega))$ is~$(\mathcal{B}^d\otimes \mathcal{G})/\mathcal{B}^d$-measurable. The measurability of the quantile function $\omega\mapsto\mathbb{Q}_{\X\vert \mathcal{G} }( u ,\omega)$ follows.\medskip\hfill$\square$

\noindent{\it Proof of  \cref{lemma:set-valued-prediction-measure} (ii) (Measurability of the conditional distribution function).}    
The 
 proof follows as for  \cref{lemma:set-valued-prediction-measure}(i) by replacing  $\mathbb{Q}_{\X\vert \mathcal{G} }$ with~$\mathbb{F}_{\X\vert \mathcal{G} }$ in each 
  step. Note that only ${\rm graph}(\mathbb{Q}_{\X\vert \mathcal{G} })$   appears in Lemma~\ref{Lemma:continouity}, so that the result still holds when replacing $(\mathbb{B}^d\times \R^d)\cap {\rm graph}(\mathbb{Q}_{\X\vert \mathcal{G} })$ by $(\R^d\times \mathbb{B}^d)\cap {\rm graph}(\mathbb{F}_{\X\vert \mathcal{G} })$ in the definition of $\Gamma$.  \hfill$\square$
\subsubsection{Proof of \cref{lemma:proba-control}}
    Since the mapping 
$\Omega\ni \omega \mapsto (X_{t+1}(\omega), \omega) \in \R^d\times \Omega$ is $\mathcal{A}/(\mathcal{B}^d\otimes \mathcal{A})$-measurable,  \cref{lemma:set-valued-prediction-measure} implies that the set-valued mapping  
 $  \Omega\ni \omega\mapsto \mathbb{F}_{t+1|t}( X_{t+1}(\omega), \omega) $ 
is $\mathcal{A}$-measurable. The first claim \eqref{Lemma2.6(1)} follows.  The second claim \eqref{Lemma2.6(2)} is a consequence of the fact that 
\begin{align*}
    \mathbb{P}\left( X_{t+1}\in  \mathcal{R}_{t+1|t}(\tau|\cdot) \bigg \vert \mathcal{F}_{\leq t}\right)(\omega) = \mu_d(\{u: {\bf Q}_{t+1|t}(u|\omega) \in  \mathbb{ Q}_{t+1|t}(\tau\mathbb{B}^d|\omega) \})
\end{align*}
with 
$ \tau\,\mathbb{B}^d \subset \{u: {\bf Q}_{t+1|t}(u|\omega) \in  \mathbb{ Q}_{t+1|t}(\tau\mathbb{B}^d|\omega) \}.  $ 
Finally, \eqref{Lemma2.6(3)} follows from the fact that under the additional assumption \eqref{moreover}, ${\bf Q}_{t+1|t}(\cdot|\omega)$ is a.e.~invertible.
\medskip\hfill$\square$


\subsection{Monotonicity and consistency of the estimated quantile map }
\subsubsection{Proof of \cref{lemma:monotone}}
    Since $\hat{\pi}$ has monotone support, we get 
\begin{align*}
\langle \widehat{\mathbb{Q}}_T(u_s)- \widehat{\mathbb{Q}}_T(u_r), u_s-u_r \rangle &= k\left\langle \sum_{j=2}^T   \hat{\pi}_{s,j}  X_{j} -\sum_{j=2}^T    \hat{\pi}_{r,i}  X_{i}, u_s-u_r \right\rangle\\
    &=  k^2\sum_{i,j=2}^T   \hat{\pi}_{s,j}   \hat{\pi}_{r,i}  \left\langle   X_{j} - X_{i}, u_s-u_r \right\rangle\\
    &=  k^2\sum_{(i,j):\hat{\pi}_{s,j}, \hat{\pi}_{r,i}>0}   \hat{\pi}_{s,j}   \hat{\pi}_{r,i}  \left\langle   X_{j} - X_{i}, u_s-u_r \right\rangle \geq 0, 
\end{align*}
so that $u\mapsto \widehat{\mathbb{Q}}_T(u| x)$ is   monotone.\hfill$\square$
\subsubsection{Proof of \cref{lemma:consistency}}

Recall that $\rm P_1$, with density $p_1$, stands for the distribution of $X_1$ and set $x\in {\rm supp}(\rm P_1) $.  Let~$f:\R^d\to \R$ be  bounded and continuous, and define 
$$ K_h\left(\frac{x-y}{h} \right) \coloneqq \frac{K\left(\frac{x-y}{h} \right)}{\int K\left(\frac{x-y}{h} \right) {\rm d}y}= \frac{K\left(\frac{x-y}{h} \right)}{ h^d\int K\left(v\right) {\rm d}v}=\frac{K\left(\frac{x-y}{h} \right)}{ h^d},$$ 
where 
 $$ \hat{r}_{f}(x) \coloneqq \frac{1}{T-1} \sum_{t=1}^{T-1} 
f(X_{t+1})K_h\left(\frac{x-X_t}{h} \right).$$
Let us show  that 
$ \mathcal{E}=\left\vert  \hat{r}_{f}(x) -  r_{f}(x)  \right\vert  \xrightarrow{\PP} 0, $
where 
$$ r_{f}(x) \coloneqq p_1(x) \E[f(X_{t+1})\vert X_t=x].$$

As usual in this context, we split $\mathcal{E}$ into bias and variance components. 

\begin{enumerate}
\item[(a)]{\it (Bias term)} It follows from stationarity that
\begin{align*}
  \E\left[\hat{r}_{f}(x)\right]&=   \frac{1}{T-1}\sum_{t=1}^{T-1} \E\left[  f(X_{t+1}) K_h\left(\frac{x-X_t}{h} \right)\right] 
   =\int  {r}_{f}(x_1) K_h\left(\frac{x-x_1}{h} \right)  {\rm d} x_1.
\end{align*}

Fix $\epsilon>0$. Since ${r}_{f}$  is continuous on  ${\rm supp}(P_1)$ and vanishes at infinity,  there exists a compactly supported continuous function $g_\epsilon $  such that 
\begin{equation}
    \label{eps/3-1}
    \|g_\epsilon-{r}_{f}\|_\infty \leq \epsilon/3.
\end{equation}
Hence, by using the fact that $\int  K_h\left(\frac{x-x_1}{h}\right) {\rm d}x_1=1$, we get 
\begin{equation}
    \label{eps/3-2}
     \left| \E\left[\hat{r}_{f}(x)\right]- \int    g_\epsilon(x_1) K_h\left(\frac{x-x_1}{h} \right)  {\rm d} x_1\right|\leq \frac{\epsilon}{3}.
    \end{equation}
Let $w$ be the modulus of continuity of the uniformly continuous function $g_\epsilon$. Then, with the change of variables $v={(x-x_1)}/{h} $
\begin{align*}
     \left\vert \int    g_\epsilon(x_1) K_h\left(\frac{x-x_1}{h} \right)  {\rm d} x_1 -  g_\epsilon(x)\right\vert  &\leq  \int   K_h\left(\frac{x-x_1}{h}  \right) \omega(x-x_1)  {\rm d} x_1 \\
     &=\int   K\left(v  \right) \omega(vh)  {\rm d} v
\end{align*}
where the function $\omega(vh) $ is bounded and tends to zero as $ h\to 0$. By the dominated convergence theorem,  there exists $h_\epsilon>0$ such that 
\begin{align}\label{previousd}
     \left\vert \int    g_\epsilon(x_1) K_h\left(\frac{x-x_1}{h} \right)  {\rm d} x_1 -  g_\epsilon(x)\right\vert  \leq \frac{\epsilon}{3} \quad \text{for all } h< h_\epsilon.
\end{align}
Together,  \eqref{eps/3-1},  \eqref{eps/3-2}, and \eqref{previousd} imply that, for $h$ small enough,  
$$ |  \E\left[\hat{r}_{f}(x)\right]-  p_1(x) \E[f(X_{t+1})\vert X_t=x]| \leq \epsilon $$
so that
$$ \E\left[\hat{r}_{f}(x)\right] \to  p_1(x) \E[f(X_{t+1})\vert X_t=x]\quad\text{ as $T\to\infty$}.$$

\item[(b1)]{\it (Variance term---Geometric ergodicity)} Let us analyze each term of the sum
\begin{multline*}
    \E[ (\hat{r}_{f}(x)-\E[\hat{r}_{f}(x)])^2] \\= \frac{1}{(T-1)^2} \sum_{s,t=1}^{T-1} 
\underbrace{{\rm Cov}\left(f(X_{t+1})K_h\left(\frac{x-X_t}{h}\right), f(X_{s+1})K_h\left(\frac{x-X_s}{h}\right) \right)}_{c_{t,s,h}}
\end{multline*}
separately. On the one hand, for $s=t$ we have 
\begin{align*}
    |c_{t,t,h}| &\leq \|f\|_\infty^2 \int K_h^2\left(\frac{x-x_1}{h}\right) p_1(x_1) {\rm d}x_1
    \leq h^{-d} \underbrace{\|f\|_\infty^2 \|p_1\|_\infty \int K^2\left(v\right)  {\rm d}v}_{C_1}.
\end{align*}
On the other hand, for $s+1<t$, letting 
$$ S_t(x) \coloneqq f(X_{t+1})K_h\left(\frac{x-X_t}{h}\right)-\E\left[f(X_{t+1})K_h\left(\frac{x-X_t}{h}\right)\right],$$
the mixing assumption yields
\begin{align*}
     |c_{t,s,h}|&= \E[ S_s(x) \E[ S_t(x)| X_{t-1}, \dots, X_s] ] \\
     &\leq  \| S_s(x)\|_{L^2(\PP)}\|\E[ S_t(x)| X_{t-1}, \dots, X_s]  \|_{L^2(\PP)}\\
     &\leq  \delta^{t-s}\| S_s(x)\|_{L^2(\PP)} \| S_t(x)\|_{L^2(\PP)}\leq C_1  \delta^{t-s} h^{-d}.
\end{align*}
As a consequence, 
\begin{align*}
    \E[ (\hat{r}_{f}(x)-\E[\hat{r}_{f}(x)])^2]&\leq \frac{C_1}{T^2 h^d} \sum_{t,s=1}^T \delta^{|t-s-1|}\\
    &\leq \frac{2 C_1}{T^2 h^d} \sum_{j=0}^{T}  (T-j)\delta^{j-1}
    \leq  \frac{2 C_1}{T h^d} \sum_{j=0}^T  \delta^{j-1}
    \leq \frac{C_2}{T h^d},
\end{align*}
which tends to zero as $T h^d\to \infty $. 

\item[(b2)]{\it (Variance term---mixing)} Let us show that $|c_{t,s,h}|$ decreases exponentially fast in $|t-s|$. Since  $S_t$ is $\mathcal{F}_{\leq t+1}$-measurable and upper bounded by $ Ch^{-d}$, where~$ C= \|K\|_\infty \| f\|_\infty$, 
we get,
for $s+1<t$, 
$$ \E[S_t S_{s} ] \leq \|S_t\|_\infty \|S_{s}\|_\infty  \alpha(|t-s-1|) \leq C^2 h^{-2d} \alpha(|t-s-1|).$$
\end{enumerate}
The convergence to zero of $\mathcal E$ follows, which completes the proof of \cref{lemma:consistency}.\hfill$\square$

\subsubsection{Proof of \cref{Markov:Estimation}}
    We know that,  for all $x\in {\rm supp}(P_1)$, 
    $$ \PP(\sup_{u\in \mathcal{K}}\| \widehat{\mathbf{Q}}_T ( u\vert x)- {\mathbf{Q}}_{2|1}( u\vert x) \|>\epsilon) \to 0 \quad\text{as $T\to\infty$.}$$
    Hence, 
    for any $R>0$, 
    \begin{align*}
        &\PP\left(\sup_{u\in \mathcal{K}}\| \widehat{\mathbf{Q}}_T ( u\vert X_{T})- {\mathbf{Q}}_{2|1}( u\vert X_{T}) \|>\epsilon\right)\\
        &\ \ \ \leq \underbrace{\PP\left(\left(\sup_{u\in \mathcal{K}}\| \widehat{\mathbf{Q}}_T ( u\vert X_T)- {\mathbf{Q}}_{2|1}( u\vert X_T) \|>\epsilon \right)\cap (X_T\in R\,\mathbb{B}^d) \cap (X_{T-1}\in R\,\mathbb{B}^d) \right) }_{A_T}\\
        &\qquad + 2\, \rm P_{1}(\R^d\setminus R\mathbb{B}^d ).
    \end{align*}
 As the second term can be made arbitrary small by increasing $R$, the result follows by showing that the first term tends to zero. Let  
 $$ \alpha_T(X_1, \dots, X_T) \coloneqq \mathbb{I}\left[\sup_{u\in \mathcal{K}}\| \widehat{\mathbf{Q}}_T ( u\vert X_T)- {\mathbf{Q}}_{2|1}( u\vert X_T) \|>\epsilon \right]. $$
By \cref{assumtpion-density-bounded}, ${\rm P}_{2|1}$  is bounded in $R\,\mathbb{B}^d \times R\,\mathbb{B}^d $  by a finite constant $\Lambda_{R}$, so that 
 \begin{align*}
     A_T&\coloneqq \int \cdots \int   \int_{ R\,\mathbb{B}^d } \int_{ R\,\mathbb{B}^d }  \alpha_T(x_1, \dots, x_T) {\rm P}_{2|1}(x_T|x_{T-1}) {\rm d}x_{T} {\rm P}_{2|1}(x_{T-1}|x_{T-2}) {\rm d}x_{T-1} \cdots  p_1(x_1) {\rm d}x_1 \\
     &\leq \Lambda_R \int \cdots \int   \int_{ R\,\mathbb{B}^d } \int_{ R\,\mathbb{B}^d }  \alpha_T(x_1, \dots, x_T) {\rm d}x_{T} {\rm P}_{2|1}(x_{T-1}|x_{T-2}) {\rm d}x_{T-1} \cdots  p_1(x_1) {\rm d}x_1 \\
     &= \Lambda_R   \int_{ R\,\mathbb{B}^d }  \int \cdots \int  \int_{ R\,\mathbb{B}^d }  \alpha_T(x_1, \dots, x_T)  {\rm P}_{2|1}(x_{T-1}|x_{T-2}) {\rm d}x_{T-1} \cdots  p_1(x_1) {\rm d}x_1 {\rm d}x_{T} 
 \end{align*}
 as $T\to \infty.$
From \cref{theorem-consistency}, for every $x\in {\rm supp}(P_1)$, it follows that 
$$ \int \cdots \int  \int_{ R\,\mathbb{B}^d }  \alpha_T(x_1, \dots, x_{T-1}, x)  {\rm P}_{2|1}(x_{T-1}|x_{T-2}) {\rm d}x_{T-1} \cdots  p_1(x_1) {\rm d}x_1  \to 0 .$$
The dominated convergence theorem concludes the proof.\hfill$\square$
\subsection{Convergence rates}
\subsubsection{Proof of \cref{lemma:wassersteinEstability}}
{Recall that $\mu_d^{(k)}$ is defined in \eqref{measure_cond_Reg} for $k=k(T)$. }
To simplify the formulas, write  $\rm P_x$ for~${\rm P}_{2|1}(\cdot|x)$. 
   By the   definition of push-forward measures,  and using the fact that $ \hat{\pi}$, defined in \eqref{kanto_reg}, is a coupling, 
    $$ \int f {\rm d} (\widehat{\rm P}_x-\rm P_x) = \int f  {\rm d}\hat{\pi} -\int f\circ \mathbf{Q}_{2|1}(\cdot|x) {\rm d} \mu_d,$$
 for any continuous and bounded function $f$. Hence, setting $f\coloneqq \psi(\cdot| x)$ (recall that $x$ is fixed) where $\nabla_z \psi(z|x) = \mathbf{F}_{2|1}(z|x)$, we obtain  
 \begin{multline*}
     \int \psi(\cdot |x) {\rm d}(\widehat{\rm P}_x-\rm P_x) + \int \psi(\cdot |x)\circ \mathbf{Q}_{2|1}(\cdot|x) {\rm d}(\mu_d-\mu_d^{(k)})\\= \int \psi(v |x) {\rm d} {\hat \pi} (u,v)-\int \psi( \mathbf{Q}_{2|1}(\cdot|x)|x) {\rm d}{\mu}_d^{(k)}
 \end{multline*}
 
Under \cref{Assumption-On-density} the function $\psi(\cdot|x)$ is $\mathcal{C}^1$ in $\R^d$ and  $\mathcal{C}^2$ in ${\rm int}({\rm supp}(\rm P))$ except on the convex set ${\bf Q}_{1|2}(0|x)=\argmin \psi(\cdot|x)$, 
which has measure zero. The convex conjugate of~$\psi(\cdot|x)$ is the function $\varphi(\cdot|x)$, which is $\mathcal{C}^2$ in $\mathbb{B}^d\setminus \{0\}$. 
Below, we write $\varphi$ and  $\psi$ instead of~$\varphi(\cdot|x)$  and $\psi(\cdot|x)$. Since $\psi$ is convex, applying Jensen's inequality in 
$$ \int \psi {\rm d}(\widehat{\rm P}_x-\rm P_x) + \int \psi\circ \mathbf{Q}_{2|1}(\cdot|x) {\rm d}(\mu_d-{\mu_d^{(k)})}= \int \psi(v) {\rm d}{\hat \pi}(u,v)-\int \psi\circ \mathbf{Q}_{2|1}(\cdot|x) {\rm d}{\mu_d^{(k)})}$$
 yields 
$$ \int \psi {\rm d}(\widehat{\rm P}_x- \rm P_x) + \int \psi\circ \mathbf{Q}_{2|1}(\cdot|x) {\rm d}(\mu_d-{\mu_d^{(k)})})\geq  \int \psi \circ \widehat{\mathbf{Q}}_{T}(\cdot|x) {\rm d}{\mu_d^{(k)})}-\int \psi\circ \mathbf{Q}_{2|1}(\cdot|x) {\rm d}{\mu_d^{(k)})}. $$
The function $\psi$ is strongly convex on the compact convex set $\mathcal{K}' $, so that, for some $\lambda>0$, 
$$ \psi(z) \geq \psi (y)+\langle \nabla \psi(y)  ,z-y \rangle + \mathbb{I}_{x,y\in \mathcal{K}'}\lambda \|z-y\|^2  ,$$
 from which  we get the estimate 
 \begin{align}\label{secondterm}
 &\int \psi {\rm d}(\widehat{\rm P}_x-\rm P_x) + \int \psi\circ \mathbf{Q}_{2|1}(\cdot|x) {\rm d}(\mu_d-{\mu_d^{(k)})}\nonumber\\
       &\geq \int \langle\nabla \psi(\mathbf{Q}_{2|1}(u|x)), \widehat{\mathbf{Q}}_{T}(u|x)-\mathbf{Q}_{2|1}(u|x)\rangle  {\rm d}{\mu_d^{(k)})}(u)+  \lambda\int_{\mathcal{V}_n} \|\mathbf{Q}_{2|1}(u|x) -\widehat{\mathbf{Q}}_{T}(u|x)\|^2  {\rm d}{\mu_d^{(k)})}(u)\nonumber\\
       &= \int \langle  u, \widehat{\mathbf{Q}}_{T}(u|x)-\mathbf{Q}_{2|1}(u|x)\rangle   {\rm d}{\mu_d^{(k)})}(u)+  \lambda \int_{\mathcal{V}_n} \|\mathbf{Q}_{2|1}(u|x) -\widehat{\mathbf{Q}}_{T}(u|x)\|^2   {\mu_d^{(k)})}.
    \end{align}
    
On the one hand, the Fenchel equality implies
$$ \int \psi\circ \mathbf{Q}_{2|1}(\cdot|x) {\rm d}(\mu_d-{\mu_d^{(k)})})=- \int \varphi(\cdot|x) {\rm d}(\mu_d-{\mu_d^{(k)})})+ \int \langle  \mathbf{Q}_{2|1}(u|x) , u\rangle  {\rm d}(\mu_d-{\mu_d^{(k)})})(u). $$
On the other hand, recalling the definition of $\widehat{\mathbf{Q}}_{T}(u|x) = \int v {\rm d}{\hat \pi}(v|u)$, 
$$  \int \langle  u, \widehat{\mathbf{Q}}_{T}(u|x)\rangle   {\rm d}{\mu_d^{(k)})}(u)=\int \langle  u, v\rangle   {\rm d}{\hat \pi}(u,v) .   $$
Finally, Kantorovich duality yields
$$  \int \langle  u, {\mathbf{Q}}_{2|1}(u|x)\rangle   {\rm d}\mu_{d}   (u)=\inf_{f} \int f {\rm d} \mu_d +  \int f^* {\rm d P}_x {\leq \int \widehat{\varphi}_T(\cdot|x) {\rm d} \mu_d + \int \widehat{\psi}_T(\cdot|x) {\rm d}{\rm P}_x   ,  } $$
{where 
$$(\widehat{\psi}_T(\cdot|x), \widehat{\varphi}_T(\cdot|x))\in   \argmin_{f(u)+ g(v)\geq \langle u,v\rangle } \int f {\rm d} \mu_d^{(k)} +  \int g  {\rm d}\widehat{\rm P}_x . $$}
%
This entails a bound on the second term of the right-hand side of \eqref{secondterm}:
\begin{align*}
     \lambda \int_{A} \|\mathbf{Q}_{2|1}(u|x)& -\widehat{\mathbf{Q}}_{T}(u|x)\|^2   {\rm d}{\mu_d^{(k)})}(u)\\
    &\leq \int \psi {\rm d}(\widehat{\rm P}_x-\rm P_x) + \int \langle u, \mathbf{Q}_{2|1}(u|x) \rangle  {\rm d} \mu_d(u)- \int \langle  u, \widehat{\mathbf{Q}}_{T}(u|x)\rangle  {\rm d}{\mu_d^{(k)})}(u)
    \\
    &\leq \int \psi {\rm d}(\widehat{\rm P}_x-\rm P_x) - \int \widehat{\psi}_{T} {\rm d}(\widehat{\rm P}_x-\rm P_x)-\int \widehat{\varphi}_T(\cdot|x)-\varphi(\cdot|x) {\rm d}({\mu_d^{(k)})}-\mu_{d}) \\
    &= \int( \psi-\widehat{\psi}_T(\cdot|x)) {\rm d}(\widehat{\rm P}_x-\rm P_x) + \int (\varphi(\cdot|x)-\widehat{\varphi}_T(\cdot|x)) {\rm d}(\mu_{d}-{\mu_d^{(k)})}).\label{secondterm}
\end{align*}
Since  ${\rm P}_{x}$ and $\mu_d$ are compactly supported, the convex functions  $\psi$, $\varphi(\cdot|x))$,  $\hat{\psi}$, and $\hat{\varphi}(\cdot|x)$ are Lipschitz. The result follows.\hfill$\square$ 

\subsubsection{Proof of \cref{Theorem:rates}}
Let  
$K_h(v)\coloneqq  {K\left({v}/{h} \right)}/{ h^d}$.
Due to \cref{lemma:wassersteinEstability}, we just need to bound $  d_{\rm BLC}(\widehat{\rm P}_x, {\rm P}_{2|1}(\cdot|x )). $
Splitting this bound  into tree different components yields 
\begin{align*} d_{\rm BLC}( \widehat{\rm P}_x, {\rm P}_{2|1}(\cdot|x ))&\leq  d_{\rm BLC}( \widehat{\rm P}_x, \widehat{\rm P}_x^{(2)})+ d_{\rm BLC}( \widehat{\rm P}_x^{(2)}, {\rm P}^h) +d_{\rm BLC}( {\rm P}^h,{\rm P}_{2|1}(\cdot|x ))\\ 
&\eqqcolon {B_1 + B_2 + B_3}, \text{ \rm say,}
\end{align*}
where 
$$ \int f {\rm d} \widehat{\rm P}_x^{(2)}= \frac{1}{np_1(x)}\sum_{t=1}^{T-1} K_h(x-X_{t+1}) f(X_t)  $$
and 
$$ \int f {\rm d} {\rm P}^h= \frac{1}{p_1(x)}\int K_h(x-v_{1}) f(v_2) {\rm dP}(v_1,v_2). $$
Let us bound each of these three components separately. 
\begin{enumerate}
\item[($B_1$)]
Observe that, for any bounded   Lipschitz function 
$f$, 
\begin{align*}
    \left|\int f {\rm d}(\widehat{\rm P}_x- \widehat{\rm P}_x^{(2)})\right| &\leq  \frac{\|f\|_\infty}{T-1}\sum_{t=1}^{T-1} K_h(x-X_{t+1})  \left| \frac{1}{\frac{1}{T-1}\sum_{t=1}^{T-1} K_h(x-X_{t+1})  }-\frac{1}{p_1(x)}\right|\\
    &=\frac{\|f\|_\infty}{p_1(x)}  \left| {p_1(x)}-{\frac{1}{T-1}\sum_{t=1}^{T-1} K_h(x-X_{t})}\right|.
\end{align*}
Decomposing into  bias and variance yields 
\begin{multline*}
    \Big| {p_1(x)}-{\frac{1}{T-1}\sum_{t=1}^{T-1} K_h(x-X_{t})}\Big| \\\leq  \underbrace{\left|{\frac{1}{T-1}\sum_{t=1}^{T-1} K_h(x-X_{t+1})}- \int K_h(x-v_1) p(v_1,v_2) {\rm d}v_1{\rm d}v_2 \right|}_{V_{{T}}}\\
    +\underbrace{\left|\int K_h(x-v_1) p(v_1,v_2) {\rm d}v_1{\rm d}v_2 -p_1(x)\right|}_{B_{{1T}}}.
\end{multline*}
For the bias term, the assumption that $p(v_1,v_2)$ is $\mathcal{C}^{1,1}$  and has a  Lipschitz derivative   with constant $L$ implies that   
\begin{multline*}
     B_{{1T}}
     \leq  \underbrace{\left|\int K_h(x-v_1)  \left\{p(x,v_2)+ \langle \nabla p(x,v_2), x-v_1\rangle \right\} {\rm d}v_1  {\rm d}v_2 -p_1(x)\right|}_{B_{{1T}}' } \\
+ \int_{\mathcal{X}}\int K_h(x-v_1) \|v_1-x\|^2 {\rm d}v_1{\rm d}v_2.
\end{multline*}
Since 
$ \int_{\mathcal{X}}\int K_h(x-v_1) \|v_1-x\|^2 {\rm d}v_1{\rm d}v_2 \leq {\rm diam}(\mathcal{X}) h^2 $
and 
\begin{align*}
    B_{{1T}}' 
    &= \Bigg|\int_{\mathcal{X}} \int K_h(x-v_1) {\rm d}v_1 p_{1,2}(x,v_2) {\rm d}v_2 \\
    &\qquad \qquad+\int_{\mathcal{X}} \int K_h(x-v_1) \langle \nabla p_{1,2}(x,v_2), x-v_1\rangle  {\rm d}v_1 {\rm d}v_2 -p_1(x)\Bigg| \\
    &=  \left|\int_{\mathcal{X}}\int K_h(x-v_1) \langle \nabla p(x,v_2), x-v_1\rangle  {\rm d}v_1 {\rm d}v_2\right| \\
    &= \left|\int K_h(z) \langle \nabla p(x,v_2), z\rangle   {\rm d}z {\rm d}v_2\right| 
    =  \left|\int_{\mathcal{X}} \left\langle \nabla p(x,v_2), \int    K_h(z) z  {\rm d}z \right\rangle {\rm d}v_2 \right| =0,
\end{align*}
(where the last equality follows from the assumption that $\int    K(z) z  {\rm d}z =0$). 

Turning to the variance term and  arguing as in the proof of \cref{lemma:consistency} yields 
$$ \E[V_{{T}^2}] \leq \frac{1}{(T-1)^2} \sum_{s,t=1}^{T-1} 
{\rm Cov}\left(K_h\left(\frac{x-X_t}{h}\right), K_h\left(\frac{x-X_s}{h}\right) \right) \lesssim \frac{1}{T h^d}.$$
We thus have
\begin{equation}
    \label{First-Term-bound-rates}
   \E[ B_1] =  \E[ d_{\rm BLC}( \widehat{\rm P}_x, \widehat{\rm P}_x^{(2)})]  \lesssim  h^2 + \frac{1}{T h^d}. 
\end{equation}

\item[($B_2$)]
 Chaining arguments are standard in this context. By \cite[Theorem~5]{Covering-convex}, the uniform-norm covering numbers $\mathcal{N}(\epsilon, {\rm BLC}(\mathcal{X}))$ of the class ${\rm BLC}(\mathcal{X})$ of bounded convex  Lipschitz functions over the compact set $\mathcal{X}$ are upper-bounded\linebreak  by
 $ \log(\mathcal{N}(\epsilon)) \lesssim \epsilon^{-{d}/{2}}.$ 
That is, for each $\epsilon>0$, there exists a finite sequence $f_1, \dots,f_{\mathcal{N}(\epsilon)} $ of bounded  convex Lipschitz functions such that 
 $ \inf_{s=1, \dots, {\mathcal{N}(\epsilon)} }\| f-f_s\|\leq \epsilon $ 
for \linebreak any~$f\in~\!{\rm BLC}(\mathcal{X})$. 
The same bound holds for the uniform-norm covering numbers~$\mathcal{N}(\epsilon, \mathcal{F}_\delta)$ of the class  
$ \mathcal{F}_\delta\coloneqq \{ f-g: \ f,g\in {\rm BLC}(\mathcal{X}) \ \|f-g\|_\infty\leq \delta \}$, $\delta>0 .$ 

 We establish  a bound on $B_2=d_{\rm BLC}( \widehat{\rm P}_x^{(2)}, {\rm P}^h)$ for $(T-1)/2\in \N$;  the general case follows along similar lines.  Fix $f\in {\rm BLC}(\mathcal{X})$ and note that, using the convexity of the exponential function, 
\begin{align*}
    &\E\left[\exp\left( \frac{\lambda}{T-1}\sum_{s=1}^{T-1}(f(X_s,X_{s+1})- \E[f(X_s,X_{s+1})])\right)\right]\\
    & =\E\Bigg[\exp\Bigg( \frac{\lambda}{(T-1)}\sum_{s=1}^{(T-1)/2}(f(X_{2s},X_{2s+1})- \E[f(X_{2s},X_{2s+1})])\\
    &\qquad \qquad \qquad+ \frac{\lambda}{(T-1)}\sum_{s=1}^{(T-1)/2}(f(X_{2s-1},X_{2s})-\E[f(X_{2s-1},X_{2s})])\Bigg)\Bigg] \\
    &\leq \frac{1}{2}\E\Bigg[\exp\Bigg( \frac{\lambda}{(T-1)/2}\sum_{s=1}^{(T-1)/2}(f(X_{2s},X_{2s+1})- \E[f(X_{2s},X_{2s+1})])\Bigg)\Bigg]\\
    &\qquad  + \frac{1}{2}\E\Bigg[\exp\Bigg(\frac{\lambda}{(T-1)/2}\sum_{s=1}^{(T-1)/2}(f(X_{2s-1},X_{2s})- \E[f(X_{2s-1},X_{2s})])\Bigg)\Bigg].
\end{align*}
Hoeffding's lemma for Markov sequences (see \cite[Theorem~1]{JianqingFan-etal.JMLR.2021}) and \cref{Assumption:MArkov-kenel} yield 
\begin{multline*}
    \E\Bigg[\exp\Bigg( \frac{\lambda}{(T-1)/2}\sum_{s=1}^{(T-1)/2}(f(X_{2s},X_{2s+1})- \E[f(X_{2s},X_{2s+1})])\Bigg)\Bigg]\\
    \leq \exp\left(\frac{2(1+\delta)\lambda^2\|f\|_\infty^2}{(1-\delta)T}   \right)
\end{multline*}
and 
\begin{multline*}
     \E\Bigg[\exp\Bigg(\frac{\lambda}{(T-1)/2}\sum_{s=1}^{(T-1)/2}(f(X_{2s-1},X_{2s})- \E[f(X_{2s-1},X_{2s}))]\Bigg)\Bigg]\\
     \leq \exp\left(\frac{2(1+\delta)\lambda^2\|f\|_\infty^2}{(1-\delta)T}   \right),
\end{multline*}
so that 
$$ \E\left[\exp\left( \frac{\lambda}{T-1}\sum_{s=1}^{T-1}(f(X_s,X_{s+1})- \E[f(X_s,X_{s+1})])\right)\right] \leq \exp\left(\frac{2(1+\delta)\lambda^2\|f\|_\infty^2}{(1-\delta)T}   \right). $$
As a consequence, for every $f$ with $\|f\|_\infty<\infty$,   $$ \E\left[\exp\left( \lambda \left(  \int f {\rm d} (\widehat{\rm P}_x^{(2)}- {\rm P}^h)\right)\right)\right] \leq \exp\left(\frac{2(1+\delta)\lambda^2 \|f\|_\infty^2}{(1-\delta) (p(x))^2T h^d}   \right). $$

 The random process $f\mapsto U_{T}(f)\coloneqq \int f {\rm d} (\widehat{\rm P}_x^{(2)}- {\rm P}^h)$ thus  is $\sigma_{T}^2$-sub-Gaussian with respect to the $\|\cdot\|_\infty$-norm, with  
$ \sigma_{T} \lesssim  {1}/({ T^{{1}/{2}} h^{{d}/{2}}}) $. Therefore,
~${V}_{T}(f)\coloneqq T^{{1}/{2}} h^{ {d}/{2}} U_n(f)$ is $\sigma^2$-sub-Gaussian with respect to the $\|\cdot\|_\infty$-norm with  
$ \sigma <\infty$ irrespective of $T$. First  assume  that $d>4$. Dudley’s entropy bound {\cite[Theorem~5.22]{Wainwright_2019}} implies that, for every~$\gamma\in (0,1)$, 
\begin{align*}
    \E\left[ \sup_{f\in {\rm BLC}(\mathcal{X})} U_{T}(f) \right]& = \frac{\E\left[ \sup_{f\in {\rm BLC}(\mathcal{X})} V_{T}(f) \right]}{T^{{1}/{2}} h^{{d}/{2}}}\\
    &\lesssim \frac{\E\left[ \sup_{f\in \mathcal{F}_\gamma} V_{T}(f) \right] + \int_{\gamma}^{1} \epsilon^{-{d}/{4}} d\epsilon }{T^{{1}/{2}} h^{{d}/{2}}}\\
    &\lesssim \E\left[ \sup_{f\in \mathcal{F}_\gamma} U_{T}(f) \right] +\frac{ \gamma^{1-{d}/{4}}-1 }{T^{{1}/{2}} h^{{d}/{2}}}
    \lesssim \gamma +\frac{ \gamma^{1- {d}/{4}}-1 }{T^{{1}/{2}} h^{{d}/{2}}}.
\end{align*}
For $d>4$ and $\gamma={T^{-\frac{2}{d}} h^{-2}}$, we obtain 
\begin{equation}
    \label{Second-Term-bound-rates}
    \E\left[ d_{\rm BLC}( \widehat{\rm P}_x^{(2)}, {\rm P}^h) \right]=\E\left[ \sup_{f\in {\rm BLC}(\mathcal{X})} U_{T}(f) \right] \lesssim \frac{1}{T^{{2}/{d}} h^2}.
    \end{equation}
For $d=4$ and $\gamma\in (0,1)$, repeating the same argument yields 
\begin{align*}
    \E\left[ \sup_{f\in {\rm BLC}(\mathcal{X})} U_{T}(f) \right]
    &\lesssim \gamma -\frac{ \log(\gamma) }{T^{{1}/{2}} h^{2}}
\end{align*}
hence, for  $\gamma={T^{- {1}/{2}} h^{-2}}$, 
\begin{align*}
    \E\left[ d_{\rm BLC}( \widehat{\rm P}_x^{(2)}, {\rm P}^h)\right]
    &\lesssim \frac{ \log({T^{{1}/{2}} h^{2}}) }{T^{{1}/{2}} h^{2}}.
\end{align*}
Finally, for $d<4$, the entropy integral converges and we get the rate 
$$ \E\left[ d_{\rm BLC}( \widehat{\rm P}_x^{(2)}, {\rm P}^h)\right]
    \lesssim \frac{ 1 }{T^{{1}/{2}} h^{{d}/{2}}}. $$

\item[($B_3$)] 
By the same argument as for 
 $B_1$ in~\eqref{First-Term-bound-rates},  
\begin{equation}\label{third} {\rm BL}(  {\rm P}^h,{\rm P}_{2|1}(\cdot|x )) \lesssim h^2.
\end{equation} 
As a consequence of \eqref{First-Term-bound-rates}, \eqref{Second-Term-bound-rates}, and \eqref{third}, we obtain 
$$ \E[d_{\rm BLC}( \widehat{\rm P}_x, {\rm P}_{2|1}(\cdot|x ))]\lesssim \begin{cases}
\frac{ 1 }{T^{{1}/{2}} h^{{d}/{2}}}+h^2&{\rm if} \ d<4,\vspace{2mm}\\
\frac{ \log({T h^{4}}) }{T^{{1}/{2}} h^{2}}+h^2  &{\rm if} \ d=4, \vspace{2mm}\\

    \frac{1}{T^{{2}/{d}} h^2}+h^2  &{\rm if} \ d>4,
\end{cases}$$
which concludes the proof of \eqref{threebounds} and (i).

To prove (ii),  fix $\epsilon>0$ and a compact subset  $\mathcal{K}$ of $\mathbb{B}^d\setminus \{0\}$. Since~${\mathbf{Q}}_{2|1}(\cdot|x)$ is a homeomorphism between  $\mathbb{B}^d\setminus \{0\}$ and ${\rm int}({\rm supp}({\rm P}_{2|1}(\cdot|x)) )\setminus  \{{\mathbf{Q}}_{2|1}(0|x)\}$ (see \cite{Barrio2023RegularityOC}),   for each $v\in \mathcal{K}$ we can find a ball $v+\alpha\mathbb{B}^d$ with center $v$ and radius~$\alpha>0$ such that 
$$ \mathcal{K}_1^\beta = \overline{\rm coh}\left({\mathbf{Q}}_{2|1}\left(v+\alpha\mathbb{B}^d\bigg\vert x\right) \right)\subset {\rm int}({\rm supp}({\rm P}_{2|1}(\cdot|x)) )\setminus  \{{\mathbf{Q}}_{2|1}(0|x)\}, $$
where $ \overline{\rm coh}(A)$ denotes the closed convex hull of a set $A$. By a compactness argument,~$\mathcal{K}$ can be covered by a finite numbers of such balls; hence, it is enough to establish the result for  one of them. Let $\beta$ be small enough for the  set $\mathcal{K}_1^\beta\coloneqq \{ u: \inf_{z\in \mathcal{K}_1}\|u-z\| \leq~\!\beta\}$, which is compact and convex,  
to be  contained in $ {\rm int}({\rm supp}({\rm P}_{2|1}(\cdot|x)) )\setminus  \{{\mathbf{Q}}_{2|1}(0|x)\}$. Then, letting 
$ \gamma_{T}\coloneqq {T^{-\frac{1}{d}}}+d_{\rm BLC}( {\mu_d^{(k)}},\mu_d)$, 
we get, for every $M>0$,  
\begin{align*}
     \PP&\left( \left|\int_{v+\alpha\mathbb{B}^d}\| \widehat{\mathbf{Q}}_{T} ( u\vert x)- {\mathbf{Q}}_{2|1}( u\vert x) \|^2 {\rm d}{\mu_d^{(k)})}(u)\right|>M \gamma_{T} \right)\\
     &\leq \PP\left(\left( \left|\int_{v+\alpha\mathbb{B}^d}\| \widehat{\mathbf{Q}}_{T} ( u\vert x)- {\mathbf{Q}}_{2|1}( u\vert x) \|^2 
     {\rm d}{{\mu_d^{(k)})}}(u)
     \right|>M \gamma_{T} \right) \cap \mathcal{W}_{T} \right) + \PP\left( \mathcal{W}_{T}^c \right)
\end{align*}
where 
$ \mathcal{W}_{T}$ is the event $\widehat{\mathbf{Q}}_{T}\Big(v+\alpha\mathbb{B}^d\big\vert x\Big) \subset \mathcal{K}_1^\beta.$
By \cref{theorem-consistency}, $\PP\left( \mathcal{W}_{T}^c \right)\to 0$, so that~(ii) follows from (i).
\hfill$\square$
\end{enumerate}


\section{Asymptotic stationarity of the  simulated series in Section~4}


\subsection{Case 1}
Let $X_t$ be as in \eqref{case1}.  
 Let 
\begin{equation*}
 G(x, z) \coloneqq \begin{bmatrix}
     \frac{x_1+x_2}{3}\\
     \sqrt{ \frac{\|\bx\|^2+5}{4} }
 \end{bmatrix} + \sin\left( \frac{\pi}{10} \|\bx\|\right) \cdot z,
\end{equation*}
where $(x_1, x_2)$ denotes the coordinates of $x\in \R^2$.   
Fixing ${\varepsilon}\sim \mathcal{N}({0}, {\rm I})$,  decompose
\begin{multline}\label{eq:proof-stationary-example-1-1}
    \EE [\| G(\bx, {\varepsilon})-G(\by, {\varepsilon}) \|^2] 
    = \underbrace{\left\| \bigg( \frac{x_1+x_2}{3} - \frac{y_1+y_2}{3}, \sqrt{ \frac{\|\bx\|^2+5}{4}} - \sqrt{\frac{\|\by\|^2+5}{4}} \bigg)  \right\|^2}_{\eqqcolon M_1} \\   + \underbrace{\left( \sin\left( \frac{\pi}{10} \|\bx\|\right)-\sin\left( \frac{\pi}{10} \|\by\|\right) \right)^2 \EE[\|{\varepsilon}\|^2]}_{\eqqcolon M_2} .
\end{multline}
By the Cauchy–Schwarz and  triangle inequalities, we get 
\begin{align}\label{eq:proof-stationary-example-1-2}
    M_1:=  \frac{1}{9}\left( \langle x-y, (1,1)^\top \rangle  \right)^2 +\left( \frac{ \frac{\|\bx\|^2+5}{4} - \frac{\|\by\|^2+5}{4}}{\sqrt{ \frac{\|\bx\|^2+5}{4}} + \sqrt{\frac{\|\by\|^2+5}{4}} }\right)^2 \leq \frac{17}{36} \|\bx-\by\|^2
\end{align}
while, since  $u\mapsto \sin(u)$ is $1$-Lipschitz,  the triangle inequality yields 
\begin{align}\label{eq:proof-stationary-example-1-3}
M_2&=   2\left(\sin\left( \frac{\pi}{10} \|\bx\|\right)-\sin\left( \frac{\pi}{10} \|\by\|\right) \right)^2 
    \leq \frac{\pi^2}{50} \|\bx-\by\|^2 .
\end{align} 
Combining \eqref{eq:proof-stationary-example-1-1}, \eqref{eq:proof-stationary-example-1-2}, and \eqref{eq:proof-stationary-example-1-3}, we obtain 
$$  \EE [\| G(\bx, {\varepsilon})-G(\by, {\varepsilon}) \|^2] \leq \left( \frac{17}{36}+ \frac{\pi^2}{50} \right),  $$
so that, as  $\frac{17}{36}+ \frac{\pi^2}{50}<1$,  the series $\{X_t\}_t$ is asymptotically stationary (see \cite[Theorem~1.1]{Diaconis.Freedman.SIAM.2000}).  \hfill$\square$
\color{black}

\subsection{Case 2}
Let $X_t$, $\varepsilon_t$ and $f$ be as in \eqref{case2}. 
The asymptotic stationarity of the process generated by \eqref{case2} follows from a contraction argument on the Borel map 
\begin{equation*}x\mapsto 
 G(x, z) \coloneqq \begin{bmatrix}
        \tanh\Big(\frac{1}{2}\big(x_1 + x_2\big)\Big) - \frac{1}{2}\\         
        \cos\Big(\frac{\pi}{10} f\big(x_1 + x_2\big) \Big) 
        \end{bmatrix} +\frac{1}{2}\|x\| \cdot z .
\end{equation*}
It is easy to see that 
$$ x\mapsto   \begin{bmatrix}
        \tanh\Big(\frac{1}{2}\big(x_1 + x_2\big)\Big) - \frac{1}{2}\\         
        \cos\Big(\frac{\pi}{10} f\big(x_1 + x_2\big) \Big) 
        \end{bmatrix} $$
is Lipschitz with constant $\frac{25+\pi^2}{50} $. Hence,  
 \begin{align*}
    \EE \left[\big\| G(\bx, {\varepsilon})-G(\by, {\varepsilon}) \big\|^2 \right] \leq \left(\frac{25+\pi^2}{50} +\frac{1}{6}\right) \|x-y\|^2
\end{align*}
and the claim follows again by \cite[Theorem~1.1]{Diaconis.Freedman.SIAM.2000}. \hfill$\square$
\color{black}
\subsection{Case 3} 
Let $X_t$ be as in \eqref{case3} and fix
$$ \eps\sim \frac{1}{4}N\big( 0 , \frac{1}{25}{\rm I}\big) + \frac{1}{4}N\big((0.866, -0.5), \frac{1}{25}{\rm I}\big) \\ 
    +\frac{1}{4}N\big((-0.866, -0.5), \frac{1}{25}{\rm I}\big) + \frac{1}{4}N\big((0,1), \frac{1}{25}{\rm I}\big). $$
Let us show that $\{X_t\}_{t}$ is asymptotically stationary when  ${\rm R}(t) $ is a fixed rotation matrix~$R$. Without loss of generality,  assume that ${\rm R}(t) ={\rm I}$. As above, we are  using a contraction argument on 
$$
G(x,z)=\begin{bmatrix}
         \dfrac{\log(\|x\|+2)}{\|x\|+2} \vspace{1mm}
         \\
         \dfrac{\|x\|}{\|x\|+ \sqrt{2}}
        \end{bmatrix} +
        \left(\sqrt{\|X_t\|+1} \right) \, z.
$$ 
We have 
\begin{multline*}
   \EE \| G(\bx, {\varepsilon})-G(\by, {\varepsilon}) \|^2 = \underbrace{\left( \frac{\log(\|\bx\|+2)}{\|\bx\|+2} - \frac{\log(\|\by\|+2)}{\|\by\|+2} \right)^2}_{=:M_1''} 
    + \underbrace{\left( \frac{\|\bx\|}{\|\bx\|+ \sqrt{2}} - \frac{\|\by\|}{\|\by\|+ \sqrt{2}}  \right)^2}_{=:M_2''} \\
    + \underbrace{\left(\sqrt{\|\bx\|+1} - \sqrt{\|\by\|+1} \right)^2 \EE\|{\varepsilon}\|^2}_{=:M_3''}.  
\end{multline*}
To bound $M_1^{\prime\prime}$, let $z_1\coloneqq \|\bx\|+2$ and $z_2\coloneqq \|\by\|+2$, so that $z_1, z_2 \geq 2$. We have
\begin{align}
    \bigg| \frac{\log(\|\bx\|+2)}{\|\bx\|+2} - \frac{\log(\|\by\|+2)}{\|\by\|+2} \bigg| &= \bigg| \frac{\log(z_1)}{z_1} - \frac{\log(z_2)}{z_2} \bigg| \nonumber \\
    &
    = \bigg|\frac{z_2\log z_1 - z_2\log z_2 + z_2\log z_2 - z_1\log z_2}{z_1 z_2} \bigg| \nonumber \\
    &= \bigg|\frac{1}{z_1}(\log z_1 -\log z_2) + \frac{z_2-z_1}{z_1}\frac{\log z_2}{z_2} \bigg| \nonumber \\
    &\leq \frac{1}{4}|z_1-z_2| + \frac{1}{2e}|z_1-z_2| \label{star}\\
    &
    < \frac{1}{2}|z_1-z_2| = \frac{1}{2}\|\bx-\by\| \nonumber 
\end{align}
where inequality \eqref{star} follows from two facts:
\begin{compactenum}
\item[(a)]
  $  \log z_1 -\log z_2 = \frac{1}{c}(z_1-z_2) \leq \frac{1}{2}(z_1-z_2)$ 
for some $z_2<c<z_1$ (assuming, without loss of generality, $z_1 > z_2$)  by the mean-value Theorem and because~$z_1, z_2 > e$; 
\item[(b)] 
  $  0<  {\log z_2}/{z_2} \leq \frac{1}{e} $
because the function $h(t) = \log t/t$ is increasing on $(0, e)$ and decreasing on $(e, +\infty)$. 
\end{compactenum}\medskip 

To bound $M_2^{\prime\prime}$,  let $h(t) \coloneqq {t}/{(t+\sqrt{2})}$: then,   with $z_1\coloneqq \|\bx\|$ and $z_2\coloneqq \|\by\|$,
\begin{equation*}
 \Big(M_2^{\prime\prime}\Big)^{1/2}\! =   | h(z_1) -h(z_2) | \leq \sup_{t>0} |h'(t)| |z_1-z_2| \leq \frac{1}{\sqrt{2}}|z_1-z_2| \leq \frac{1}{\sqrt{2}}\|\bx-\by\|.
\end{equation*}

As for $M_3^{\prime\prime}$, 
\begin{align*}
 \Big(M_3^{\prime\prime}/ \EE\|{\varepsilon}\|^2\Big)^{1/2}\! =   \Big| \sqrt{\|\bx\|+1} - \sqrt{\|\by\|+1} \Big| &\leq \frac{\|\bx\|-\|\by\|}{\sqrt{\|\bx\|+1} + \sqrt{\|\by\|+1}} \leq \frac{1}{2}\|\bx-\by\| .
\end{align*}
Combining these  bounds yields 
\begin{align*}
    \EE \| G(\bx, {\varepsilon})-G(\by, {\varepsilon}) \|^2 & < \frac{1}{4}\|\bx-\by\|^2 + \frac{1}{2}\|\bx-\by\|^2 + \frac{1}{4}\|\bx-\by\|^2 \EE\|{\varepsilon}\|^2 < \|\bx-\by\|^2 
\end{align*}
since, from \eqref{gaussian_mixture}, $\EE\|{\varepsilon}\|^2 \approx 0.83 <1$. Asymptotic stationarity follows.\hfill$\square$


\bibliographystyle{apalike}

\bibliography{ref-2.bib}

\end{document}